\documentclass[reqno]{amsart}
\usepackage[T1]{fontenc}
\usepackage{latexsym,amssymb,amsmath,xcolor}

\usepackage[all,ps]{xy}
\usepackage{amsthm}
\usepackage{longtable}
\usepackage{epsfig}
\usepackage{mathrsfs}
\usepackage{hhline}
\usepackage{epic}
\usepackage{enumerate}
\usepackage{xcolor}
\usepackage{pgf,tikz}
\usepackage{mathrsfs}
\usetikzlibrary{arrows,shapes}

 \newcommand{\fl}{\mathfrak{fl}}

\newcommand{\as}{\operatorname{a}}
\newcommand{\asc}{\operatorname{asc}}

 \newcommand{\C}{\mathbb{C}}

\newcommand{\N}{\mathbb{N}}
 
 \newcommand{\Z}{\mathbb{Z}}
 
 \newcommand{\Sym}{\mathfrak{S}}

\newcommand{\core}{\mathscr{C}}

\newcommand{\Fr}{\mathfrak F}
\newcommand{\Ar}{\mathcal A}
\newcommand{\Le}{\mathcal L}
\newcommand{\Arv}{\mathfrak A}
\newcommand{\Lev}{\mathfrak L}

\newdimen\shadedBaseline\shadedBaseline=-4mm
\newcount\tableauRow\newcount\tableauCol
\newcommand\ShadedTableau[2][\relax]{%
  \begin{tikzpicture}[scale=0.4,draw/.append style={thick,black},baseline=\shadedBaseline]
    \ifx\relax#1\relax%
    \else 
      \foreach\bx in {#1} { \filldraw[blue!20]\bx+(-.5,-.5)rectangle++(.5,.5); }
    \fi
    \tableauRow=0
    \foreach \Row in {#2} {
       \tableauCol=1
       \foreach\k in \Row {
          \draw(\the\tableauCol,\the\tableauRow)+(-.5,-.5)rectangle++(.5,.5);
          \draw(\the\tableauCol,\the\tableauRow)node{\k};
          \global\advance\tableauCol by 1
       }
       \global\advance\tableauRow by -1
    }
  \end{tikzpicture}%
}
\newcommand\diag[3][\relax]{%
  \begin{tikzpicture}[scale=0.4,draw/.append style={thick,black},baseline=\shadedBaseline]
    \ifx\relax#1\relax%
    \else 
      \foreach\bx in {#1} { \filldraw[blue!20,dashed]\bx+(-.5,-.5)rectangle++(.5,.5); }
    \fi
    \tableauRow=0
    \foreach \Row in {#2} {
       \tableauCol=1
       \foreach\k in \Row {
          \draw(\the\tableauCol,\the\tableauRow)+(-.5,-.5)rectangle++(.5,.5);
          \draw(\the\tableauCol,\the\tableauRow)node{\k};
          \global\advance\tableauCol by 1
       }
       \global\advance\tableauRow by -1
    }
    \foreach \x in {#3} {
      \draw[red,dashed](\x+3,-2-\x)+(-.5,-.5)rectangle++(.5,.5);
    }
  \end{tikzpicture}%
}

\newcommand\frob[7][\relax]{%
  \begin{tikzpicture}[scale=0.4,draw/.append style={black},baseline=\shadedBaseline]
    \ifx\relax#1\relax%
    \else 
      \foreach\bx in {#1} { \filldraw[gray!20,dashed]\bx+(-.5,-.5)rectangle++(.5,.5); }
    \fi
    \tableauRow=0
    \foreach \Row in {#2} {
       \tableauCol=1
       \foreach\k in \Row {
          \draw(\the\tableauCol,\the\tableauRow)+(-.5,-.5)rectangle++(.5,.5);
          \draw(\the\tableauCol,\the\tableauRow)node{\k};
          \global\advance\tableauCol by 1
       }
       \global\advance\tableauRow by -1
    }
        \foreach \x in {#3} {
      \draw[gray!80,dashed](\x+3,-2-\x)+(-.5,-.5)rectangle++(.5,.5);
    }
    \foreach\bx in {#4} {
        \draw[thick]\bx+(-.3,0)rectangle++(.3,.0);
    }
    \foreach\bx in {#5} {
        \draw[thick]\bx+(0,.3)rectangle++(0,-.3);
    }
    \foreach\bx in {#6} {
        \draw[thick,densely dotted]\bx+(-.3,0)--++(.3,.0);
    }
    \foreach\bx in {#7} {
        \draw[thick,densely dotted]\bx+(0,.3)--++(0,-.3);
    }

  \end{tikzpicture}%
}

\newtheorem{theorem}{Theorem}[section] 
\newtheorem{lemma}[theorem]{Lemma}     
\newtheorem{corollary}[theorem]{Corollary}
\newtheorem{proposition}[theorem]{Proposition}

\newtheorem{definition}[theorem]{Definition}
\newtheorem{notation}[theorem]{Notation}
\newtheorem{example}[theorem]{Example}
\theoremstyle{definition}
\newtheorem{remark}[theorem]{Remark}

\title[]
{Cores and quotients of partitions 
through the Frobenius symbol}
\author{Olivier Brunat}

\address{Universit\'e Paris-Diderot Paris 7\\ Institut de math\'ematiques de
         Jussieu -- Paris Rive Gauche\\ UFR de math\'e\-matiques\\ Case
7012\\ 75205 Paris Cedex 13\\
         France.}
\email{olivier.brunat@imj-prg.fr}

\author{Rishi Nath}

\address{York College, City University of New York, 
94--20 Guy R. Brewer Blvd. \\
Jamaica, NY 11435\\
USA
}
\email{rnath@york.cuny.edu}

\subjclass[2010]{05A17,\, 11P83,\, 20C30}

\begin{document} 

\begin{abstract} 

The Frobenius symbol was first introduced in 1900 by Frobenius as a way to
encode an integer partition. In 1941, motivated by the modular representation
theory of the symmetric group, Nakayama introduced the idea of a $p$-core
partition, for $p$ prime, using hook removals.  In the following decade,
Robinson, Littlewood, Staal and Farahat codified the $p$-quotient of such
a partition, using variations of the star diagram.  Since the 1970s, the
convention has been to build up the theory of both core and quotients with
the abacus construction, first introduced by G. James.
 
In this paper we return to the earlier point of view. First we show that,
for any positive integer $t$, the $t$-core and $t$-quotient of an integer
partition can be directly obtained from its Frobenius symbol. The argument
also works in the opposite direction: that is, given the Frobenius symbol
of a $t$-core and a $t$-tuple of Frobenius symbols, one can recover the
Frobenius symbol of the corresponding partition. One immediate application
is the calculation of the Durfee number of the associated partition from
the Frobenius symbols of the core and quotient.

In 1991, J. Scopes gathered together the $p$-core partitions into families
to prove that Donovan's conjecture holds for the symmetric groups at the
prime $p$.  We describe, using our methods, the action of the affine Weyl
group $W_p$ of type $A$ on Frobenius symbols, and use this to parametrize
and compute the explicit number of Scopes families. In particular we
enumerate both the infinite and finite Scopes families.

Core partitions have also attracted recent interest in number theory. By
constructing explicit and combinatorial bijections, we revisit some
well-known identities originally obtained using sophisticated methods.

We end with a close study of the relationship between certain hooks in the
quotient and certain hooks in the associated partition.
\end{abstract}
\maketitle

\section{Introduction} 

A \emph{partition} of a nonnegative integer $n$ is a non-increasing
sequence of positive integers, called \emph{parts}, that sum to $n$. The
\emph{size} of a partition $\lambda$ of $n$ is $|\lambda|=n$. We denote by
$\mathcal P$ the set of partitions of integers, and by $\mathcal P_n$ the
one of partitions of $n$.

The \emph{Young diagram} $[\lambda]$ of a partition $\lambda$ is a
collection of boxes arranged in left- and top-aligned rows, so that if
$\lambda=(\lambda_1\geq\lambda_2\geq\ldots\geq\lambda_r)$, then the number
of boxes in the $j$th row (for $1\leq j\leq r$) of $[\lambda]$ is
$\lambda_j$.  Each box of $[\lambda]$ has coordinate $(i,j)$ using matrix
notation.

\begin{example}
\label{ex:example1}
Let $\lambda=(5,5,4,2,1,1)$ be a partition of $n=18$. The Young diagram of
$\lambda$ is 
\medskip
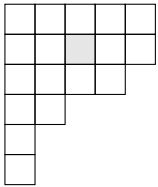
\begin{center} 
\frob[(3,-1)]
    {{\relax,\relax,\relax,\relax,\relax},{\relax,\relax,\relax,\relax,\relax},{\relax,\relax,\relax,\relax},{\relax,\relax},{\relax},{\relax}}
    {}
    {}
    {}
    {}
    {}
\end{center}
\medskip
The colored box has coordinate $(2,3)$.
\bigskip
\end{example}

The \emph{Frobenius symbol}, introduced in 1900 by Frobenius
\cite{frobenius}, is obtained from the Young diagram of $\lambda$ as
follows. For any $1\leq j\leq s$ such that $(j,j)$ are coordinates of a
box of $[\lambda]$, write $a_{s-j}$ for the number of boxes of $[\lambda]$
in the same row and to its right, and $l_{s-j}$ for the ones in the same
column and below. The integers $a_j$ and $l_j$ (for $1\leq j\leq s$) are
respectively called \emph{the arms} and \emph{the legs} of $\lambda$. The
integer $s$ is called \emph{the Durfee number} of $\lambda$, and also
counts the dimensions of the largest square that can fit inside the Young
diagram. The Frobenius symbol of $\lambda$ is denoted by
$$\Fr(\lambda)=(l_{s-1},\ldots,l_0\mid a_0,\ldots,a_{s-1}).$$ 

\begin{example}
\label{ex:example2}
We continue Example~\ref{ex:example1} with $\lambda=(5,5,4,2,1,1)$. The
arms and legs are represented by horizontal and vertical dashes on
$[\lambda]$.
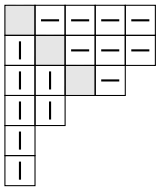
\medskip

\begin{center} 
\frob[(1,0),(2,-1),(3,-2)]
    {{\relax,\relax,\relax,\relax,\relax},{\relax,\relax,\relax,\relax,\relax},{\relax,\relax,\relax,\relax},{\relax,\relax},{\relax},{\relax}}
    {}
    {(2,0),(3,0),(4,0),(5,0),(3,-1),(4,-1),(5,-1),(4,-2)}
    {(1,-1),(1,-2),(1,-3),(1,-4),(1,-5),(2,-2),(2,-3)}
    {}
    {}
\end{center}
\medskip

Note that the box $(3,3)$ has no boxes in the same column and below. The
leg $l_0$ of this box is then $0$. We obtain
$$\Fr(\lambda)=(5,2,0\mid 1,3,4).$$
\end{example}


Let $t$ be a positive integer and $\lambda$ be a partition. The $t$-core
of $\lambda$ was first introduced by Nakayama \cite{Nakayama} in 1941
for $t$ prime using hook removals on Young diagrams. In 1952, Farahat
\cite{Farahat} showed that the star diagram method, introduced by
Robinson \cite{Robinson} in 1947 and furthered by Staal \cite{Staal}
in 1950, was equivalent to the $t$-quotient construction proposed by
Littlewood \cite{Littlewood} in 1951. For $t$ (not necessary a prime),
G. James \cite{james} proposed in 1978 the abacus method to arriving
at the $t$-core and $t$-quotient of an integer partition. One of the
main results of this article is to show that the $t$-core and the
$t$-quotient of $\lambda$ can be obtained directly from the Frobenius
symbol of $\lambda$. Conversely, we will show that the process, albeit
more complicated, also works in the opposite direction: we can recover
the Frobenius symbol of $\lambda$ from the ones of the $t$-core and of
the partitions appearing in the $t$-quotient.

The paper is organized as follows. In Section~\ref{sec:part1}, we recall
some usual concepts on partitions interpreted in term of Frobenius symbol.
In particular, in \S\ref{subsec:frobeniuslabel}, we will introduce a new
labeling of the boxes of the Young diagram of the partition with respect
to the Frobenius symbol. This labeling, called the Frobenius label of the
partition, is helpful to describe its parts, hooks, hooklengths etc, so
that we can read this information directly on the Frobenius symbol.
In~\S\ref{subsec:pointedabacus} we introduce the \emph{pointed abacus}
associated to a partition. In a way, this abacus is a geometric
interpretation of the Frobenius symbol, and our preferred representative
of the equivalence class of abaci associated to the partition. Note that
the approach proposed in the present work was also pursued by Olsson
in~\cite{OlssonFrob}, who used it to investigate the bar-core and
bar-quotient of a bar-partition.

In Section~\ref{sec:part2}, we discuss the $t$-core and $t$-quotient of a
integer partition. In order to make the connection with the Frobenius
symbol, we first introduce the \emph{pointed $t$-abacus} of a partition.
This gives a $t$-tuple of abaci, each of which, are {\it not}, in general,
the pointed abacus of some partition. We then introduce the
\emph{characteristic vector} of the partition, that appears naturally as a
correction factor to obtain $t$ pointed abaci of partitions (corresponding
to a $t$-tuple of partitions giving the $t$-quotient of the partition).
The characteristic vector is itself a $t$-tuple of integers whose entries
sum to $0$. We note that Garvan, Kim and Stanton~\cite{GarvanKimStanton},
using a star diagram-type approach, first noticed that $t$-core partitions
can be labeled by characteristic vectors. However, their construction is
not explicit, and the process introduced in the present paper gives a
natural bijection between the set of characteristic vectors and the one of
$t$-core partitions, which is of independent interest.  In
\S\ref{subsec:main} we describe how to recover the Frobenius symbol of a
partition from those of its $t$-cores and the partitions appearing in its
$t$-quotient. As a consequence, in \S\ref{subsec:durfee}, we calculate the
Durfee square of the partition from the ones of its $t$-core and
$t$-quotient. For this we will need some additional information, which we
can also obtain from the Frobenius symbol of the partition. 

In Section~\ref{sec:part3}, we give some applications in representation
theory. First, for any integer $t\geq 2$ and $h\in\Z$ coprime to $t$, we
describe the level $k$ action of the affine Weyl group $W_t$ of type $A$
on the Frobenius symbol of an integer partition. The action of $W_t$ on
the set of partitions was initially studied by Lascoux~\cite{Lascoux} for
the level $1$ and generalized for height level $k$ by Fayers
in~\cite{Fayers}. We then connect this to the process developed by Scopes
in $1991$ while proving Donovan's conjecture for the symmetric
groups~\cite{scopes}; that is, for any prime number $p$, there are a
finite number of $p$-blocks of symmetric groups with a given defect group
up to Morita equivalence. To do this, Scopes endows the set of $p$-blocks
of symmetric groups with an equivalence relation so that when two
$p$-blocks are in the same class, they are Morita equivalent. In the
following, the equivalence classes for this relation will be called Scopes
families. On the other hand, Scopes associates to each $p$-block of
symmetric groups a $p$-abacus, and shows that we can read on it by an
algorithmic process whenever two $p$-blocks lie in the same Scopes family.
Scopes then proves combinatorially that there are a finite number of
families by giving an explicit bound for this number.
In~\S\ref{subsec:scopes}, using the level $1$ action of $W_p$ above
described, we give alternate description of the Scopes's process. As a
consequence, we explicitly count the number of Scopes families. We also
compute the number of Scopes families that are infinite, as well as those
that are finite.

In Section~\ref{sec:part4}, we give some number-theoretic consequences of
our formulation for $t$-core partitions.  Such partitions appear in the
context of Ramanujan's congruences identities~\cite{Garvan}.
In~\cite[Corollary 2]{GarvanKimStanton}, it is proved, using generating
functions, that the number of self-conjugate $5$-core partitions of $n$ is
the same as that of $2n+1$ and of $5n+4$.  They also show that the numbers
of self-conjugate $7$-core partitions of $n$ is equal to that of $6n+4$.
In~\cite[Theorem 5.3]{BaruahBerndt}, Baruah and Berndt prove, using
modular forms, that the number of $3$-core partitions of $n$ and of $4n+1$
are equal.

Such relations are generalized in~\cite{HirschhornSellers} for $3$-core
partitions, and in~\cite{BaruahNath} for self-conjugate $5$-core and
$7$-core partitions. Note also that in~\cite[Theorem 4.1]{BaruahSarmah},
it is proved that the number of self-conjugate $9$-core partitions of $2n$
and that of $8n+10$ are the same.

Here we give explicit and combinatorial bijections that explain the
original equalities and their generalizations. 

Finally, in Section~\ref{sec:part5}, we discuss connections between
hooklengths of a partition and those of its $t$-quotient. In particular,
recall that the set of diagonal hooks of self-conjugate partitions arises
in the study of the representation theory of the alternating groups, since
it describes the non integral values of the character table of these
groups. In \S\ref{subsec:diagsym}, we will show that the diagonal hooks of
a self-conjugate partition can be recovered from its characteristic vector
and the Frobenius symbol of the partitions appearing in its $t$-quotient.
Obtaining such $t$-information is of independent interest. For example, a
crucial argument in the authors' previous work verifying the Navarro
conjecture for the alternating groups~\cite{BrNa} involved obtaining the
diagonal hooks from the $t$-core and $t$-quotient of self-conjugate
partitions, in order to describe the action on the character table of
alternating groups of some suitable Galois automorphisms that act on
$t'$-roots of unity by a power of $t$.  To finish, we will also describe
in \S\ref{subsec:bij} an explicit bijection between the set of boxes of
the partitions appearing in $t$-quotient of $\lambda$, and the boxes of
$\lambda$ whose hooklengths are divisible by $t$. This makes precise a
well-known correspondence between these two sets~\cite{MorrisOlsson}.

\section{The Frobenius symbol and the pointed abacus}
\label{sec:part1}

Throughout this section $\lambda$ denotes a partition of an integer $n$,
with Durfee number $s$ and Frobenius symbol
$\Fr(\lambda)=(l_{s-1},\ldots,l_0\mid a_0,\ldots,a_{s-1})$.  Note that, if
$n=0$ and $\lambda=\emptyset$, then we write $\Fr(\emptyset)=(\emptyset
\mid \emptyset)$ by convention.

\subsection{Arms, legs, coarms and colegs}
\label{subsec:armlegcoarmcoleg}

We denote by $\Ar^+(\lambda)=\{a_0,\ldots,a_{s-1}\}$ and
$\Le^+(\lambda)=\{l_0,\ldots,l_{s-1}\}$ the arm and leg sets of $\lambda$.

\begin{remark}
\label{rk:armlegbijpart}
Arms and legs sequences are strictly increasing and completely determine
$\lambda$. In particular, these sequences (and Frobenius symbol) can be
recovered from $\Ar^+(\lambda)$ and $\Le^+(\lambda)$, and we have
\begin{equation}
\label{eq:sizebyarmleg}
|\lambda|=s+\sum_{a\in\Ar^+(\lambda)}a+\sum_{l\in\Le^+(\lambda)}l.
\end{equation}
Conversely, any two finite sets $A$ and $L$ of nonnegative, strictly
increasing integers with the same cardinality $s$ can be interpreted as
the arms and legs sets of a partition.
\end{remark}

Now, we define the sets of \emph{coarms} $\Ar^-(\lambda)$ and of
\emph{colegs} $\Le^-(\lambda)$ associated to $\lambda$ by setting
$$\Ar^-(\lambda)=\{0,1,\ldots,a_{s-1}\}\backslash
\Le^+(\lambda)\quad\text{and}\quad\Le^-(\lambda)=
\{0,1,\ldots,l_{s-1}\}\backslash\Ar^+(\lambda),$$

\begin{remark}
\label{rk:abusdenotation}
Strictly speaking, arms, legs, coarms and colegs are non-negative
integers. Nevertheless, sometime in the following, we abuse the language
by confusing these integers and the corresponding boxes on the Young
diagram (see the picture of Example~\ref{ex:example1}). We can then speak
about the horizontal or vertical ``line'' that supports them.
\end{remark}

A diagonal box is a box of the positive quarter plane with coordinate
$(j,j)$ for some positive integer $j$, and we denote by $\mathfrak b_j$
the corresponding box in the following. We write $\mathfrak D^+(\lambda)$
for the set of diagonal boxes of $[\lambda]$.

To any arm, leg, coarm and coleg $x$ of $\lambda$, we associate a diagonal
box $d(x)$ by intersecting the diagonal and the line that supports $x$.
Note that if $x\in\Ar^+(\lambda)\cup\Le^+(\lambda)$, then
$d(x)\in\mathfrak D^+(\lambda)$. We denote by $\mathfrak D^-(\lambda)$ the
set of $d(x)$ for $x\in\Ar^-(\lambda)\cup\Le^-(\lambda)$. In particular,
for a coarm $a$ of $\lambda$, the number of horizontal boxes between the
Young diagram of $\lambda$ and $d(a)$ is $a$. Similarly, when $l$ is a
coleg, the number of vertical boxes is $l$. Note here what is meant by $x$
its set of boxes, not simply the integer; see
Remark~\ref{rk:abusdenotation}.

\begin{example}
\label{ex:example3}
For $\lambda=(5,5,4,2,1,1)$, we have 
$$\Ar^-(\lambda)=\{1,3,4\}\quad \text{and}\quad
\Le^-(\lambda)=\{0,2\}.$$ 
To each row outside of the Durfee square there corresponds a coarm, which
counts the horizontal boxes between the end of the row in the diagram and
the diagonal. Similarly for each column beyond the Durfee square, there is
a colegs which counts the number of vertical boxes between the end of the
column in the diagram and the diagonal.
\medskip

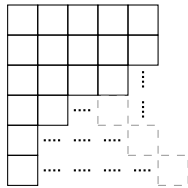
\begin{center} 
\frob[]
    {{\relax,\relax,\relax,\relax,\relax},{\relax,\relax,\relax,\relax,\relax},{\relax,\relax,\relax,\relax},{\relax,\relax},{\relax},{\relax}}
    {1,2,3}
    {}
    {}
    {(3,-3),(3,-4),(3,-5),(2,-5),(4,-5),(5,-5),(4,-4),(2,-4)}
    {(5,-2),(5,-3)}
\end{center}
\medskip

We have
$$\mathfrak D^-(\lambda)=\{\mathfrak b_4,\,\mathfrak
b_5,\,\mathfrak b_6\}.$$
Again, we notice that there are no empty boxes above $\mathfrak b_4$ and
the diagram. Hence, $0$ is a coleg of $\lambda$. We have $d(3)=\mathfrak
b_5=d(2)$. Note that $\mathfrak b_6$ is associated to the coarm $4$ but
has no associated coleg. In particular, there is no reason that the
numbers of coarms and colegs of $\lambda$ coincide.

We return again to Remark~\ref{rk:abusdenotation}.  The notation $d(3)$
can be confusing, because $3$ can be an arm, leg, coarm or coleg. It then
depends on the context, and again, $3$ can only be interpreted as the
three associated boxes of the Young diagram.  When $\lambda=(5,5,4,1,1,1)$
we have both $d(2)=\mathfrak b_4$ and $d(2)=\mathfrak b_5$. In the first
case, $2$ is a coarm, and in the second case, it is a coleg.
\end{example}

\subsection{Frobenius label}
\label{subsec:frobeniuslabel}

Having defined the sets of arms, legs, coarms and colegs of $\lambda$, we
can now introduce a new parametrization, called the \emph{Frobenius label}
of $\lambda$, of the boxes of $[\lambda]$ with respect to the diagonal
boxes. This section is devoted to the description of this parametrization.

First, we remark that to any two diagonal boxes $d$ and $d'$, we can
associate two boxes $x^+$ and $x^-$ that are the two other corners of the
square with diagonal $[dd']$. Note that when $d=d'$, the square has only
one box, and the associated box is $x^+=x^-=d$.
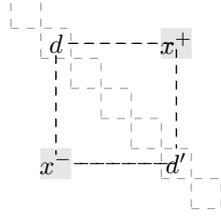
\bigskip

\begin{center}
\begin{tikzpicture}[scale=0.4,draw/.append style={black},baseline=\shadedBaseline]
           \foreach \x in {-1,0,1,2,3,4,5} {
      \draw[gray!80,dashed](\x+3,-2-\x)+(-.5,-.5)rectangle++(.5,.5);
      }
      \filldraw[gray!20,dashed](7,-2)+(-.5,-.5)rectangle++(.5,.5);
      \filldraw[gray!20,dashed](3,-6)+(-.5,-.5)rectangle++(.5,.5);
      \draw(3,-2)node{$d$}; 
      \draw(7,-2)node{$x^+$};
      \draw(3,-6)node{$x^-$};
      \draw(7,-6)node{$d'$};
      \draw[dashed](3,-3)+(0,0.6)rectangle++(0,-2.7);
      \draw[dashed](7,-2.5)+(0,0.3)rectangle++(0,-3);
      \draw[dashed](4,-2)+(-.6,0)rectangle++(3,.0);
      \draw[dashed](4,-6)+(-.4,0)rectangle++(3,.0);
\end{tikzpicture}
\end{center}
\bigskip
Let $j,\,j'$ be to positive integers. We set $d=\mathfrak b_j$ and
$d'=\mathfrak b_{j'}$. When $j\leq j'$, we write $$d\leq d'.$$ Assume
$d\leq d'$. We then denote $x^-_{dd'}$ and $x^+_{dd'}$ for the boxes
corresponding to $d$ and $d'$ (as in the above picture).
\bigskip

Let $\lambda$ be a partition. We write $\Ar^{\pm}(\lambda)$ and
$\Le^{\pm}(\lambda)$ for the sets of arms, coarms, legs and colegs of
$\lambda$ as above, and we set $$\mathfrak D(\lambda)=\mathfrak
D^+(\lambda)\cup\mathfrak D^-(\lambda),$$ where $\mathfrak
D^{\pm}(\lambda)$ are defined in \S\ref{subsec:armlegcoarmcoleg}. 

Let $d\in\mathfrak D(\lambda)$. If $d\in\mathfrak D(\lambda)^+$, then we denote
by $a(d)$ and $l(d)$ the arm and the leg of $\lambda$ attached to $d$. If
$d\in\mathfrak D(\lambda)^-$, then there is exactly one coarm $a(d)$ or
one coleg $l(d)$ of $\lambda$ that corresponds to $d$. 

\begin{lemma} 
\label{lem:justificationfroblabel}
Let $\lambda$ be a partition, $d\in\mathfrak D(\lambda)^+$ and
$d'\in\mathfrak D(\lambda)^-$.
\begin{enumerate}[(i)]
\item Assume $l(d')\in\Le^-(\lambda)$. Then the box $x^+_{dd'}$ belongs to
$[\lambda]$ if and only if $a(d)>l(d')$.
\item Assume $a(d')\in\Ar^-(\lambda)$. Then the box $x^-_{dd'}$ belongs to
$[\lambda]$ if and only if $l(d)>a(d')$.
\end{enumerate}
\end{lemma}

\begin{proof}
We only prove (i). The proof of (ii) is similar. Consider
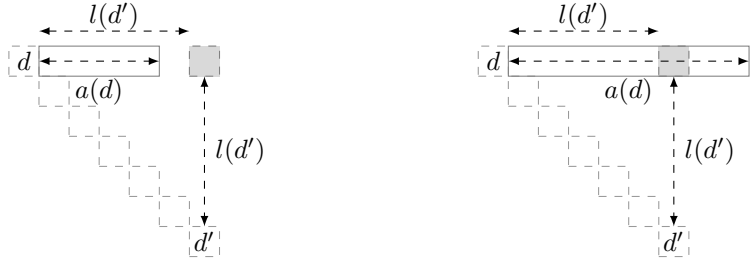
\begin{center}
\begin{tabular}{lcr}
\begin{tikzpicture}[scale=0.4,draw/.append style={black},baseline=\shadedBaseline]
      \foreach \x in {-1,0,1,2,3,4,5} {
      \draw[gray!80,dashed](\x+3,-2-\x)+(-.5,-.5)rectangle++(.5,.5);
      }
      \filldraw[gray!30,dashed](8,-1)+(-.5,-.5)rectangle++(.5,.5);
      \draw[gray,dashed](8,-1)+(-.5,-.5)rectangle++(.5,.5);

      \draw(2,-1)node{$d$}; 
      \draw(8,-7)node{$d'$};
      \draw[latex-latex,dashed] (8,-1.5)--(8,-6.5);
      \draw(5,0.5)node{$l(d')$};
      \draw(4.5,-2)node{$a(d)$};
      \draw(9.2,-4)node{$l(d')$};
      \draw[latex-latex,dashed] (2.5,0)--(7.5,0);
      \draw[latex-latex,dashed] (2.5,-1)--(6.5,-1);
      \draw[gray](3,-1)+(-.5,-.5)rectangle++(3.5,.5);
\end{tikzpicture}
&\hspace{2cm}
&
 \begin{tikzpicture}[scale=0.4,draw/.append style={black},baseline=\shadedBaseline]
         \foreach \x in {-1,0,1,2,3,4,5} {
       \draw[gray!80,dashed](\x+3,-2-\x)+(-.5,-.5)rectangle++(.5,.5);
       }
       \draw[fill,gray!30,dashed](8,-1)+(-.5,-.5)rectangle++(.5,.5);
       \draw[gray,dashed](8,-1)+(-.5,-.5)rectangle++(.5,.5);
       \draw(2,-1)node{$d$}; 
       \draw(8,-7)node{$d'$};
       \draw[latex-latex,dashed] (8,-1.5)--(8,-6.5);
       \draw(5,0.5)node{$l(d')$};
       \draw(6.5,-2)node{$a(d)$};
       \draw(9.2,-4)node{$l(d')$};
       \draw[latex-latex,dashed] (2.5,0)--(7.5,0);
       \draw[latex-latex,dashed] (2.5,-1)--(10.5,-1);
       \draw[gray](3,-1)+(-.5,-.5)rectangle++(7.5,.5);
 \end{tikzpicture}
\end{tabular}
\end{center}

The picture on the left shows that if $l(d')\geq a(d)$, then the box
$x^+_{dd'}$ does not belong to $[\lambda]$.  In the picture on the right,
we see that if $a(d)<l(d')$ then $x^+_{dd'}$ belongs to $[\lambda]$, as
required.
\end{proof}

\begin{proposition}
\label{prop:froblabelprop}
Let $\lambda$ be a partition.
\begin{enumerate}[(i)]
\item Let $l\in \Le^+(\lambda)$ and $a\in\Ar^+(\lambda)$. Then there exist
$d$ and $d'$ in $\mathfrak D^+(\lambda)$ such that $a(d)=a$ and $l(d')=l$,
and we associate to $(a,l)$ the box $x^+_{dd'}$ if $d\leq d'$ and
$x^-_{dd'}$ if $d\geq d'$. 
\item Let $l\in \Le^-(\lambda)$ and $a\in\Ar^+(\lambda)$ be such that
$a>l$. Then there exist $d\in\mathfrak D^+(\lambda)$ and $d'\in\mathfrak
D^(\lambda)$ such that $d\leq d'$, $a=a(d)$ and $l=l(d')$, and we
associate to $(a,l)$ the box $x^+_{dd'}$. 
\item Let $l\in \Le^+(\lambda)$ and $a\in\Ar^-(\lambda)$ be such that
$l>a$. Then there exist $d\in\mathfrak D^+(\lambda)$ and $d'\in\mathfrak
D^(\lambda)$ such that $d\leq d'$, $l=l(d)$ and $a=a(d')$, and we
associate to $(a,l)$ the box $x^-_{dd'}$. 
\end{enumerate}
Furthermore, each box of $\lambda$ is labeled by a unique pair $(a,l)$ as
in (i), (ii) or (iii).
\end{proposition}

\begin{proof}
This is a consequence of Lemma~\ref{lem:justificationfroblabel}. 
\end{proof}

\begin{definition}
\label{def:nombdesboxes}
The boxes obtained in (i) of Proposition~\ref{prop:froblabelprop} are
called the \emph{Durfee boxes} of $\lambda$. The ones obtained in (ii) and
(iii) are respectively the \emph{Arm-Coleg boxes} and the \emph{Leg-Coarm
boxes} of $\lambda$. A pair $(a,l)$ satifying the conditions of (ii) or
(iii) is called an \emph{admissible pair}.
\end{definition}

\begin{notation}
For a box $b$ of $\lambda$ with corresponding pair $(a,l)$ as in
Proposition~\ref{prop:froblabelprop}, we write 
$$\fl(b)=
\begin{cases}
(a,l)_{++}&\text{if $b$ is a Durfee box},\\
(a,l)_{+-}&\text{if $b$ is an Arm-Coleg box},\\
(a,l)_{-+}&\text{if $b$ is a Leg-Coarm box}.
\end{cases}
$$
Note that the notation $(a,l)_{\delta\nu}$ directly indicates that
$a\in\Ar^\delta(\lambda)$ and $l\in\Le^\nu(\lambda)$, and in the case that
$\delta\neq\nu$, the pair $(a,l)$ is admissible. For any box $b$ of
$\lambda$, $\fl(b)$ is called the \emph{Frobenius label} of $b$, and the
labeling of the boxes of $[\lambda]$ obtained in this way is \emph{the
Frobenius label} of $\lambda$.
\end{notation}

\begin{remark}
\label{rk:froblabelparam}
The Durfee boxes of $\lambda$ are labeled by
$\Ar^+(\lambda)\times\Le^+(\lambda)$. There are $s^2$ such boxes, where
$s$ is the Durfee number of $\lambda$. The Arm-Coleg boxes and the
Leg-Coarm boxes of $\lambda$ are respectively labeled by
$$ \{(a,l)_{+-}\in \Ar^+(\lambda)\times
\Le^-(\lambda)\mid a>l)\}\quad\text{and}\quad 
\{(a,l)_{-+}\in\Ar^-(\lambda)\times
\Le^+(\lambda)\mid l>a\}.$$
\end{remark}

\begin{remark}
There is a geometric interpretation to the box $b$ of $[\lambda]$ with
Frobenius label $\fl(b)=(a,l)_{\delta\nu}$ for signs $\delta$ and $\nu$.
By construction, it is the box of $[\lambda]$ at the intersection of the
lines supporting $a$ and $l$.
\end{remark}

\begin{example}
Consider 
\medskip

\begin{center} 
\frob[(1,0),(2,-1),(3,-2)]
    {{\relax,\relax,\relax,\relax,\relax},{\relax,\relax,$\alpha$,\relax,\relax},{\relax,\relax,\relax,\relax},{$\beta$,\relax},{\relax},{\relax}}
    {1,2,3}
    {}
    {}
    {}
    {}
\end{center}
\medskip

The box $\alpha$ is a Durfee box of $\lambda$. We have $\alpha=x^+_{dd'}$
where $d$ and $d'$ are the diagonal boxes $\mathfrak b_2$ and $\mathfrak
b_3$.  Furthermore, $a(d)=3$, $l(d)=2$, $a(d')=1$ and $l(d')=0$. Since
$d<d'$, the Frobenius label of $\alpha$ is $(a(d),l(d'))_{++}=(3,0)_{++}$,
and $\alpha$ is at the intersection of the second row (the line that
supports $a(d)$) and the third column (the line that supports $l(d')$).

Consider now $\beta=x^-_{dd'}$, where $d=\mathfrak b_1\in\mathfrak
D^+(\lambda)$ and $d'=\mathfrak b_4 \in\mathfrak D^-(\lambda)$. The
Frobenius label of $\beta$ is then $(a(d'),l(d))_{-+}=(1,5)_{-+}$. Note
that we can detect directly from the label that $\beta$ is a box of
$[\lambda]$ because $5-1=4>0$, hence $(1,5)$ is an admissible pair.
Nevertheless, the diagonal boxes $\mathfrak b_3\in\mathfrak D^+(\lambda)$ and
$\mathfrak b_5\in\mathfrak D^-(\lambda)$ label no boxes of $[\lambda]$. Indeed, we
have $a(\mathfrak b_3)=1$, $l(\mathfrak b_3)=0$, $a(\mathfrak b_5)=3$
and $l(\mathfrak b_5)=2$, and 
$$a(\mathfrak b_3)-l(\mathfrak b_5)=-1<0\quad\text{and}\quad
l(\mathfrak b_3)-a(\mathfrak b_5)=-3<0.$$ 
On the other hand, no box in the $[\lambda]$ is labeled by the choice of
$d=\mathfrak b_3$ on the diagonal and $d'=\mathfrak b_5$ on the extended
diagonal. Hence, it is an not an admissible pair.
\end{example} 
\smallskip
\begin{lemma}
\label{lem:parts}
Let $\lambda$ be a partition with Frobenius symbol
$\Fr(\lambda)=(l_{s-1},\ldots,l_0\mid a_0,\ldots,a_{s-1})$. The parts of
$\lambda$ are parametrized by $\Ar^+(\lambda)\cup\Ar^-(\lambda)$. 
\begin{enumerate}[(i)]
\item If $a\in\Ar^+(\lambda)$, then the corresponding part is
$s+|\{l\in\Le^-(\lambda)\mid l<a\}|$. 
\item If $a\in\Ar^-(\lambda)$, then the corresponding part is
$|\{l\in\Le^+(\lambda)\mid a<l\}|$.
\end{enumerate}
\end{lemma}

\begin{proof}
To each row of $[\lambda]$ there is exactly one associated arm or coarm.
Let $a\in\Ar^+(\lambda)$. The boxes of $[\lambda]$ lying in the row
associated to $a$ have Frobenius label 
$$\{(a,l)_{++}\mid l\in \Le^+(\lambda)\}\cup\{(a,l)_{+-}\mid l\in
\Le^-(\lambda),\, a-l>0\}.$$ 
Since $s=|\{(a,l)_{++}\mid l\in \Le^+(\lambda)\}|$, (i) follows.
Similarly, if $a$ is a coarm of $\lambda$, the admissible pairs $(a,l)$
with $l\in\Le^+(\lambda)$ label the boxes of the rows of $[\lambda]$
corresponding to $a$. 
\end{proof}
\smallskip

\begin{example}
Let $\lambda$ be such that $\Fr(\lambda)=(5,2,0\mid 1,3,4)$. This
partition has $s+|\Ar^-(\lambda)|=3+3=6$ parts. Then $\{(4,0),(4,1)\}$ are
the admissible pairs with arm $4$. Hence, the first part of $\lambda$ is
$3+2=5$.  
\end{example}

\begin{remark}
\label{rk:colonnes}
Let $\lambda$ be the partition with Frobenius symbol $(\Ar\mid\Le)$.  
By an argument similar to the proof of Lemma~\ref{lem:parts}, the columns
of $[\lambda]$ are parametrized by the set
$\Le^+(\lambda)\cup\Le^-(\lambda)$.  More precisely, if
$l\in\Le^+(\lambda)$, then the boxes of $[\lambda]$ lying in the column
associated to $l$ have Frobenius label 
$$\{(a,l)_{++}\mid a\in
\Ar^+(\lambda)\}\cup\{(a,l)_{-+}\mid a\in \Ar^-(\lambda),\, l-a>0\}.$$  
If $l\in\Le^-(\lambda)$, then $\{(a,l)_{+-}\mid a\in\Ar^+(\lambda)\mid
l<a\}$ is the set of the Frobenius label of the corresponding column boxes
of $[\lambda]$. 
\end{remark}

\subsection{Conjugate partitions}
\begin{definition}
Let $\lambda$ be a partition with Frobenius symbol
$\Fr(\lambda)=(\Le\mid\Ar)$. The \emph{conjugate partition} of $\lambda$,
denoted by $\lambda^*$, is the partition with Frobenius symbol
$(\Ar\mid\Le)$. A partition $\lambda$ is called \emph{self-conjugate} when
$\lambda=\lambda^*$. 
\end{definition}

We note that geometrically, $[\lambda^*]$ is obtained from $[\lambda]$ by
applying a reflection with respect to the diagonal.  In particular, we
have the following result.

\begin{lemma} Let $\Ar$ and $\Le$ be two sets with the same cardinality.
Then the parts of the partition with Frobenius symbol $(\Ar\mid\Le)$ are
the columns of the partition with Frobenius symbol $(\Le\mid\Ar)$.
\end{lemma}

\begin{remark}
The Durfee number of a partition is invariant under conjugation.
\end{remark}
\begin{example}
Let $\lambda=(5,5,4,2,1,1)$. Then $\lambda^*$ has Frobenius symbol
$(4,3,1\mid 0,2,5)$, that is, $\lambda^*=(6,4,3,3,2).$
\end{example}

\subsection{Hooks and hooklengths of a partition}
The notion of \emph{hooks} and \emph{hooklengths} attached to a partition
arise from the representation theory of the symmetric groups. For example
the dimension of the complex irreducible representation labeled by a
partition $\lambda$ can be derived from the hooklengths of $\lambda$.
Recall that the \emph{hook} associated to the box $b$ of $[\lambda]$ is
the set consisting of $b$ and of the boxes of $[\lambda]$ to the right and
below $b$.  The box $b$ is called \emph{the corner} of the hook. The set
of boxes in the hook is denoted by $\mathcal H(b)$ and the
\emph{hooklength} of the hook is $\mathfrak h(b)=|\mathcal H(b)|$. A hook
with $k$ boxes will be call a $k$-hook.

\begin{example}
\label{ex:hookcorner}
The hook with corner $(2,1)$ of the partition $\lambda=(5,5,4,2,1,1)$ is
\medskip

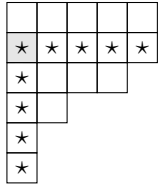
\begin{center} 
\frob[(1,-1)]
    {{\relax,\relax,\relax,\relax,\relax},{$\star$,$\star$,$\star$,$\star$,$\star$},{$\star$,\relax,\relax,\relax},{$\star$,\relax},{$\star$},{$\star$}}
    {}
    {}
    {}
    {}
    {}
\end{center}
\medskip
Its corresponding hooklength is $9$.
\end{example}

We now can describe the hooks and hooklengths of $\lambda$ using just the
Frobenius label of $\lambda$. 

\begin{lemma}
\label{lem:hooklength} Let $\lambda$ be a partition.
\begin{enumerate}[(i)]
\item If $a\in\Ar^+(\lambda)$ and $l\in\Le^+(\lambda)$, then
$$\mathfrak h(b)=a+l+1,$$
where $b$ is the Durfee box of $\lambda$ with Frobenius label
$(a,l)_{++}$.
\item Assume that $(a,l)$ is an admissible pair labeling a Leg-Coarm box
$b$ of $\lambda$.
Then 
$$\mathfrak h(b)=l-a.$$
\item Assume that $(a,l)$ is an admissible pair labeling an Arm-Coleg box
$b$ of $\lambda$.
Then 
$$\mathfrak h(b)=a-l.$$
\end{enumerate}
\end{lemma}

\begin{proof} 
Assume $d,\,d'\in \mathfrak D^+(\lambda)$ with $d\leq d'$ are such that
$a=a(d)$ and $l=l(d')$. Then $b=x^+_{dd'}$, and by construction the number
of boxes from $b$ to $d$ and to $d'$ is the same, in particular, the
number of boxes in the hook associated to $b$ over the diagonal is $a$.
The number of boxes under the diagonal is $l$.  Finally, to obtain
$\mathfrak h(b)$, we have to add one that corresponds to the box $d'$.

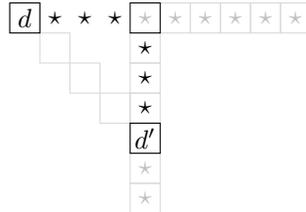
\begin{center}
\vspace{-10pt}
\begin{tikzpicture}[scale=0.4,draw/.append style={black},baseline=\shadedBaseline]
      \foreach \x in {0,1,2,3,4} {
      \draw[gray!30](\x+3,-2-\x)+(-.5,-.5)rectangle++(.5,.5);
      }
      \draw[gray!50](10,-2)node{$\star$};
      \draw[gray!50](9,-2)node{$\star$};
      \draw[gray!50](11,-2)node{$\star$};
      \draw(6,-2)node{$\star$};
      \draw(5,-2)node{$\star$};
      \draw(4,-2)node{$\star$};
      \draw(7,-3)node{$\star$};
      \draw(7,-4)node{$\star$};
      \draw(7,-5)node{$\star$};
      \draw[gray!50](12,-2)node{$\star$};
      \draw[gray!50](8,-2)node{$\star$};
      \draw[gray!50](7,-8)node{$\star$};
      \draw[gray!50](7,-7)node{$\star$};
      \draw[gray!30](12,-2)+(-.5,-.5)rectangle++(.5,.5);
      \draw[gray!30](11,-2)+(-.5,-.5)rectangle++(.5,.5);
      \draw[gray!30](10,-2)+(-.5,-.5)rectangle++(.5,.5);
      \draw[gray!30](9,-2)+(-.5,-.5)rectangle++(.5,.5);
      \draw[gray!30](8,-2)+(-.5,-.5)rectangle++(.5,.5);
      \draw[gray!30](7,-3)+(-.5,-.5)rectangle++(.5,.5);
      \draw[gray!30](7,-4)+(-.5,-.5)rectangle++(.5,.5);
      \draw[gray!30](7,-5)+(-.5,-.5)rectangle++(.5,.5);
      \draw[gray!30](7,-7)+(-.5,-.5)rectangle++(.5,.5);
      \draw[gray!30](7,-8)+(-.5,-.5)rectangle++(.5,.5);
      \draw(7,-2)+(-.5,-.5)rectangle++(.5,.5);
      \draw(7,-6)+(-.5,-.5)rectangle++(.5,.5);
      \draw(3,-2)+(-.5,-.5)rectangle++(.5,.5);
      \draw(3,-2)node{$d$}; 
      \draw[gray!50](7,-2)node{$\star$};
      \draw(7,-6)node{$d'$};
\end{tikzpicture}
\end{center}
The argument is similar for the case that $x^-_{dd'}=b$.  This proves (i).
We now consider $d\in\mathfrak D^+(\lambda)$ and  $d'\in\mathfrak
D^-(\lambda)$ such that $l(d)=l$ and $a(d')=a$. Hence, $b=x^{-}_{d,d'}$.
By construction, the number of boxes from $b$ to $d$ and $d'$ is the same,
denoted by $m$. Since the number of boxes in $[\lambda]$ right of $b$ is
equal to $m-a$, (ii) follows.

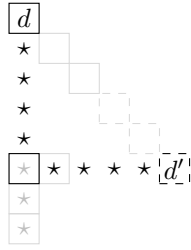
\begin{center}
\vspace{-10pt}
\begin{tikzpicture}[scale=0.4,draw/.append style={black},baseline=\shadedBaseline]
      \foreach \x in {0,1,2} {
      \draw[gray!30](\x+3,-2-\x)+(-.5,-.5)rectangle++(.5,.5);
      }
       \foreach \x in {3,4,5} {
      \draw[gray!30,dashed](\x+3,-2-\x)+(-.5,-.5)rectangle++(.5,.5);
      }
      \draw(3,-3)node{$\star$};
      \draw(3,-4)node{$\star$};
      \draw(3,-5)node{$\star$};
      \draw(3,-6)node{$\star$};
      \draw[gray!50](3,-7)node{$\star$};
      \draw[gray!50](3,-8)node{$\star$};
      \draw[gray!50](3,-9)node{$\star$};
      \draw(4,-7)node{$\star$};
      \draw(5,-7)node{$\star$};
      \draw(6,-7)node{$\star$};
      \draw(7,-7)node{$\star$};

      \foreach \x in {5,6,7} {
      \draw[gray!30](3,-\x-2)+(-.5,-.5)rectangle++(.5,.5);
      }
      
      \draw[gray!30](4,-7)+(-.5,-.5)rectangle++(.5,.5);
      \draw(3,-7)+(-.5,-.5)rectangle++(.5,.5);
      \draw(3,-2)+(-.5,-.5)rectangle++(.5,.5);
      \draw[dashed](8,-7)+(-.5,-.5)rectangle++(.5,.5);
      \draw(3,-2)node{$d$}; 
      \draw(8,-7)node{$d'$};
\end{tikzpicture}
\end{center}
\medskip

The argument is similar for (iii).
\end{proof}

\begin{remark}
\label{rk:removehooks}
Let $\lambda$ be a partition and $b$ be a box of $[\lambda]$. We now
describe, in terms of Frobenius symbol, how to define a new partition
after removing a hook with box $b$.
\begin{enumerate}[(i)]
\item If $b$ is a Durfee box, then the Frobenius label of $b$ is
$(a,l)_{++}$ where $a\in\Ar^+(\lambda)$ and $l\in\Le^+(\lambda)$. We then
associate a new partition $\lambda\backslash\mathcal{H}(b)$ with sets of
arms and legs 
$$\Ar^+(\lambda\backslash\mathcal H(b))=\Ar^+(\lambda)\backslash\{a\}
\quad\text{and}\quad 
\Le^+(\lambda\backslash\mathcal H(b))=\Le^+(\lambda)\backslash\{l\}.$$
In particular, Equality~(\ref{eq:sizebyarmleg}) gives 
$$|\lambda\backslash\mathcal H(b)|=|\lambda|-(a+l+1)=|\lambda|-\mathfrak h(b).$$
We say that $\lambda\backslash\mathcal H(b)$ is the partition obtained
from $\lambda$ by removing the $\mathfrak H(b)$-hook with corner $b$.
Geometrically, to obtain $[\lambda\backslash\mathcal H(b)]$ from
$[\lambda]$ we proceed as follows. First, remove, in $[\lambda]$, the arm
$a$, the leg $l$ and the diagonal box $\mathfrak b_s$, where $s$ is
the Durfee number of $\lambda$.  Then, slide up along the diagonal the
arms of $\lambda$ smaller that $a$, and the legs smaller that $l$. 
\end{enumerate}
\bigskip 

\begin{example}
We remove the hook $b$ of $\lambda=(5,5,4,2,1,1)$ with corner $(2,1)$; see
Example~\ref{ex:hookcorner}. The Frobenius label of $b$ is $(3,5)_{++}$.
Then
$$\Fr(\lambda\backslash\mathcal H(b))=(2,0\mid 1,4),$$
and $\lambda\backslash\mathcal H(b)=(5,2,1)$.
\bigskip

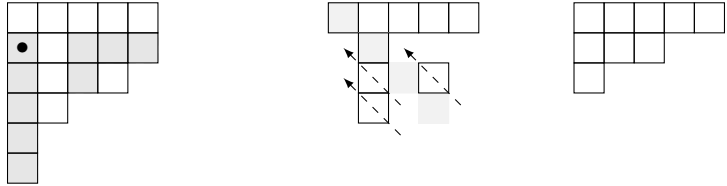
\begin{center} 
\frob[(1,-1),(1,-2),(1,-3),(1,-4),(1,-5),(3,-1),(4,-1),(5,-1), (3,-2)]
    {{\relax,\relax,\relax,\relax,\relax},{$\bullet$,$\relax$,$\relax$,$\relax$,$\relax$},{$\relax$,\relax,\relax,\relax},{$\relax$,\relax},{$\relax$},{$\relax$}}
    {}
    {}
    {}
    {}
    {}
\hspace{2cm}
\begin{tikzpicture}[scale=0.4,draw/.append style={black},baseline=\shadedBaseline]
      \draw[gray!10,fill](0,0)+(-.5,-.5)rectangle++(.5,.5);
      \draw(0,0)+(-.5,-.5)rectangle++(.5,.5);
      \draw(1,0)+(-.5,-.5)rectangle++(.5,.5);
      \draw(2,0)+(-.5,-.5)rectangle++(.5,.5);
      \draw(3,0)+(-.5,-.5)rectangle++(.5,.5);
      \draw(4,0)+(-.5,-.5)rectangle++(.5,.5);
      \draw[gray!10,fill](1,-1)+(-.5,-.5)rectangle++(.5,.5);
      \draw[gray!10,fill](2,-2)+(-.5,-.5)rectangle++(.5,.5);
      \draw[gray!10,fill](3,-3)+(-.5,-.5)rectangle++(.5,.5);
      \draw(1,-1)+(-.5,-.5)rectangle++(.5,.5);
      \draw(1,-2)+(-.5,-.5)rectangle++(.5,.5);
      \draw(1,-3)+(-.5,-.5)rectangle++(.5,.5);
      \draw(3,-2)+(-.5,-.5)rectangle++(.5,.5);
      \draw[latex-,dashed](0,-2)--(2,-4);
      \draw[latex-,dashed](0,-1)--(2,-3);
      \draw[latex-,dashed](2,-1)--(4,-3);
\end{tikzpicture}
\hspace{1cm}
\frob[]
    {{\relax,\relax,\relax,\relax,\relax},{$\relax$,\relax,\relax},{$\relax$}}
    {}
    {}
    {}
    {}
    {}
\end{center}
\medskip
\end{example}
\begin{enumerate}[(ii)]
\item Assume $b$ is a Leg-Coarm box of $\lambda$ with Frobenius label
$(a,l)_{-+}$. We define $\lambda\backslash\mathcal H(b)$ by the partition
with arm and leg sets
$$\Ar^+(\lambda\backslash\mathcal H(b))=\Ar^+(\lambda)\quad\text{and}\quad
\Le^+(\lambda\backslash\mathcal H(b))=\Le^+(\lambda)\backslash\{l\}\cup\{a\}.$$
We remark that $\lambda$ and $\lambda\backslash\mathcal H(b)$ have the
same Durfee number $s$, and
$$|\lambda\backslash\mathcal H(b)|=s+\sum_{x\in\Ar^+(\lambda)}x+\sum_{y\in\Le^+(\lambda)\backslash\{l\}\cup\{a\}}y=|\lambda|+a-l=|\lambda|-\mathfrak h(b).$$
As above, $\lambda\backslash\mathcal H(b)$ is called the partition
obtained from $\lambda$ by removing the $\mathfrak H(b)$-hook with corner
$b$. We describe the process to move to $[\lambda\backslash\mathcal
H(b)]$ from $[\lambda]$ as follows. First, we remove the leg $l$ of
$[\lambda]$. We then slide up along the diagonal past the legs smaller
that $l$, and we insert a leg $a$ so that the decreasing order of legs is
respected. 
\end{enumerate}

\begin{example}  
Consider $\lambda=(5,5,4,2,1,1)$ such that $\Fr(\lambda)=(5,2,0\mid
1,3,4)$. Now we remove the box with position $(4,1)$. Its Frobenius label
is $(1,5)_{-+}\in\Ar^-(\lambda)\times\Le^+(\lambda)$. Then
$$\Fr(\lambda\backslash\mathfrak H(b))=(2,1,0 \mid 1,3,4)\quad\text{and}\quad\lambda\backslash\mathfrak H(b)=(5,5,4).$$
\vspace{-10pt}
\begin{center} 
\frob[(1,-1),(1,-2),(1,-3),(1,-4),(1,-5)]
    {{\relax,\relax,\relax,\relax,\relax},{$\relax$,$\relax$,$\relax$,$\relax$,$\relax$},{$\relax$,\relax,\relax,\relax},{$\bullet$,\relax},{$\relax$},{$\relax$}}
    {}
    {}
    {}
    {}
    {}
\hspace{1cm}
\begin{tikzpicture}[scale=0.4,draw/.append style={black},baseline=\shadedBaseline]
      \draw(0,0)+(-.5,-.5)rectangle++(.5,.5);
      \draw(1,0)+(-.5,-.5)rectangle++(.5,.5);
      \draw(2,0)+(-.5,-.5)rectangle++(.5,.5);
      \draw(3,0)+(-.5,-.5)rectangle++(.5,.5);
      \draw(4,0)+(-.5,-.5)rectangle++(.5,.5);
      \draw(2,-1)+(-.5,-.5)rectangle++(.5,.5);
      \draw(3,-1)+(-.5,-.5)rectangle++(.5,.5);
      \draw(4,-1)+(-.5,-.5)rectangle++(.5,.5);
      \draw(2,-2)+(-.5,-.5)rectangle++(.5,.5);
      \draw[gray!20,fill](2,-3)+(-.5,-.5)rectangle++(.5,.5);

            \draw(1,-1)+(-.5,-.5)rectangle++(.5,.5);
      \draw(1,-2)+(-.5,-.5)rectangle++(.5,.5);
      \draw(1,-3)+(-.5,-.5)rectangle++(.5,.5);
      \draw(3,-2)+(-.5,-.5)rectangle++(.5,.5);
      \end{tikzpicture}
\hspace{1cm}
\begin{tikzpicture}[scale=0.4,draw/.append style={black},baseline=\shadedBaseline]
      \draw[gray!10,fill](0,0)+(-.5,-.5)rectangle++(.5,.5);
      \draw[](0,0)+(-.5,-.5)rectangle++(.5,.5);
      \draw(1,0)+(-.5,-.5)rectangle++(.5,.5);
      \draw(2,0)+(-.5,-.5)rectangle++(.5,.5);
      \draw(3,0)+(-.5,-.5)rectangle++(.5,.5);
      \draw(4,0)+(-.5,-.5)rectangle++(.5,.5);
      \draw(2,-1)+(-.5,-.5)rectangle++(.5,.5);
      \draw(3,-1)+(-.5,-.5)rectangle++(.5,.5);
      \draw(4,-1)+(-.5,-.5)rectangle++(.5,.5);
      \draw(2,-2)+(-.5,-.5)rectangle++(.5,.5);
      \draw(2,-3)+(-.5,-.5)rectangle++(.5,.5);

      \draw[gray!10,fill](1,-1)+(-.5,-.5)rectangle++(.5,.5);
      \draw[gray!10,fill](2,-2)+(-.5,-.5)rectangle++(.5,.5);
      \draw(2,-2)+(-.5,-.5)rectangle++(.5,.5);

      \draw(1,-1)+(-.5,-.5)rectangle++(.5,.5);
      \draw(1,-2)+(-.5,-.5)rectangle++(.5,.5);
      \draw(1,-3)+(-.5,-.5)rectangle++(.5,.5);
      \draw(3,-2)+(-.5,-.5)rectangle++(.5,.5);
      \draw[latex-,dashed](0,-2)--(2,-4);
      \draw[latex-,dashed](0,-1)--(3,-4);
\end{tikzpicture}
\hspace{1cm}
\frob[]
    {{\relax,\relax,\relax,\relax,\relax},{$\relax$,\relax,\relax,\relax,\relax},{$\relax$,\relax,\relax,\relax}}
    {}
    {}
    {}
    {}
    {}
\end{center}
\medskip
\end{example}
The Arm-Coleg case is similar.
\end{remark}

\subsection{Frobenius symbol and pointed abacus}
\label{subsec:pointedabacus}
There is a natural way to represent a Frobenius symbol using its
associated pointed abacus.  An \emph{abacus} is a ``strip'', also called
\emph{runner}, with an infinity of regularly spaced slots. A \emph{pointed
abacus} is an abacus which is ``pointed'', that is two consecutive slots
are chosen and a dash, called \emph{the fence} of the abacus, is drawn
between them. We then index the slots over and the slots under the fence
by the sets of nonnegative integers. The slots over the fence are called
\emph{the positive slots}; the ones under the fence \emph{the negative
slots}.  We likewise assume that there are white or black beads at each
slots and that the number $u$ of black beads over the fence and $v$ the
one of white beads under are finite. The beads over and under the fence
are called \emph{positive} and \emph{negative}, respectively.  

Consider a pointed abacus. With the notation as above, when $u=v$, we say
this is a\emph{ pointed abacus of a partition}.

Let $\lambda$ be a partition with $\Fr(\lambda)=(l_{s-1},\ldots,l_0\mid
a_0,\ldots, a_{s-1})$.  The pointed abacus of $\lambda$ is constructed as
follows. First, we take a pointed runner with fence $\mathfrak f$, and
place white beads on the positive slots and black beads on the negative
ones. Then for $0\leq j\leq s-1$, we replace the white bead at the $a_j$
positive slot by a black bead, and black bead at the $l_j$ negative slot
by a white bead. 

Remark~\ref{rk:armlegbijpart} shows that there is a bijection between the
set of pointed abacus of partitions and the set of partitions of integers. 

\begin{example}
\label{ex:pointabacus}
The pointed abacus of $\lambda=(5,5,4,2,1,1)$ with $\Fr(\lambda)=(5,2,0|1,3,4)$ is
\medskip
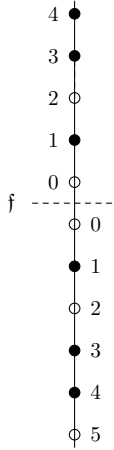
\begin{center}
\definecolor{sqsqsq}{rgb}{0.12549019607843137,0.12549019607843137,0.12549019607843137}
\begin{tikzpicture}[line cap=round,line join=round,>=triangle
45,x=0.7cm,y=0.7cm, scale=0.8,every node/.style={scale=0.8}]
\draw (0.,6.3)-- (0.,-4.3);
\draw [dash pattern=on 2pt off 2pt](0.,4.)-- (0.,5.);
\draw [dash pattern=on 2pt off 2pt](-1,1.5)-- (1,1.5);
\draw(-0.5,2)node{$0$};
\draw(-0.5,3)node{$1$};
\draw(-0.5,4)node{$2$};
\draw(-0.5,5)node{$3$};
\draw(-0.5,6)node{$4$};

\draw(-1.5,1.5)node{$\mathfrak f$};
\draw(0.5,1)node{$0$};
\draw(0.5,0)node{$1$};
\draw(0.5,-1)node{$2$};
\draw(0.5,-2)node{$3$};
\draw(0.5,-3)node{$4$};
\draw(0.5,-4)node{$5$};

\begin{scriptsize}
\draw [] (0.,2.) circle (2.5pt);
\draw [fill=black] (0.,3.) circle (2.5pt);
\draw [] (0.,4.) circle (2.5pt);
\draw [fill=black] (0.,5.) circle (2.5pt);
\draw [fill=black] (0.,6.) circle (2.5pt);

\draw [] (0.,1.) circle (2.5pt);
\draw [fill=black] (0.,0.) circle (2.5pt);

\draw  (0.,-1.) circle (2.5pt);
\draw [fill=black](0.,-2.) circle (2.5pt);
\draw [fill=black] (0.,-3.) circle (2.5pt);
\draw [] (0.,-4.) circle (2.5pt);

\end{scriptsize}
\end{tikzpicture}
\end{center}
\end{example}
\begin{remark}
\label{rk:abacusinfo}
Many of properties of $\lambda$ can be directly seen on its pointed
abacus. Here is a summary.
\begin{enumerate}[(i)]
\item Coarms and colegs can be immediately read on the pointed abacus of
$\lambda$. Indeed, the coarms to the black beads below, and colegs to the
white beads above, the fence.  
\item The number of parts of $\lambda$ is the number of black beads that
appear above the first white bead. The part associated to a black bead is
the number of white beads under it. See Lemma~\ref{lem:parts}.
\item By Remark~\ref{rk:froblabelparam}, the boxes of $[\lambda]$ are in
bijection with the pairs of beads $(b,w)$ where $b$ is black, $w$ is
white, and $w$ is under $b$. 
\item By Lemma~\ref{lem:hooklength}, the hook of $\lambda$ with corner
labeled by $(b,w)$ as in (iii) has hooklength the number of beads strictly
between $b$ and $w$ plus one. On the other hand, if we exchange the color
of these two beads, then we obtain the pointed abacus of the partition
obtained from $\lambda$ by removing the hook.
\smallskip
\end{enumerate}
\end{remark}

\begin{remark} 
\label{rk:usualabacus}
We now recall the ``usual'' way to represent a partition on an
abacus~\cite{james}. Write $\mathcal S$ for the set of sequences
$(l_j)_{j\in\Z}$ with $l_j\in\{0,1\}$ and there are $a,\,b\in\Z$ such that
$l_j=1$ (resp. $l_j=0$) for any $j\leq a$ (resp. $j\geq b$). For any
partition $\lambda$, we associate to $\lambda$ such a sequence as follows.
First fix $a\in \Z$. For any $j\leq a$, set $l_j=1$. We begin on the
bottom of the Young diagram of $[\lambda]$ and following the rim, set, at
each step $k$, $l_{a+k}=0$ if we go right, or $l_{a+k}=1$ if we go up. We
continue by following the $x$-axis, hence $l_j=0$ for all $j\geq b$ for
some $b\geq a$. Note that this construction depends on a choice of
$a\in\Z$. Thus, we associate bijectively to $\lambda$ such a sequence, up
to a shift.

Let $(l_j)_{j\in\Z}\in\mathcal S$ be associated to $\lambda$ as above.  A
corresponding $\beta$-sequence is the set $X=\{j\in\Z\mid l_j=1\}$.  Here
we consider here $\beta$-sequences instead $\beta$-sets as it is
in~\cite{james}. They are essentially the same object: a $\beta$-sequence
can be viewed as a $\beta$-set completed by the negative numbers; see also
\cite[page 224]{OlssonFrob}.  For $r\in\Z$, we set $X^{+r}=X+r$ which is
another $\beta$-sequence associated to the same partition $\lambda$. Such
a $\beta$-sequence $X^{+r}$ can be represented on an abacus (where slot
are labeled by $\Z$) by putting a black bead on the slot labeled $x\in
X^{+r}$ and a white bead otherwise. The abacus of $X^{+r}$ can be obtained
from the one of $X$ by pushing it $r$ times up if $r\geq 0$ and $|r|$
times down otherwise. Let $\lambda=(\lambda_1,\ldots,\lambda_m)$. For
$j>m$, set $\lambda_j=0$. Then \cite[2.2.7]{James-Kerber} implies that the
$\beta$-sequences of $\lambda$ are of the form
$$X^{+r}(\lambda)=\{\lambda_j-j+r\mid j\geq 1\}\quad\text{for }r\in\Z.$$
Now, we will show that the pointed abacus of $\lambda$ corresponds to the
usual abacus of $\lambda$ associated to $r=0$.  Write $\phi:\N\rightarrow
\Z_-^*,\, t\mapsto -t-1$, and $Y$ for the $\beta$-sequence whose elements
are $\Ar^+(\lambda)\cup\phi(\Ar^-(\lambda))$. The abacus attached to $Y$
is the pointed abacus of $\lambda$. Furthermore, we have
$$\Ar^+(\lambda)=\{\lambda_j-j\mid \lambda_j-j\geq 0\}.$$ Let $j\geq 1$ be
such that $\lambda_j-j<0$. By Lemma~\ref{lem:parts}, $\lambda_j$ is
parametrized by a coarm $a\in \Ar^-(\lambda)$.  Observe that
\medskip

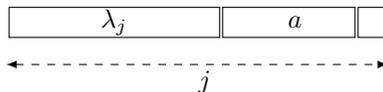
\begin{center}
\vspace{-10pt}
\begin{tikzpicture}[scale=0.4,draw/.append style={black},baseline=\shadedBaseline]
      \draw(3,-2)+(3.6,-.5)rectangle++(8,.5);
      \draw(3,-2)+(-3.5,-.5)rectangle++(3.5,.5);
      \draw(11.6,-2)+(-.5,-.5)rectangle++(.5,.5);
      \draw(3,-2)node{$\lambda_j$}; 
      \draw(9,-2)node{$a$};
      \draw(6,-4)node{$j$};
      \draw[latex-latex,dashed] (-0.6,-3.4)--(12.1,-3.4);
\end{tikzpicture}
\end{center}
\bigskip
Hence, $\lambda_j+a+1=j$, and 
$$\alpha(a)=\alpha(j-\lambda_j-1)=\lambda_j-j+1-1=\lambda_j-j.$$
It follows that $$Y=\{\lambda_j-j\mid j\geq 1\},$$ and the pointed abacus
of $\lambda$ is then the usual abacus of $\lambda$ obtained for $r=0$, as
required.
\end{remark}

\begin{remark}
\label{rk:pointedconjugate}
Let $\lambda$ be a partition with $\Fr(\lambda)=(\Le\mid\Ar)$. Since
$\Fr(\lambda^*)=(\Ar\mid\Le)$, the pointed abacus of the conjugate
partition $\lambda^*$ can be obtained from the one of $\lambda$ by 
\begin{enumerate}[(i)]
\item Reflecting the beads with respect to across the fence line.
\item Exchanging the color of each bead.
\end{enumerate}

Let $\mathcal R$ be a runner with white and black beads and a fence. Here
we do not assume that the abacus is pointed. We denote by $\overline{R}$
the abacus obtained by applying operations (i) and (ii) described in
Remark \ref{rk:pointedconjugate}. We have

\begin{lemma}
\label{lem:pointedtabaconjugate}
An abacus $\mathcal R$ is the pointed abacus of $\lambda$ if and only if
$\overline{\mathcal R}$ is the pointed abacus of $\lambda^*$.
\end{lemma}

\begin{example} Let $\lambda=(5,5,4,2,1,1)$.  In
Figure~\ref{fig:exconjugue}, we describe the process to obtain the abacus
$\overline{\mathcal R}$ of $\lambda^*$ from the one of $\lambda$ given in
Example~\ref{ex:pointabacus}. 
\begin{figure}
\small
\medskip

\begin{center}
\begin{tabular}{lcccr}
\begin{tikzpicture}[line cap=round,line join=round,>=triangle
45,x=0.7cm,y=0.7cm, scale=0.8,every node/.style={scale=0.8}]
\draw (0.,6.3)-- (0.,-4.3);
\draw [dash pattern=on 2pt off 2pt](0.,4.)-- (0.,5.);
\draw [dash pattern=on 2pt off 2pt](-1,1.5)-- (1,1.5);
\draw(-0.5,3)node{$1$};
\draw(-0.5,5)node{$3$};
\draw(-0.5,6)node{$4$};

\draw(-1.5,1.5)node{$\mathfrak f$};
\draw(0.5,1)node{$0$};
\draw(0.5,-1)node{$2$};
\draw(0.5,-4)node{$5$};

\draw(0.1,-5)node{$\mathcal R$};

\begin{tiny}
\draw [] (0.,2.) circle (2.5pt);
\draw [fill=black] (0.,3.) circle (2.5pt);
\draw [] (0.,4.) circle (2.5pt);
\draw [fill=black] (0.,5.) circle (2.5pt);
\draw [fill=black] (0.,6.) circle (2.5pt);

\draw [] (0.,1.) circle (2.5pt);
\draw [fill=black] (0.,0.) circle (2.5pt);
\draw  (0.,-1.) circle (2.5pt);
\draw [fill=black](0.,-2.) circle (2.5pt);
\draw [fill=black] (0.,-3.) circle (2.5pt);
\draw [] (0.,-4.) circle (2.5pt);

\end{tiny}
\end{tikzpicture}
&
\hspace{0.1cm}
\begin{tikzpicture}[line cap=round,line join=round,>=triangle
45,x=0.7cm,y=0.7cm, scale=0.8,every node/.style={scale=0.8}]
\path[line width=1pt,-latex](-1,10.65) edge (1,10.65);
\draw [dash pattern=on 2pt off 2pt](-1,10.65)-- (1,10.6);
\draw(-0.5,4)node{\ };
\draw(-0.1,11.2)node{(i)};
\end{tikzpicture}
\hspace{0.1cm}
&
\begin{tikzpicture}[line cap=round,line join=round,>=triangle
45,x=0.7cm,y=0.7cm, scale=0.8,every node/.style={scale=0.8}]
\draw (0.,7.3)-- (0.,-4.3);
\draw [dash pattern=on 2pt off 2pt](0.,4.)-- (0.,5.);
\draw [dash pattern=on 2pt off 2pt](-1,1.5)-- (1,1.5);
\draw(-0.5,2)node{$0$};
\draw(-0.5,4)node{$2$};
\draw(-0.5,7)node{$5$};

\draw(-1.5,1.5)node{$\mathfrak f$};
\draw(0.5,0)node{$1$};
\draw(0.5,-2)node{$3$};
\draw(0.5,-3)node{$4$};
\draw(0.1,-5.1)node{\ };

\begin{tiny}

\draw [] (0.,1.) circle (2.5pt);
\draw [fill=black] (0.,0.) circle (2.5pt);
\draw [] (0.,-1.) circle (2.5pt);
\draw [fill=black] (0.,-2.) circle (2.5pt);
\draw [fill=black] (0.,-3.) circle (2.5pt);
\draw (0.,-4.) circle (2.5pt);

\draw [] (0.,2.) circle (2.5pt);
\draw [fill=black] (0.,3.) circle (2.5pt);
\draw (0.,4.) circle (2.5pt);
\draw [fill=black] [] (0.,5) circle (2.5pt);
\draw [fill=black] (0.,6) circle (2.5pt);
\draw [] (0.,7) circle (2.5pt);

\end{tiny}
\end{tikzpicture}
&
\hspace{0.1cm}
\begin{tikzpicture}[line cap=round,line join=round,>=triangle
45,x=0.7cm,y=0.7cm, scale=0.8,every node/.style={scale=0.8}]
\path[line width=1pt,-latex](-1,10.65) edge (1,10.65);
\draw [dash pattern=on 2pt off 2pt](-1,10.65)-- (1,10.6);
\draw(-0.5,4)node{\ };
\draw(-0.1,11.2)node{(ii)};
\end{tikzpicture}
\hspace{0.1cm}
&
\begin{tikzpicture}[line cap=round,line join=round,>=triangle
45,x=0.7cm,y=0.7cm, scale=0.8,every node/.style={scale=0.8}]
\draw (0.,7.3)-- (0.,-4.3);
\draw [dash pattern=on 2pt off 2pt](0.,4.)-- (0.,5.);
\draw [dash pattern=on 2pt off 2pt](-1,1.5)-- (1,1.5);
\draw(-0.5,2)node{$0$};
\draw(-0.5,4)node{$2$};
\draw(-0.5,7)node{$5$};

\draw(-1.5,1.5)node{$\mathfrak f$};
\draw(0.5,0)node{$1$};
\draw(0.5,-2)node{$3$};
\draw(0.5,-3)node{$4$};
\draw(0.1,-5)node{$\overline{ \mathcal R}$};

\begin{tiny}

\draw [fill=black] (0.,1.) circle (2.5pt);
\draw [] (0.,0.) circle (2.5pt);
\draw [fill=black] (0.,-1.) circle (2.5pt);
\draw [] (0.,-2.) circle (2.5pt);
\draw [] (0.,-3.) circle (2.5pt);
\draw [fill=black] (0.,-4.) circle (2.5pt);

\draw [fill=black] (0.,2.) circle (2.5pt);
\draw [] (0.,3.) circle (2.5pt);
\draw [fill=black](0.,4.) circle (2.5pt);
\draw [] (0.,5) circle (2.5pt);
\draw [] (0.,6) circle (2.5pt);
\draw [fill=black] (0.,7) circle (2.5pt);

\end{tiny}
\end{tikzpicture}
\end{tabular}
\end{center}
\caption{Abacus of the conjugate partition of $\lambda=(5,5,4,2,1,1)$}
\label{fig:exconjugue}
\end{figure}
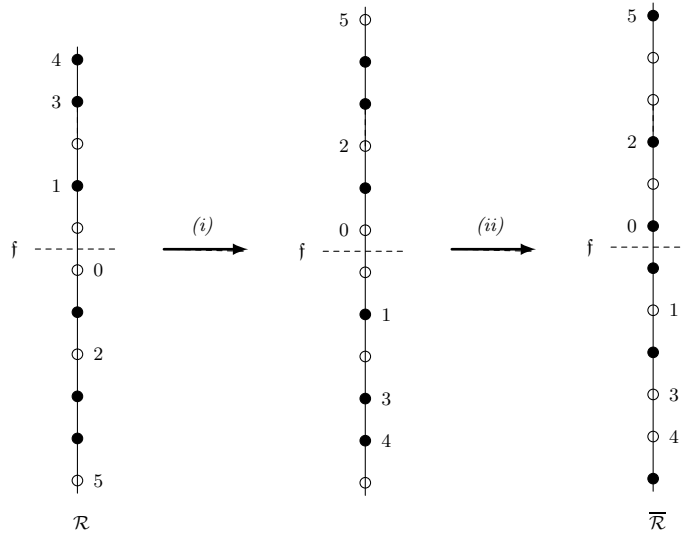
\end{example}
\normalsize
\end{remark}

\section{Quotient, core and Frobenius symbol}
\label{sec:part2}

Let $t$ be a positive integer. In this section, we recover some well-known
results from the point of view of the Frobenius symbol.  In particular, we
will see that the $t$-core and $t$-quotient of a partition can be directly
read on its Frobenius symbol.

\subsection{Pointed $t$-abacus, characteristic vector and $t$-quotient of
a partition}
\label{subsec:quotient}
A $t$-abacus is a picture consisting of $t$ runners, labeled from $0$ to
$t-1$, so that each runner is an abacus. The $t$-abacus is said to be
\emph{pointed} when the abacus on runner $j$ (for $0\leq j\leq t-1$) is a
pointed abacus and each fences are aligned on the picture.
\smallskip
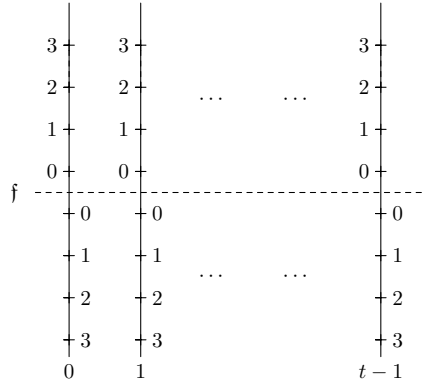
\begin{center}
\definecolor{sqsqsq}{rgb}{0.12549019607843137,0.12549019607843137,0.12549019607843137}
\begin{tikzpicture}[line cap=round,line join=round,>=triangle
45,x=0.7cm,y=0.7cm, scale=0.8,every node/.style={scale=0.8}]

\draw [dash pattern=on 2pt off 2pt](-2.5,1.5)-- (7,1.5);
\draw(-3,1.5)node{$\mathfrak f$};
\draw (1.2,4) node[anchor=north west] {$\cdots$};
\draw (1.2,-0.2) node[anchor=north west] {$\cdots$};
\draw (3.2,4) node[anchor=north west] {$\cdots$};
\draw (3.2,-0.2) node[anchor=north west] {$\cdots$};

\draw (-2,-2.4) node[anchor=north west] {$0$};

\draw (-1.7,6)-- (-1.7,-2.4);
\draw [dash pattern=on 2pt off 2pt](-1.7,4.)-- (-1.7,5.);
\draw(-2.1,2)node{$0$};
\draw(-2.1,3)node{$1$};
\draw(-2.1,4)node{$2$};
\draw(-2.1,5)node{$3$};

\draw(-1.3,1)node{$0$};
\draw(-1.3,0)node{$1$};
\draw(-1.3,-1)node{$2$};
\draw(-1.3,-2)node{$3$};

\begin{scriptsize}
\draw (-1.7,2)-- ++(-2.5pt,0 pt) -- ++(5.0pt,0 pt) ++(-2.5pt,-2.5pt) -- ++(0 pt,5.0pt);
\draw (-1.7,3)-- ++(-2.5pt,0 pt) -- ++(5.0pt,0 pt) ++(-2.5pt,-2.5pt) -- ++(0 pt,5.0pt);
\draw (-1.7,4)-- ++(-2.5pt,0 pt) -- ++(5.0pt,0 pt) ++(-2.5pt,-2.5pt) -- ++(0 pt,5.0pt);
\draw (-1.7,5)-- ++(-2.5pt,0 pt) -- ++(5.0pt,0 pt) ++(-2.5pt,-2.5pt) -- ++(0 pt,5.0pt);

\draw (-1.7,1)-- ++(-2.5pt,0 pt) -- ++(5.0pt,0 pt) ++(-2.5pt,-2.5pt) -- ++(0 pt,5.0pt);
\draw (-1.7,0)-- ++(-2.5pt,0 pt) -- ++(5.0pt,0 pt) ++(-2.5pt,-2.5pt) -- ++(0 pt,5.0pt);
\draw (-1.7,-1)-- ++(-2.5pt,0 pt) -- ++(5.0pt,0 pt) ++(-2.5pt,-2.5pt) -- ++(0 pt,5.0pt);
\draw (-1.7,-2)-- ++(-2.5pt,0 pt) -- ++(5.0pt,0 pt) ++(-2.5pt,-2.5pt) -- ++(0 pt,5.0pt);

\end{scriptsize}

\draw (-0.3,-2.4) node[anchor=north west] {$1$};

\draw (0.,6)-- (0.,-2.4);
\draw [dash pattern=on 2pt off 2pt](0.,4.)-- (0.,5.);
\draw(-0.4,2)node{$0$};
\draw(-0.4,3)node{$1$};
\draw(-0.4,4)node{$2$};
\draw(-0.4,5)node{$3$};

\draw(0.4,1)node{$0$};
\draw(0.4,0)node{$1$};
\draw(0.4,-1)node{$2$};
\draw(0.4,-2)node{$3$};
\begin{scriptsize}
	
\draw (0,2)-- ++(-2.5pt,0 pt) -- ++(5.0pt,0 pt) ++(-2.5pt,-2.5pt) -- ++(0 pt,5.0pt);
\draw (0,3)-- ++(-2.5pt,0 pt) -- ++(5.0pt,0 pt) ++(-2.5pt,-2.5pt) -- ++(0 pt,5.0pt);
\draw (0,4)-- ++(-2.5pt,0 pt) -- ++(5.0pt,0 pt) ++(-2.5pt,-2.5pt) -- ++(0 pt,5.0pt);
\draw (0,5)-- ++(-2.5pt,0 pt) -- ++(5.0pt,0 pt) ++(-2.5pt,-2.5pt) -- ++(0 pt,5.0pt);

\draw (0,1)-- ++(-2.5pt,0 pt) -- ++(5.0pt,0 pt) ++(-2.5pt,-2.5pt) -- ++(0 pt,5.0pt);
\draw (0,0)-- ++(-2.5pt,0 pt) -- ++(5.0pt,0 pt) ++(-2.5pt,-2.5pt) -- ++(0 pt,5.0pt);
\draw (0,-1)-- ++(-2.5pt,0 pt) -- ++(5.0pt,0 pt) ++(-2.5pt,-2.5pt) -- ++(0 pt,5.0pt);
\draw (0,-2)-- ++(-2.5pt,0 pt) -- ++(5.0pt,0 pt) ++(-2.5pt,-2.5pt) -- ++(0 pt,5.0pt);

\end{scriptsize}

\draw (5,-2.4) node[anchor=north west] {$t-1$};

\draw (5.7,6)-- (5.7,-2.4);
\draw [dash pattern=on 2pt off 2pt](5.7,4.)-- (5.7,5.);
\draw(5.3,2)node{$0$};
\draw(5.3,3)node{$1$};
\draw(5.3,4)node{$2$};
\draw(5.3,5)node{$3$};

\draw(6.1,1)node{$0$};
\draw(6.1,0)node{$1$};
\draw(6.1,-1)node{$2$};
\draw(6.1,-2)node{$3$};

\begin{scriptsize}
\draw (5.7,2)-- ++(-2.5pt,0 pt) -- ++(5.0pt,0 pt) ++(-2.5pt,-2.5pt) -- ++(0 pt,5.0pt);
\draw (5.7,3)-- ++(-2.5pt,0 pt) -- ++(5.0pt,0 pt) ++(-2.5pt,-2.5pt) -- ++(0 pt,5.0pt);
\draw (5.7,4)-- ++(-2.5pt,0 pt) -- ++(5.0pt,0 pt) ++(-2.5pt,-2.5pt) -- ++(0 pt,5.0pt);
\draw (5.7,5)-- ++(-2.5pt,0 pt) -- ++(5.0pt,0 pt) ++(-2.5pt,-2.5pt) -- ++(0 pt,5.0pt);

\draw (5.7,1)-- ++(-2.5pt,0 pt) -- ++(5.0pt,0 pt) ++(-2.5pt,-2.5pt) -- ++(0 pt,5.0pt);
\draw (5.7,0)-- ++(-2.5pt,0 pt) -- ++(5.0pt,0 pt) ++(-2.5pt,-2.5pt) -- ++(0 pt,5.0pt);
\draw (5.7,-1)-- ++(-2.5pt,0 pt) -- ++(5.0pt,0 pt) ++(-2.5pt,-2.5pt) -- ++(0 pt,5.0pt);
\draw (5.7,-2)-- ++(-2.5pt,0 pt) -- ++(5.0pt,0 pt) ++(-2.5pt,-2.5pt) -- ++(0 pt,5.0pt);

\end{scriptsize}
\end{tikzpicture}
\end{center}
\smallskip
Let $m$ be an integer. We denote by $q_t(m)$ and $r_t(m)$ the quotient and
the remainder obtained by the euclidean division of $m$ by $t$.  We obtain
a bijection 
$$\varphi_t:\N\longrightarrow \N\times\{0,\ldots,t-1\},\ 
m\longmapsto (q_t(m),r_t(m)).$$
Using $\varphi_t$, we associate to every pointed abacus a pointed
$t$-abacus as follows.
\begin{itemize}
\item  If $x$ is a positive slot labeled by $m$, then we associate to $x$
the $q_t(m)$ positive slot on the runner $r_t(m)$. 
\item  If $x$ is a negative slot labeled by $m$, then we associate to $x$
the $q_t(m)$ negative slot on the runner $t-r_t(m)-1$. 
\end{itemize}
By abuse of notation, we still denote by $\varphi_t(x)$ the image of $x$
on the pointed $t$-abacus described above.
\smallskip

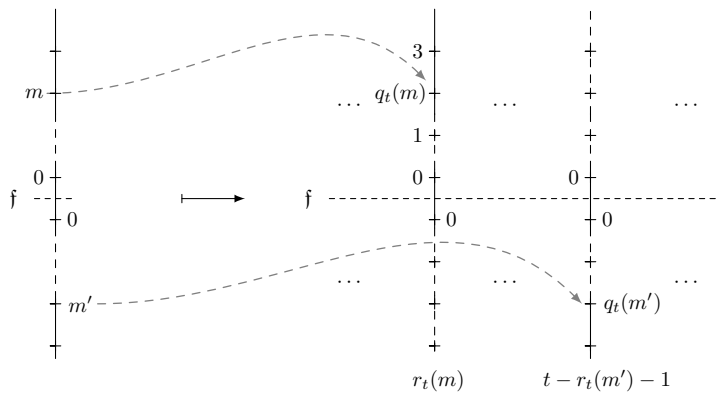
\begin{center}
\hspace{1cm}
\definecolor{sqsqsq}{rgb}{0.12549019607843137,0.12549019607843137,0.12549019607843137}
\begin{tikzpicture}[line cap=round,line join=round,>=triangle
45,x=0.7cm,y=0.7cm, scale=0.8,every node/.style={scale=0.8}]

\draw [dash pattern=on 2pt off 2pt](-2.5,1.5)-- (7,1.5);
\draw(-3,1.5)node{$\mathfrak f$};
\draw (1.2,4) node[anchor=north west] {$\cdots$};
\draw (1.2,-0.2) node[anchor=north west] {$\cdots$};
\draw (-2.5,4) node[anchor=north west] {$\cdots$};
\draw (-2.5,-0.2) node[anchor=north west] {$\cdots$};
\draw (5.5,4) node[anchor=north west] {$\cdots$};
\draw (5.5,-0.2) node[anchor=north west] {$\cdots$};

\draw(-10,1.5)node{$\mathfrak f$};
\draw [dash pattern=on 2pt off 2pt](-9.5,1.5)-- (-8.5,1.5);

\draw(-6,1.4)--(-6,1.6);
\draw[-latex](-6,1.5)--(-4.5,1.5);
\draw (-9,1.5)-- (-9,0.8);
\draw[dashed] (-9,0.8)-- (-9,-0.8);
\draw (-9,-0.8)-- (-9,-2.3);
\draw (-9,1.5)-- (-9,2.2);
\draw[dashed] (-9,2.2)-- (-9,3.6);
\draw (-9,3.5)-- (-9,6);

\draw(-9.4,2)node{$0$};
\draw(-9.5,4)node{$m$};

\draw(-8.6,1)node{$0$};
\draw(-8.4,-1)node{$m'$};
\path[line width=0.5pt,-latex,dashed,gray](-8,-1) edge [out=0] (3.5,-1);
\path[line width=0.5pt,-latex,dashed,gray](-9.2,4) edge [out=0] (-0.2,4.3);

\begin{scriptsize}
\draw (-9,2)-- ++(-2.5pt,0 pt) -- ++(5.0pt,0 pt) ++(-2.5pt,-2.5pt) -- ++(0 pt,5.0pt);
\draw (-9,4)-- ++(-2.5pt,0 pt) -- ++(5.0pt,0 pt) ++(-2.5pt,-2.5pt) -- ++(0 pt,5.0pt);
\draw (-9,5)-- ++(-2.5pt,0 pt) -- ++(5.0pt,0 pt) ++(-2.5pt,-2.5pt) -- ++(0 pt,5.0pt);

\draw (-9,1)-- ++(-2.5pt,0 pt) -- ++(5.0pt,0 pt) ++(-2.5pt,-2.5pt) -- ++(0 pt,5.0pt);
\draw (-9,-1)-- ++(-2.5pt,0 pt) -- ++(5.0pt,0 pt) ++(-2.5pt,-2.5pt) -- ++(0 pt,5.0pt);
\draw (-9,-2)-- ++(-2.5pt,0 pt) -- ++(5.0pt,0 pt) ++(-2.5pt,-2.5pt) -- ++(0 pt,5.0pt);

\end{scriptsize}

\draw (-0.7,-2.4) node[anchor=north west] {$r_t(m)$};

\draw (0,1.5)-- (0,0.8);
\draw[dashed] (0,0.8)-- (0,-0.8);
\draw[dashed] (0,-0.8)-- (0,-2.3);
\draw (0,1.5)-- (0,2.2);
\draw[dashed] (0,2.2)-- (0,3.6);
\draw (0,3.5)-- (0,6);

\draw(-0.4,2)node{$0$};
\draw(-0.4,3)node{$1$};
\draw(-0.8,4)node{$q_t(m)$};
\draw(-0.4,5)node{$3$};

\draw(0.4,1)node{$0$};
\begin{scriptsize}
	
\draw (0,2)-- ++(-2.5pt,0 pt) -- ++(5.0pt,0 pt) ++(-2.5pt,-2.5pt) -- ++(0 pt,5.0pt);
\draw (0,3)-- ++(-2.5pt,0 pt) -- ++(5.0pt,0 pt) ++(-2.5pt,-2.5pt) -- ++(0 pt,5.0pt);
\draw (0,4)-- ++(-2.5pt,0 pt) -- ++(5.0pt,0 pt) ++(-2.5pt,-2.5pt) -- ++(0 pt,5.0pt);
\draw (0,5)-- ++(-2.5pt,0 pt) -- ++(5.0pt,0 pt) ++(-2.5pt,-2.5pt) -- ++(0 pt,5.0pt);

\draw (0,1)-- ++(-2.5pt,0 pt) -- ++(5.0pt,0 pt) ++(-2.5pt,-2.5pt) -- ++(0 pt,5.0pt);
\draw (0,0)-- ++(-2.5pt,0 pt) -- ++(5.0pt,0 pt) ++(-2.5pt,-2.5pt) -- ++(0 pt,5.0pt);
\draw (0,-1)-- ++(-2.5pt,0 pt) -- ++(5.0pt,0 pt) ++(-2.5pt,-2.5pt) -- ++(0 pt,5.0pt);
\draw (0,-2)-- ++(-2.5pt,0 pt) -- ++(5.0pt,0 pt) ++(-2.5pt,-2.5pt) -- ++(0 pt,5.0pt);

\end{scriptsize}

\draw (2.4,-2.4) node[anchor=north west] {$t-r_t(m')-1$};

\draw (3.7,1.5)-- (3.7,0.8);
\draw[dashed] (3.7,0.8)-- (3.7,-0.8);
\draw[] (3.7,-0.8)-- (3.7,-2.3);
\draw (3.7,1.5)-- (3.7,2.2);
\draw[dashed] (3.7,2.2)-- (3.7,3.6);
\draw[dashed] (3.7,3.5)-- (3.7,6);

\draw(3.3,2)node{$0$};

\draw(4.1,1)node{$0$};
\draw(4.7,-1)node{$q_t(m')$};
\begin{scriptsize}

\draw (3.7,2)-- ++(-2.5pt,0 pt) -- ++(5.0pt,0 pt) ++(-2.5pt,-2.5pt) -- ++(0 pt,5.0pt);
\draw (3.7,3)-- ++(-2.5pt,0 pt) -- ++(5.0pt,0 pt) ++(-2.5pt,-2.5pt) -- ++(0 pt,5.0pt);
\draw (3.7,4)-- ++(-2.5pt,0 pt) -- ++(5.0pt,0 pt) ++(-2.5pt,-2.5pt) -- ++(0 pt,5.0pt);
\draw (3.7,5)-- ++(-2.5pt,0 pt) -- ++(5.0pt,0 pt) ++(-2.5pt,-2.5pt) -- ++(0 pt,5.0pt);

\draw (3.7,1)-- ++(-2.5pt,0 pt) -- ++(5.0pt,0 pt) ++(-2.5pt,-2.5pt) -- ++(0 pt,5.0pt);
\draw (3.7,0)-- ++(-2.5pt,0 pt) -- ++(5.0pt,0 pt) ++(-2.5pt,-2.5pt) -- ++(0 pt,5.0pt);
\draw (3.7,-1)-- ++(-2.5pt,0 pt) -- ++(5.0pt,0 pt) ++(-2.5pt,-2.5pt) -- ++(0 pt,5.0pt);
\draw (3.7,-2)-- ++(-2.5pt,0 pt) -- ++(5.0pt,0 pt) ++(-2.5pt,-2.5pt) -- ++(0 pt,5.0pt);

\end{scriptsize}
\end{tikzpicture}
\end{center}

\begin{remark}
If a positive (resp. negative) slot is at position $q$ on the $j$th
runner, then the label of the corresponding positive bead on the pointed
abacus is $tq+j$ (resp. $tq+t-j-1$). 
\end{remark}

Let $\lambda$ be a partition. The \emph{pointed $t$-abacus associated to
$\lambda$} is the $t$-abacus with beads obtained from the pointed
$t$-abacus of $\lambda$ so that the bead at the position $m$ (on the
positive or negative slot of the pointed abacus of $\lambda$) is put on
the corresponding slot described by the previous process.

\begin{example}
\label{ex:pointedtabacus}
Let $t=5$ and $\lambda=(5,5,4,2,1,1)$. The pointed $t$-abacus associated
to $\lambda$ is
\medskip
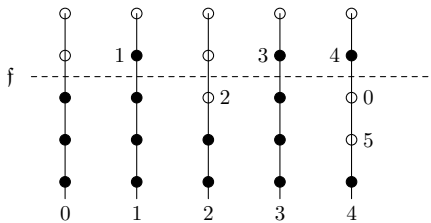
\begin{center}
\definecolor{sqsqsq}{rgb}{0.12549019607843137,0.12549019607843137,0.12549019607843137}
\begin{tikzpicture}[line cap=round,line join=round,>=triangle
45,x=0.7cm,y=0.7cm, scale=0.8,every node/.style={scale=0.8}]

\draw [dash pattern=on 2pt off 2pt](-2.5,1.5)-- (7,1.5);
\draw(-3,1.5)node{$\mathfrak f$};

\draw (-2,-1.4) node[anchor=north west] {$0$};

\draw (-1.7,3)-- (-1.7,-1.4);

\begin{scriptsize}

\draw (-1.7,2) circle (2.5pt);
\draw (-1.7,3) circle (2.5pt);

\draw [fill=black] (-1.7,1) circle (2.5pt);
\draw [fill=black] (-1.7,0) circle (2.5pt);
\draw [fill=black] (-1.7,-1) circle (2.5pt);
\end{scriptsize}

\draw (-0.3,-1.4) node[anchor=north west] {$1$};

\draw (0.,3)-- (0.,-1.4);

\draw(-0.4,2)node{$1$};

\begin{scriptsize}
	
\draw [fill=black] (0,2) circle (2.5pt);
\draw (0,3) circle (2.5pt);

\draw [fill=black](0,1) circle (2.5pt);
\draw [fill=black] (0,0) circle (2.5pt);
\draw [fill=black] (0,-1) circle (2.5pt);

\end{scriptsize}

\draw (1.4,-1.4) node[anchor=north west] {$2$};

\draw (1.7,3)-- (1.7,-1.4);

\draw(2.1,1)node{$2$};
\begin{scriptsize}
	
\draw (1.7,2) circle (2.5pt);
\draw (1.7,3) circle (2.5pt);

\draw (1.7,1) circle (2.5pt);
\draw [fill=black] (1.7,0) circle (2.5pt);
\draw [fill=black] (1.7,-1) circle (2.5pt);

\end{scriptsize}

\draw (3.1,-1.4) node[anchor=north west] {$3$};

\draw (3.4,3)-- (3.4,-1.4);

\draw(3,2)node{$3$};

\begin{scriptsize}
	
\draw [fill=black] (3.4,2) circle (2.5pt);
\draw (3.4,3) circle (2.5pt);

\draw [fill=black] (3.4,1) circle (2.5pt);
\draw [fill=black] (3.4,0) circle (2.5pt);
\draw [fill=black] (3.4,-1) circle (2.5pt);

\end{scriptsize}

\draw (4.8,-1.4) node[anchor=north west] {$4$};

\draw (5.1,3)-- (5.1,-1.4);

\draw(4.7,2)node{$4$};

\draw(5.5,1)node{$0$};
\draw(5.5,0)node{$5$};

\begin{scriptsize}

\draw [fill=black] (5.1,2) circle (2.5pt);
\draw (5.1,3) circle (2.5pt);

\draw (5.1,1) circle (2.5pt);
\draw (5.1,0) circle (2.5pt);
\draw [fill=black] (5.1,-1) circle (2.5pt);

\end{scriptsize}
\end{tikzpicture}
\end{center}
Here the number just to the left of a bead refers to its label in the
pointed abacus of $\lambda$.
\end{example}

Let a pointed $t$-abacus be such that the numbers of positive black beads
and negative white beads are finite. Such an $t$-abacus corresponds to a
partition if and only if these two numbers are equal. In this case, we say
that it is  \emph{a pointed $t$-abacus of a partition}.  In fact, by
construction and following the bijection given before, the positions of
the black beads above $\mathfrak f$ correspond to the value of the arms of
$\lambda$ and the positions of the white beads below $\mathfrak f$ to its
legs. More precisely, for any $0\leq j\leq q-1$, set
\begin{equation}
\label{eq:armj}
\Ar_j^+(\lambda)=\{a\in\Ar^+(\lambda)\mid r_t(a)=j\} \quad \text{and}\quad
\widetilde{\Ar}^+_j(\lambda)=\{q_t(a)\mid a\in\Ar_j^+(\lambda)\},
\end{equation}
and
\begin{equation}
\label{eq:legj}
\Le_j^+(\lambda)=\{l\in\Le^+(\lambda)\mid r_t(l)=t-j-1\} \quad
\text{and}\quad \widetilde{\Le}^+_j(\lambda)=\{q_t(l)\mid l\in\Le_j^+(\lambda)\}.
\end{equation}
We remark that the arms congruent to $j$ modulo $t$ lie on the
$jth$-runner.  Furthermore, if we know the pointed $t$-abacus of a
partition, then we can recover its set of arms with residue $j$ modulo $t$
by
\begin{equation}
\label{eq:armjrecover}
\Ar_j^+(\lambda)=\left\{qt+j\mid q\in\widetilde{\Ar}_j^+(\lambda)\right\}.
\end{equation}
Similarly, the legs on the $jth$-runner are congruent to $t-j-1$ and
\begin{equation}
\label{eq:legjrecover}
\Le_j^+(\lambda)=\left\{qt+t-j-1\mid
q\in\widetilde{\Le}_j^+(\lambda)\right\}.
\end{equation}

Note also that we can read the \emph{same} information on a pointed abacus
of a partition as on its pointed $t$-abacus. In particular, the
observations in Remark~\ref{rk:abacusinfo} are valid.  Furthermore, we
remark that for all nonnegative integers $x$ and $y$, we have $$x\leq y
\quad\text{if and only if}\quad (q_t(x),r_t(x))\leq (q_t(y),r_t(y)),$$
where the order relation in the right member is the lexicographic order.

\begin{remark}
When we look at a pointed $t$-abacus of a partition, however, we observe
that each individual the runner is not necessarily a pointed abacus of a
partition. The runners $1$, $2$, $3$ and $4$ in
Example~\ref{ex:pointedtabacus} are an example of this.
\end{remark}

In order to recover a $t$-tuple of pointed abaci, we now will introduce
the concept of \emph{characteristic vector} of a pointed $t$-abacus.

\medskip

For any abacus, we can produce a new abacus by moving each bead one slot
position up with respect to the fence. We define hence  the \emph{push up}
operation $p^+$. Similarly, we define  the \emph{push down} operation
$p^-$ by moving each bead one slot position down. Let $\delta\in\Z$.  The
\emph{$\delta$-push operation} $p^{\delta}$ on the abacus is obtained by
applying $\delta$ times $p^+$ successively on the abacus if $\delta\geq
0$, and $-\delta$ times $p^-$ otherwise.

\begin{lemma}
\label{lem:pushdown}
Consider a pointed abacus, and put beads in slots so that the numbers of
positive black beads $u$ and of negative white beads $v$ are finite.  Then
$p^{v-u}$ produces a pointed abacus of a partition.
\end{lemma}

\begin{proof}
When we apply $p^+$, if the bead on the negative slot $0$ is white, then
the resulting abacus has $u$ positive black beads and $v-1$ negative white
beads.  If it is black, then the new abacus has $u+1$ positive black beads
and $v$ negative white beads. This is similar for $p^-$, when the bead on
the positive slot $0$ is black, the new abacus has $u-1$ positive black
beads and $v$ negatives white beads, and when the bead is white, there is
$u$ positive black beads and $v+1$ negative white beads in the resulting
abacus. 

Let $d$ be a positive integer. Denote by $\alpha$ and $\beta$ the numbers
of white and black beads on the negative slots labeled from $0$ to $d-1$.
Then the new abacus obtained from the operation $p^c$ has $u+\beta$
positive black beads and $v-\alpha$ negative white beads. In particular,
$\alpha+\beta=d$, and we have $u+\beta=v-\alpha$ if and only if
$\alpha+\beta=v-u$. Hence, for $d=v-u$, the resulting abacus is a pointed
abacus of a partition. Now, when $d$ is negative, $\alpha$ and $\beta$
denote the numbers of white and black beads on the first $d$ positive
slots, and $p^c$ produces an abacus with $u-\beta$ positive black beads
and $v+\alpha$ negative white beads.  Since $|d|=\alpha+\beta$, we obtain
that $u-\beta=v+\alpha$ if and only if $v-u=-\alpha-\beta=-|d|=d$, as
required.
\end{proof}

\begin{lemma}
\label{lem:char}
Let $t$ be a positive integer and $\lambda$ be a partition, and consider
its pointed $t$-abacus. For all $0\leq j\leq t-1$, denote by $u_j$ and
$v_j$ the numbers of positive black beads and negative white beads on the
$j$th runner. Write $c_j=u_j-v_j$. Then 
$$\sum_{j=0}^{t-1}c_j=0.$$  
\end{lemma}

\begin{proof}
The pointed $t$-abacus of $\lambda$ has the same number $x$ of positive
black beads and negative white beads. Write $c^+$ and $c^-$ for the sum of
positive and negative $c_j$'s.  First, we apply $p^{-c_j}$ on the runner
$j$ for all $0\leq j\leq t-1$.  After these operations, each runner is a
pointed abacus of a partition by Lemma~\ref{lem:pushdown}. In particular,
in the resulting pointed $t$-abacus, the numbers  of positive black beads
and negative white beads are equal, and denoted by $y$. Now, denote by
$k^+$ the number of positive black beads that appear after these
operations on all runners with $c_j\leq 0$, and $k^-$ the number of
negative white beads that appear on all runners with $c_j\geq 0$. The
proof of Lemma~\ref{lem:pushdown} implies that $c^+-k^-$ positive black
beads and $|c^-|-k^+$ negative white beads disappear.  Hence, the number
of positive black beads and negative white beads in the new pointed
$t$-abacus are $x+k^+-(c^+-k^-)$ and $x+k^--(|c^-|-k^+)$, and it follows
that
$$x+k^+-(c^+-k^-)=y\quad\text{and}\quad x+k^--(|c^-|-k^+)=y,$$
and $k^++k^--|c^-|=k^++k^--c^+$. Thus, $c^+=-c^-$ and the result follows.
\end{proof}

\begin{definition}
\label{def:charquotient}
Let $t$ be a positive integer and $\lambda$ be a partition.  For all
$0\leq j\leq t-1$, define $c_j\in\Z$ as in Lemma~\ref{lem:char}.  The
tuple $c_t(\lambda)=(c_0,\ldots,c_{t-1})$ is called \emph{the
characteristic vector} of $\lambda$. Furthermore, denote by $\lambda^j$
the partition with pointed abacus obtained by applying $p^{-c_j}$ on the
$j$th runner of the pointed $t$-abacus of $\lambda$. The $t$-tuple of
partitions
$$\lambda^{(t)}=(\lambda^0,\ldots,\lambda^{t-1})\in\mathcal P^t$$ 
is called \emph{the $t$-quotient} of $\lambda$.   
\end{definition}

\begin{remark}
By Remark~\ref{rk:usualabacus}, the $t$-quotient of a partition introduced
in Definition~\ref{def:charquotient} coincides with the ``usual''
$t$-quotient of partitions given in~\cite[2.7.29]{James-Kerber}.
\end{remark}

Let $t$ be a positive integer and $c=(c_0,\ldots,c_{t-1})$ be a $t$-tuple
of integers that sum to $0$. For every pointed $t$-abacus $\mathcal
T=(\mathcal R_0,\ldots,\mathcal R_{t-1})$, we define its \emph{$c$-push}
to be the pointed $t$-abacus
\begin{equation}
\label{eq:push}
p^c(\mathcal T)=(p^{c_0}(\mathcal R_0),\ldots,p^{c_{t-1}}(\mathcal
R_{t-1})),
\end{equation}
obtained by applying $p^{c_j}$ the $c_j$-push to the runner $j$ for all
$0\leq j\leq t-1$.

\begin{example}
\label{ex:carquotient}
We continue Example~\ref{ex:pointedtabacus}. We see that
$c_5(\lambda)=(0,1,-1,1,-1)$. Hence, after applying the
$(0,-1,1,-1,1)$-push to the pointed  $5$-abacus of $\lambda$, we obtain
\medskip
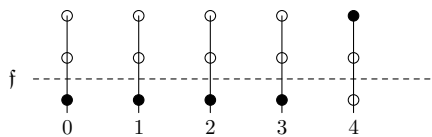
\begin{center}
\definecolor{sqsqsq}{rgb}{0.12549019607843137,0.12549019607843137,0.12549019607843137}
\begin{tikzpicture}[line cap=round,line join=round,>=triangle
45,x=0.7cm,y=0.7cm, scale=0.8,every node/.style={scale=0.8}]

\draw [dash pattern=on 2pt off 2pt](-2.5,1.5)-- (7,1.5);
\draw(-3,1.5)node{$\mathfrak f$};

\draw (-2,0.7) node[anchor=north west] {$0$};

\draw (-1.7,3)-- (-1.7,0.7);

\begin{scriptsize}

\draw (-1.7,2) circle (2.5pt);
\draw (-1.7,3) circle (2.5pt);

\draw [fill=black] (-1.7,1) circle (2.5pt);
\end{scriptsize}

\draw (-0.3,0.7) node[anchor=north west] {$1$};

\draw (0.,3)-- (0.,0.7);

\begin{scriptsize}
	
\draw  (0,2) circle (2.5pt);
\draw (0,3) circle (2.5pt);

\draw [fill=black](0,1) circle (2.5pt);

\end{scriptsize}

\draw (1.4,0.7) node[anchor=north west] {$2$};

\draw (1.7,3)-- (1.7,0.7);

\begin{scriptsize}
	
\draw (1.7,2) circle (2.5pt);
\draw (1.7,3) circle (2.5pt);

\draw [fill=black](1.7,1) circle (2.5pt);

\end{scriptsize}

\draw (3.1,0.7) node[anchor=north west] {$3$};

\draw (3.4,3)-- (3.4,0.7);

\begin{scriptsize}
	
\draw (3.4,2) circle (2.5pt);
\draw (3.4,3) circle (2.5pt);

\draw [fill=black] (3.4,1) circle (2.5pt);

\end{scriptsize}

\draw (4.8,0.7) node[anchor=north west] {$4$};

\draw (5.1,3)-- (5.1,0.7);

\begin{scriptsize}

\draw (5.1,2) circle (2.5pt);
\draw [fill=black](5.1,3) circle (2.5pt);

\draw (5.1,1) circle (2.5pt);

\end{scriptsize}
\end{tikzpicture}
\end{center}
Then $\Fr(\lambda^4)=(0|1)$ and $\lambda^4=(2)$. It follows that the
$5$-quotient of $\lambda$ is
$$\lambda^{(5)}=(\emptyset,\emptyset,\emptyset,\emptyset,(2)).$$
\end{example}

\begin{lemma}
\label{lem:addchar}
Let $t$ be a positive integer and $c=(c_0,\ldots,c_{t-1})$ be a $t$-tuple
of integers that sum to $0$. Then for any pointed $t$-abacus $\mathcal T$
of a partition, $p^c(\mathcal T)$ is also a pointed $t$-abacus of a
partition. 
\end{lemma}

\begin{proof} 
Write $c^+$ and $c^-$ for the sum of positive and negative $c_j$'s. Let
$k^+$ be the number of positive black beads that appear after the push on
all runner with positive $c_j$ and $k^-$ be the number of negative beads
that appear on all runner with negative $c_j$. Then $c^+-k^+$ negative
white beads and $|c^-|-k^-$ positive black beads disappear after the
manipulation.  Since we start with a pointed $t$-abacus of a partition,
the number of positive black beads and negative white beads coincide and
is denoted by $a$. After the $c$-push, the number of positive black beads
is
$$b=a+k^+-(|c^-|-k^-)=a+k^++k^--|c^-|=a+k^++k^-+c^-,$$
and the number of negative white beads is
$$w=a+k^--(c^+-k^+)=a+k^++k^--c^+=a+k^++k^-+c^-$$
since $c^++c^-=0$. Thus, $b=w$, as required.
\end{proof}
\begin{remark}
\label{rk:carquotient}
\noindent
\begin{enumerate}[(i)]
\item 
Write 
\begin{equation}
\label{eq:defct}
\core_t=\left\{(c_0,\ldots,c_{t-1})\mid \sum_{j=0}^{t-1}c_j=0\right\}.
\end{equation}
The application 
\begin{equation}
\label{eq:bijectioncharquo}
\psi_t: \mathcal P\longrightarrow \core_t\times \mathcal P^t,
\quad\lambda\longmapsto(c_t(\lambda),\lambda^{(t)})
\end{equation}
is bijective. We give its inverse map. Let $c=(c_0,\ldots,c_{t-1})\in
\core_t$ and $\underline
\lambda=(\lambda^0,\ldots,\lambda^{t-1})\in\mathcal P^t$ be a $t$-tuple of
partitions. First, we construct a pointed $t$-abacus of partition so that
its $j$th runner is the pointed abacus of $\lambda^j$ for $0\leq j\leq
t-1$. Then we apply the $c_j$-push to each runner $j$. Since
$\sum_{j}c_j=0$, we obtain by Lemma~\ref{lem:addchar} a pointed $t$-abacus
of a partition $\lambda$. By construction, we have
$\psi_t^{-1}=(c,\underline\lambda)\mapsto \lambda$.  \item The
characteristic vector of $\lambda$ can be obtained from it Frobenius
symbol. Indeed, for $0\leq j\leq t-1$, Lemma~\ref{lem:char} gives
\begin{equation}
\label{eq:valchar}
c_j=|\widetilde{\Ar}_j^+(\lambda)|-|\widetilde{\Le}_j^+(\lambda)|,
\end{equation}
where $\widetilde{\Ar}_j^+(\lambda)$ and $\widetilde{\Le}_j^+(\lambda)$
are defined in~(\ref{eq:armj}) and~(\ref{eq:legj}).  
\item Let $0\leq j\leq t-1$.  To recover the parts of the partition
$\lambda^j$ of the $t$-quotient of $\lambda$, we do not have to construct
explicitly its pointed abacus. Indeed the property
Remark~\ref{rk:abacusinfo}(ii) is invariant under the $c$-push. So we can
recover the parts of $\lambda_j$ directly from the $j$th runner of the
original pointed $t$-abacus of $\lambda$.  For example, we see in
Example~\ref{ex:pointedtabacus} that $\lambda^4$ has one part (because it
has only one arm corresponding to the positive black bead and no coarm) of
size $2$ because there are two white beads under the positive black bead.
See also Example~\ref{ex:carquotient}.  
\end{enumerate} 
\end{remark}
\medskip

\subsection{The characteristic vector and $t$-quotient of the conjugate
partition}

For $\mathcal T=(\mathcal R_0,\ldots,\mathcal R_{t-1})$ a pointed
$t$-abacus with runners $\mathcal R_j$ (for $0\leq j\leq t-1$), we define
a pointed $t$-abacus (not necessarily of a partition) by setting
$$\mathcal T^*=(\overline{\mathcal R}_{t-1},\ldots,\overline{\mathcal
R}_0).$$

\begin{lemma}
\label{lem:preservestar}
Let $\mathcal R$ be an abacus (not necessarily pointed) and $c\in\Z$.
Then $$\overline{p^c(\mathcal R)}=p^{-c}(\overline{\mathcal R}).$$
\end{lemma}

\begin{proof}
We give a proof only for $c\geq 0$. The other case is similar.  For a
color $u\in\{\text{black},\text{white}\}$, we write $\overline u$ for its
opposite color.  Let $x\in\N$. Suppose that the bead at the slot labeled
by $x$ has color $u$. We have three cases to consider. 

First, suppose that $x$ labels a positive slot of $\mathcal R$. Then $x$
is sent to $x+c$ by the $c$-push, and it gives a bead of color $\overline
u$ at negative slot $x+c$ after applying manipulations (i) and (ii) of
Remark~\ref{rk:pointedconjugate}.  On the other hand, $x$ is sent to the
negative slot $x$ with bead $\overline u$ by (i) and (ii). Applying
$p^{-c}$, we also obtain a bead with color $\overline u$ at the negative
slot $x+c$.

Assume now that $0\leq x\leq c-1$. After the $c$-push, the new slot is
$c-x-1$ in the positive part. Then (i) and (ii) give a bead of color
$\overline u$ at the negative slot $c-x-1$. On the other hand, the bead on
$x$ is sent to a bead of color $\overline u$ at the positive slot $x$ by
(i) and (ii). Applying $p^{-c}$, we obtain, at the negative slot $c-x-1$,
a bead colored by $\overline u$, as required.

Finally, suppose that $x\leq c$ labels a negative slot. After the
$c$-push, there is a bead colored by $u$ at the negative slot $x-c$, and
(i) and (ii) give a bead with color $\overline u$ at the positive slot
$x-c$.  Now, (i) and (ii) send the initial bead on a bead colored by
$\overline u$ at the positive slot $x$. Applying $p^{-c}$, this bead is
sent at positive slot $x-c$.
\end{proof}

For any characteristic vector $c=(c_0,\ldots,c_{t-1})\in \core_t$ and
$\underline \lambda=(\lambda^0,\ldots,\lambda^{t-1})\in \mathcal P^t$, we
define
\begin{equation}
\label{eq:conjcharquotient} 
c^*=(-c_{t-1},\ldots,-c_0)\quad\text{and}\quad\underline\lambda^*=
\left(\lambda_{t-1}^*,\ldots,\lambda_0^*)\right).
\end{equation}

\begin{proposition}
Let $t$ be a positive integer and $\lambda$ be a conjugate partition with
characteristic vector $c(\lambda)=(c_0,\ldots,c_{t-1})$ and $t$-quotient
$\lambda^{(t)}=(\lambda_0,\ldots,\lambda_{t-1})$. Then
$$\psi_t(\lambda^*)=\left(c(\lambda)^*,(\lambda^{(t)})^*\right).$$
In particular, $c(\lambda^*)=c(\lambda)^*$ and
$(\lambda^*)^{(t)}=(\lambda^{(t)})^*$.
\end{proposition}

\begin{proof}
Recall that if $\Fr(\lambda)=(\Le\mid\Ar)$, then
$\Fr(\lambda^*)=(\Ar\mid\Le)$. Let $0\leq j\leq t-1$, and $at+j\in
\Ar_j^+(\lambda)$. Thus, $at+j\in\Le^+(\lambda^*)$ and by construction of
the pointed $t$-abacus of $\lambda^*$, this leg is sent on runner $t-j-1$
at negative position $a$. Similarly, if $lt+t-j-1\in\Le_j^+(\lambda)$ then
$lt+t-j-1\in\Ar^+(\lambda^*)$ is sent on the runner $t-j-1$ at positive
position $l$. 
Hence, if we denote by $\mathcal T=(\mathcal R_0,\ldots,\mathcal R_{t-1})$
the pointed $t$-abacus of $\lambda$, then $\mathcal T^*$ is the one of
$\lambda^*$. In particular, 
$$|\widetilde{\Le}^+_{j}(\lambda^*)|=|\widetilde{\Ar}^+_{t-j-1}(\lambda)|\quad\text{and}\quad
|\widetilde{\Ar}^+_j(\lambda^*)|=|\widetilde{\Le}^+_{t-j-1}(\lambda)|,$$
and it follows from~(\ref{eq:valchar}) that
$$c_j^*=|\widetilde{\Ar}_j^+(\lambda^*)|-|\widetilde{\Le}_j^+(\lambda^*)|=
|\widetilde{\Le}_{t-j-1}^+(\lambda)|-|\widetilde{\Ar}_{t-j-1}^+(\lambda)|
=-c_{t-j-1},$$
where $c(\lambda^*)=(c_0^*,\ldots,c_{t-1}^*)$.  Denote by $(\mathcal
T_0,\ldots,\mathcal T_{t-1})$ and $(\mathcal S_0,\ldots,\mathcal S_{t-1})$
the pointed $t$-abaci that define the $t$-quotient of $\lambda$ and
$\lambda^*$ respectively (see Definition~\ref{def:charquotient}).  
Hence, we have
$$(\mathcal T_0,\ldots,\mathcal T_{t-1})=(p^{-c_0}(\mathcal
R_0),\ldots,p^{-c_{t-1}}(\mathcal R_{t-1})).$$
Since $c(\lambda^*)=(-c_{t-1},\ldots,-c_0)$ and $\mathcal T^*$ is the
pointed $t$-abacus of $\lambda^*$, applying Lemma~\ref{lem:preservestar},
we deduce that
\begin{align*}
(\mathcal
S_0,\ldots,\mathcal S_{t-1})&=(p^{c_t}(\overline{\mathcal
R}_{t-1}),\ldots,p^{c_0}(\overline{\mathcal R}_0))\\
&=(\overline{p^{-c_t}(\mathcal
R_{t-1})},\ldots,\overline{p^{-c_0}(\mathcal R_0)})\\
&=(\overline{\mathcal T}_{t-1},\ldots,\overline{\mathcal T}_0).
\end{align*}
Since $\mathcal T_j$ is a pointed abacus for all $0\leq j\leq t-1$,
Lemma~\ref{lem:pointedtabaconjugate} gives the result.
\end{proof}

\begin{corollary}
\label{cor:descconjugate}
Let $\lambda$ be a self-conjugate partition with characteristic vector
$c(\lambda)=(c_0,\ldots,c_{t-1})$ and $t$-quotient
$\lambda^{(t)}=(\lambda^0,\ldots,\lambda^{t-1})$. Then
$$c(\lambda)=
\begin{cases}
(c_0,\ldots,c_{(t-3)/2},0,-c_{(t-3)/2},\ldots,-c_0)&\text{if $t$ is odd,}\\
(c_0,\ldots,c_{t/2},-c_{t/2},\ldots,-c_0)&\text{if $t$ is even.}\\
\end{cases}
$$ 
and
$$\lambda^{(t)}=
\begin{cases}
(\lambda^0,\ldots,\lambda^{(t-3)/2},\lambda^{(t-1)/2},(\lambda^{(t-3)/2})^*,\ldots,(\lambda^0)^*)&\text{if $t$ is odd,}\\
(\lambda^0,\ldots,\lambda^{t/2},(\lambda^{t/2})^*,\ldots,(\lambda^0)^*)&\text{if $t$ is even,}\\
\end{cases}
$$ 
where $\lambda^{(t-1)/2}$ is a self-conjugate partition (when $t$ is odd).
\end{corollary}

\begin{example}
\label{ex:sympart}
Let $\lambda=(12,11,6,5^2,3,2^5,1^2)$. Assume $t=5$. We remark that
$\Fr(\lambda)=(11,9,2,1,0\mid 0,1,2,9,11)$, and the pointed $5$-abacus of
$\lambda$ is
\bigskip

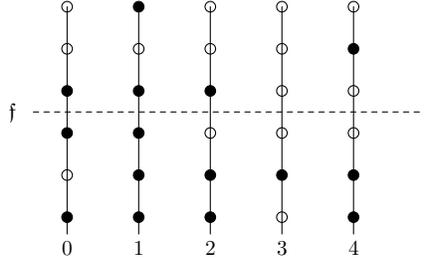
\begin{center}
\definecolor{sqsqsq}{rgb}{0.12549019607843137,0.12549019607843137,0.12549019607843137}
\begin{tikzpicture}[line cap=round,line join=round,>=triangle
45,x=0.7cm,y=0.7cm, scale=0.8,every node/.style={scale=0.8}]

\draw [dash pattern=on 2pt off 2pt](-2.5,1.5)-- (7,1.5);
\draw(-3,1.5)node{$\mathfrak f$};

\draw (-2,-1.4) node[anchor=north west] {$0$};

\draw (-1.7,4)-- (-1.7,-1.4);

\begin{scriptsize}

\draw [fill=black](-1.7,2) circle (2.5pt);
\draw (-1.7,3) circle (2.5pt);
\draw (-1.7,4) circle (2.5pt);

\draw [fill=black] (-1.7,1) circle (2.5pt);
\draw (-1.7,0) circle (2.5pt);
\draw [fill=black] (-1.7,-1) circle (2.5pt);
\end{scriptsize}

\draw (-0.3,-1.4) node[anchor=north west] {$1$};

\draw (0.,4)-- (0.,-1.4);

\begin{scriptsize}
	
\draw [fill=black] (0,2) circle (2.5pt);
\draw (0,3) circle (2.5pt);
\draw [fill=black](0,4) circle (2.5pt);

\draw [fill=black](0,1) circle (2.5pt);
\draw [fill=black] (0,0) circle (2.5pt);
\draw [fill=black] (0,-1) circle (2.5pt);

\end{scriptsize}

\draw (1.4,-1.4) node[anchor=north west] {$2$};

\draw (1.7,4)-- (1.7,-1.4);

\begin{scriptsize}
	
\draw [fill=black](1.7,2) circle (2.5pt);
\draw (1.7,3) circle (2.5pt);
\draw (1.7,4) circle (2.5pt);

\draw (1.7,1) circle (2.5pt);
\draw [fill=black] (1.7,0) circle (2.5pt);
\draw [fill=black] (1.7,-1) circle (2.5pt);

\end{scriptsize}

\draw (3.1,-1.4) node[anchor=north west] {$3$};

\draw (3.4,4)-- (3.4,-1.4);

\begin{scriptsize}
	
\draw (3.4,2) circle (2.5pt);
\draw (3.4,3) circle (2.5pt);
\draw (3.4,4) circle (2.5pt);

\draw (3.4,1) circle (2.5pt);
\draw [fill=black] (3.4,0) circle (2.5pt);
\draw (3.4,-1) circle (2.5pt);

\end{scriptsize}

\draw (4.8,-1.4) node[anchor=north west] {$4$};

\draw (5.1,4)-- (5.1,-1.4);

\begin{scriptsize}

\draw (5.1,2) circle (2.5pt);
\draw [fill=black](5.1,3) circle (2.5pt);
\draw (5.1,4) circle (2.5pt);

\draw (5.1,1) circle (2.5pt);
\draw [fill=black](5.1,0) circle (2.5pt);
\draw [fill=black] (5.1,-1) circle (2.5pt);

\end{scriptsize}
\end{tikzpicture}
\end{center}
Hence, we deduce that
$$c(\lambda)=(0,2,0,-2,0)\quad \text{and}\quad \lambda^{(5)}=((1^2),(1),(1),(1),(2)).$$
\end{example}

\subsection{Characteristic vectors and $t$-core partitions} 
\label{subsec:charcore}
Let $t$ be a positive integer. Let $c=(c_0,\ldots,c_{t-1})\in \core_t$ be
a tuple of integers that sum to $0$; see Remark~\ref{rk:carquotient}. We
associate to $c$ a pointed $t$-abacus by beginning with the empty
partition, that is only black beads below the fence, and only white beads
above, and then putting $c_j$ positive black beads on the runner $j$ if
$c_j>0$ and $|c_j|$ negative white beads otherwise. Since
$\sum_{j=0}^{t-1}c_j=0$, the numbers of positive black beads and negative
white beads are equal. Hence, we obtain a pointed $t$-abacus of a
partition. The partitions obtained by this process are called
\emph{$t$-core partitions}. Denote by $\lambda_c$ the $t$-core partition
corresponding to $c$. If $\lambda$ is a partition with characteristic
vector $c_t(\lambda)$, then $\lambda_{c_t(\lambda)}$, also denoted
$\lambda_{(t)}$, is the \emph{$t$-core of $\lambda$}.  

\begin{example} Assume $t=5$ and $c= (0,1,-1,1,-1)\in
\core_5$. The associated
pointed $5$-abacus is
\medskip
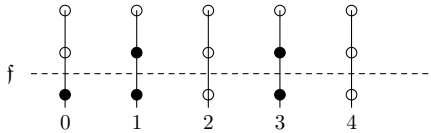
\begin{center}
\definecolor{sqsqsq}{rgb}{0.12549019607843137,0.12549019607843137,0.12549019607843137}
\begin{tikzpicture}[line cap=round,line join=round,>=triangle
45,x=0.7cm,y=0.7cm, scale=0.8,every node/.style={scale=0.8}]

\draw [dash pattern=on 2pt off 2pt](-2.5,1.5)-- (7,1.5);
\draw(-3,1.5)node{$\mathfrak f$};

\draw (-2,0.7) node[anchor=north west] {$0$};

\draw (-1.7,3)-- (-1.7,0.7);

\begin{scriptsize}

\draw (-1.7,2) circle (2.5pt);
\draw (-1.7,3) circle (2.5pt);

\draw [fill=black] (-1.7,1) circle (2.5pt);
\end{scriptsize}

\draw (-0.3,0.7) node[anchor=north west] {$1$};

\draw (0.,3)-- (0.,0.7);

\begin{scriptsize}
	
\draw  (0,2)[fill=black] circle (2.5pt);
\draw (0,3) circle (2.5pt);

\draw [fill=black](0,1) circle (2.5pt);

\end{scriptsize}

\draw (1.4,0.7) node[anchor=north west] {$2$};

\draw (1.7,3)-- (1.7,0.7);

\begin{scriptsize}
	
\draw (1.7,2) circle (2.5pt);
\draw (1.7,3) circle (2.5pt);

\draw (1.7,1) circle (2.5pt);

\end{scriptsize}

\draw (3.1,0.7) node[anchor=north west] {$3$};

\draw (3.4,3)-- (3.4,0.7);

\begin{scriptsize}
	
\draw [fill=black](3.4,2) circle (2.5pt);
\draw (3.4,3) circle (2.5pt);

\draw [fill=black] (3.4,1) circle (2.5pt);

\end{scriptsize}

\draw (4.8,0.7) node[anchor=north west] {$4$};

\draw (5.1,3)-- (5.1,0.7);

\begin{scriptsize}

\draw (5.1,2) circle (2.5pt);
\draw (5.1,3) circle (2.5pt);

\draw (5.1,1) circle (2.5pt);

\end{scriptsize}
\end{tikzpicture}
\end{center}
The corresponding $5$-core has Frobenius symbol $\Fr(\lambda_c)=(2,0\mid
1,3)$, hence $\lambda_c=(4,3,1)$.
\end{example}

\begin{remark}
\label{rk:usualcore}
Denote by $\mathcal T=(\mathcal R_0,\ldots,\mathcal R_{t-1})$ and
$\mathcal T'=(\mathcal R_0',\ldots,\mathcal R_{t-1}')$ the pointed
$t$-abacus of $\lambda$ and $\lambda_{(t)}$.  We can deduce $\mathcal T'$
from $\mathcal T$ as follows.  For all $0\leq j\leq t-1$, we look at the
first white bead (by reading from bottom to top) on the runner $\mathcal
R_j$. If there is a black bead above, then we exchange the two beads.  We
repeat the procedure successively until it is no longer possible. By
definition of the characteristic vector of $\lambda$, we obtain $\mathcal
T'$ at the end. This shows that $\lambda_{(t)}$ is the $t$-core of the
partition $\lambda$ defined in~\cite[\S2.7]{James-Kerber}.
\end{remark}
\begin{remark}
Denote by $\mathcal{C}_t$ the set of $t$-core partitions. The bijection
$\psi_t$ given in Remark~\ref{rk:abacusinfo} induces a bijection
\begin{equation}
\label{eq:partcorequotient}
\Psi_t:\mathcal P\longrightarrow \mathcal{C}_t\times \mathcal P^t,\
\lambda\longmapsto (\lambda_{(t)},\lambda^{(t)}).
\end{equation}
\end{remark}
\smallskip

\begin{remark} Let $\mathcal T=(\mathcal R_0,\ldots,\mathcal R_{t-1})$ be
the pointed $t$-abacus of $\lambda$. Denote by $\mathcal R$ its pointed
abacus.  By Remark~\ref{rk:abacusinfo}(iv), exchanging a pair of
consecutive black and white beads (denoted by $b$ and $w$ and satisfying
$b$ over $w$) on the runner $\mathcal R_j$ removes a hook of size $1$
(that is a box) of $\lambda^{j}$. On the other hand, on $\mathcal T$,
there are $j$ slots at the right of $b$ and $t-j-1$ slots at the right of
$w$. This means that there are (strictly) $j+t-j-1=t-1$ slots between
$\varphi^{-1}_t(b)$ and $\varphi^{-1}_t(w)$ on $\mathcal R$. Hence,
Remark~\ref{rk:abacusinfo}(iv) again gives that the hooklength associated
to $(b,w)$ is $t$.  Now, consider the opposite procedure of
Remark~\ref{rk:usualcore},  that allow us to obtain $\mathcal T$ from the
pointed $t$-abacus of $\lambda_{(t)}$ by exchanging a black bead $b$ and a
white bead $w$ with $w$ over $b$. We can assume, without loss of
generality, that we exchange only consecutive beads. Therefore,
Remark~\ref{rk:removehooks} shows that at each step, the new partition has
$t$ boxes more.  Beginning with the pointed $t$-abacus of $\lambda_{(t)}$,
we obtain $\mathcal T$ after $|\lambda^0|+\cdots+|\lambda^{t-1}|$ steps.
In particular,
$$|\lambda|=|\lambda_{(t)}|+t\sum_{j=0}^{t-1}|\lambda_j|.$$  
\end{remark}

\begin{lemma}
\label{lem:armlegcore}
Let $c\in \core_t$ and $0\leq j\leq t-1$.
\begin{enumerate}[(i)] 
\item If $c_j\geq 0$, then $\Le^+_j(\lambda_c)=\emptyset$ and
$$
\Ar_j^+(\lambda_c)=\{qt+j\mid q\in C_j\},$$ where $C_j=\{0,\ldots,c_j-1\}$.
\item If $c_j\leq 0$, then $\Ar^+_j(\lambda_c)=\emptyset$ and
$$\
\Le_j^+(\lambda_c)=\{qt+t-j-1\mid q\in C_j\},$$
where $C_j=\{0,\ldots,|c_j|-1\}$.
\end{enumerate}
\end{lemma}

\begin{proof}
Let $0\leq j\leq t-1$. By construction of $\lambda_c$, we have
$\widetilde{\Ar}_j^+(\lambda_c)=C_j$ and
$\widetilde{\Le}_j^+(\lambda_c)=\emptyset$ if $c_j\geq 0$, and
$\widetilde{\Le}_j^+(\lambda_c)=C_j$ and
$\widetilde{\Ar}_j^+(\lambda_c)=\emptyset$ if $c_j\leq 0$.
The result then follows from~(\ref{eq:armjrecover}) and~(\ref{eq:legjrecover}).
\end{proof}

\begin{corollary}
\label{cor:sizecore}
Let $t$ be a positive integer, and $c\in\mathcal C_t$. Then
$$|\lambda_c|=\frac t 2\sum_{j=0}^{t-1}c_j^2 +\sum_{j=0}^{t-1}jc_j.$$
\end{corollary}

\begin{proof}
Write $C^+$ and $C^-$ for the sets of $0\leq j\leq t-1$ such that $c_j$ is
positive or negative. By Lemma~\ref{lem:armlegcore}, the Durfee number of
$\lambda_c$ is 
\begin{equation}
\label{eq:duftycore1}
s=\sum_{j\in C^+}c_j=\sum_{j\in C^-}|c_j|.
\end{equation}
In particular, $2s=\sum_{j\in C^+}c_j+\sum_{j\in C^-}|c_j|$ hence
\begin{equation}
\label{eq:duftycore2}
s=\frac{1}{2}\left(\sum_{j\in C^+}c_j+\sum_{j\in C^-}|c_j|\right)=\frac 1 2\sum_{j=0}^{t-1}|c_j|. 
\end{equation}
Furthermore, Lemma~\ref{lem:armlegcore} gives
\begin{align}
\label{eq:armcore}
\sum_{a\in\Ar^+(\lambda_c)}a&=\sum_{j=0}^{t-1}\sum_{a\in\Ar^+_j(\lambda_c)}a
=\sum_{j\in C^+}\sum_{a\in \Ar^+_j(\lambda_c)}a\\\nonumber
&=\sum_{j\in C^+}\sum_{q=0}^{c_j-1}(qt+j)\\\nonumber
&=\sum_{j\in C^+}\left(jc_j +\frac{c_j(c_j-1)}{2}t\right).
\end{align}
Similarly,
\begin{equation}
\label{eq:legcore}
\sum_{l\in\Le^-(\lambda_c)}l=\sum_{j\in C^-}\left((t-j-1)|c_j|+
t\frac{|c_j|(|c_j|-1)}{2}\right).
\end{equation}
Now, by~(\ref{eq:duftycore1}), (\ref{eq:duftycore2}), (\ref{eq:armcore}),
(\ref{eq:legcore}) and (\ref{eq:sizebyarmleg}), we obtain
\begin{align*}
|\lambda_c|&=s+t\sum_{j=0}^{t-1}\frac{|c_j|(|c_j|-1|)}2+\sum_{j\in C^+}jc_j+\sum_{j\in C^-}-j|c_j|+(t-1)\sum_{j\in C^-}|c_j|\\
&=s+t\sum_{j=0}^{t-1}\frac{|c_j|(|c_j|-1|)}2+\sum_{j=0}^{s-1}jc_j+(t-1)s\\
&=t\left(\sum_{j=0}^{t-1}\frac 1 2(|c_j|^2-|c_j|)+s\right)+\sum_{j=0}^{t-1}jc_j\\
&=t\left(\frac 1 2\sum_{j=0}^{t-1}|c_j|^2-\frac 1
2\sum_{t=0}^{t-1}|c_j|+\frac 1
2\sum_{j=0}^{t-1}|c_j|\right)+\sum_{j=0}^{t-1}jc_j\\
&=\frac t 2\sum_{j=0}^{t-1}c_j^2+\sum_{j=0}^{t-1}jc_j,\\
\end{align*}
as required.
\end{proof}

\begin{corollary}
\label{cor:sizesym}
Let $t$ be a positive integer.
\begin{enumerate}[(i)]
\item Assume $t=2m+1$. Let $c=(-c_m,\ldots,-c_{1},0,c_{1},\ldots,c_m)\in
SC_{t,n}$. Then
$$n=t\sum_{j=1}^mc_j^2+\sum_{j=1}^m 2j  c_j.$$
\item Assume $t=2m$. Let $c=(-c_m,\ldots,-c_{1},c_{1},\ldots,c_m)\in
SC_{t,n}$. Then
$$n=t\sum_{j=1}^mc_j^2+\sum_{j=1}^m (2j-1) c_j.$$
\end{enumerate}
\end{corollary}

\begin{proof}
Assume $t=2m+1$. By Corollary~\ref{cor:sizecore}, we have
\begin{eqnarray*}
n&=&t \sum_{j=1}^mc_j^2
+\sum_{j=1}^{m-1}j(-c_{m-j})+\sum_{j=1}^m(m+j)c_j,\\
&=&t\sum_{j=1}^mc_j^2+\sum_{j=1}^{m-1}(j-m)c_j+\sum_{j=1}^{m-1}(m+j)c_j+2mc_m,\\
&=&t\sum_{j=1}^m c_j^2+2\sum_{j=1}^{m-1}jc_j+2mc_m,
\end{eqnarray*}
as required. Assume now $t=2m$. Similarly, 
\begin{eqnarray*}
n&=&t \sum_{j=1}^mc_j^2
+\sum_{j=1}^{m-1}j(-c_{m-j})+\sum_{j=1}^m(m+j-1)c_j,\\
&=&t\sum_{j=1}^mc_j^2+\sum_{j=1}^{m-1}(j-m)c_j+(m-m)c_m+\sum_{j=1}^{m}(m+j-1)c_j,\\
&=&t\sum_{j=1}^m c_j^2+2\sum_{j=1}^{m}(2j-m+m-1)c_j,
\end{eqnarray*}
and the result follows.
\end{proof}

\subsection{Moving between the Frobenius symbols of a partition and its
$t$-core and $t$-quotient}
\label{subsec:main}
In this section, we show first, in Proposition \ref{prop:mainsensdirect},
how to obtain the Frobenius symbols of the quotient from the
characteristic vector and the Frobenius symbol of the partition.
Afterwards, we arrive at main result, Theorem \ref{thm:main}, which shows
the reverse direction; that is how to obtain the Frobenius symbol of a
partition $\lambda$ from the characteristic vector and the Frobenius
symbols of the $t$-quotient. We begin with some notation.

Let $\lambda$ be a partition and $t$ be a positive integer.  We set 
\begin{equation}
\label{eq:virtualarm}
\Arv^-(\lambda)=\Ar^-(\lambda)\cup \{k\in\N\mid
k>\max(\Le^+(\lambda))\},
\end{equation}
and
\begin{equation}
\label{eq:virtualleg}
\Lev^-(\lambda)=\Le^-(\lambda)\cup \{k\in\N\mid
k>\max(\Ar^+(\lambda))\}.
\end{equation}
Furthermore, we define
\begin{equation}
\label{eq:coarmj}
\Ar_j^-(\lambda)=\{a\in\Arv^-(\lambda)\mid r_t(a)=t-j-1\} \ \text{and}\
\widetilde{\Ar}^-_j(\lambda)=\{q_t(a)\mid a\in\Ar_j^-(\lambda)\},
\end{equation}
and
\begin{equation}
\label{eq:colegj}
\Le_j^-(\lambda)=\{l\in\Lev^-(\lambda)\mid r_t(l)=j\} \quad
\text{and}\quad \widetilde{\Le}^-_j(\lambda)=\{q_t(l)\mid l\in\Le_j^-(\lambda)\}.
\end{equation}

\begin{remark}
Note that, for $\varepsilon\in\{-,+\}$, we have
$\widetilde{\Ar}_j^\varepsilon\cup \widetilde{\Le}_j^{-\varepsilon}=\N.$ 
\end{remark}
We denote by $\lambda^{(t)}=(\lambda^0,\ldots,\lambda^{t-1})$ the
$t$-quotient of $\lambda$, and for $0\leq j\leq t-1$, we write 
$$\Fr(\Le_j|\Ar_j),$$
where $\Le_j=\Le^+(\lambda^j)$ and $\Ar_j=\Ar^+(\lambda^j)$ are the sets
of legs and arms of $\lambda^j$.
\medskip
\begin{proposition}
\label{prop:mainsensdirect}
Let $c_t(\lambda)=(c_0,\ldots,c_{t-1})\in \core_t$ be the characteristic
vector of $\lambda$, and $0\leq j\leq t-1$. 
\begin{enumerate}[(i)]
\item If $c_j\geq 0$, then
$$\Ar_j= \{a-c_j\mid a\in \widetilde{\Ar}_j^+(\lambda)\quad
\text{and}\quad
a\geq c_j\},$$
and
$$\Le_j=\{l+c_j\mid l\in\widetilde{\Le}_j^+(\lambda)\}\,\cup\, \{c_j-l-1\mid
l\in\widetilde{\Le}_j^-(\lambda)\ \text{such that}\ l<c_j\}.$$
\item If $c_j\leq 0$, then
$$\Ar_j=\{a+|c_j|\mid a\in\widetilde{\Ar}_j^+(\lambda)\}\,\cup\,
\{|c_j|-a-1\mid
a\in\widetilde{\Ar}_j^-(\lambda)\ \text{such that}\ a<|c_j|\},$$
and
$$\Le_j= \{l-|c_j|\mid l\in \widetilde{\Le}_j^+(\lambda)\quad
\text{and}\quad l\geq |c_j|\}.$$
\end{enumerate}
\end{proposition}

\begin{proof}
Assume first that $c_j\geq 0$. By Definition~\ref{def:charquotient}, the
pointed abacus of the partition $\lambda^j$ is obtained by applying
$p^{-c_j}$ to the $j$th-runner of the $t$-pointed abacus of $\lambda$. 
Now, let $a\in \widetilde{\Ar}_j^+(\lambda)$. Then by the construction of
the $t$-pointed abacus $\mathcal T$ of $\lambda$ given
in~\S\ref{subsec:quotient}, the bead on the slot labeled by $a$ on the
$j$th runner of $\mathcal T$ is black.  After applying $p^{-c_j}$, this
bead lies on the positive slot labeled by $a-c_j$ if $a\geq c_j$ or on the
negative slot labeled by $c_j-a-1$ if $a<c_j$. In particular, $a$
contributes to an arm of $\lambda^j$ if and only if $a\geq c_j$. In this
case, $a-c_j\in\Ar_j$. Furthermore, to any $l\in \widetilde{\Le}_{j}^+$
there is a white bead on the negative slot $l$ of the $j$th runner of
$\mathcal T$. After applying $p^{-c_j}$, we obtain a negative white bead
of the pointed abacus of $\lambda^j$ on the slot labeled by $c_j+l$. Hence
$c_j+l\in \Le_j$. Finally, by the same argument as above, each coleg
$l\in\widetilde{\Le}_j^-$ such that $l<c_j$ give a leg of $\lambda^j$ on
the negative slot $c_j-l-1$. It follows that
$$\Le_j=\{c_j+l\mid l\in\widetilde{\Le}_j^+(\lambda)\}\ \cup\
\{c_j-l-1\mid l\in \widetilde{\Le}_j^-(\lambda)\ \text{and }l<c_j\}.$$
The argument is similar when $c_j\leq 0$.
\end{proof}

\begin{remark}
The proof of Proposition~\ref{prop:mainsensdirect} can be interpreted
graphically as follows.
\medskip
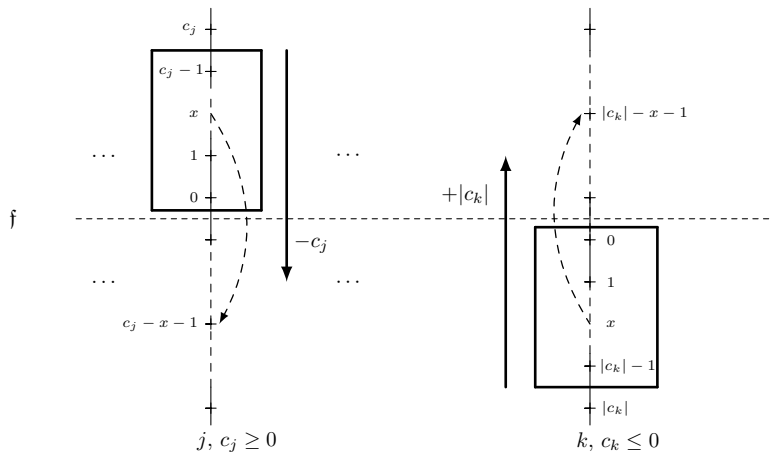
\begin{center}
\definecolor{sqsqsq}{rgb}{0.12549019607843137,0.12549019607843137,0.12549019607843137}
\begin{tikzpicture}[line cap=round,line join=round,>=triangle
45,x=0.7cm,y=0.7cm, scale=0.8,every node/.style={scale=0.8}]

\draw [dash pattern=on 2pt off 2pt](-1.5,1.5)-- (15,1.5);
\draw(-3,1.5)node{$\mathfrak f$};

\draw [line width=1.pt,latex-](3.5,0) --(3.5,5.5);
\draw (3.5,1.3)  node[anchor=north
west] {$-c_{j}$}; 
\draw [line width=1.pt](0.3,1.7)--(2.9,1.7);
\draw [line width=1.pt](0.3,1.7)--(0.3,5.5);
\draw [line width=1.pt](0.3,5.5)--(2.9,5.5);
\draw [line width=1.pt](2.9,5.5)--(2.9,1.7);

\draw (-0.8,0)node{$\dots$};
\draw (-0.8,3)node{$\dots$};
\draw (1.2,-3.4) node[anchor=north west] {$j$,\ $c_j\geq 0$};
\draw (1.7,-2.5)-- (1.7,-3.4);
\draw[dashed] (1.7,-2.5)-- (1.7,1);
\draw (1.7,1)-- (1.7,3.4);
\draw[dashed] (1.7,3.4)-- (1.7,4.5);
\draw (1.7,4.6)-- (1.7,6.5);
\begin{scriptsize}

\draw(1.3,2)node{$0$};
\draw(1.3,3)node{$1$};
\draw(1.3,4)node{$x$};
\draw(1,5)node{$c_j-1$};
\draw(1.2,6)node{$c_j$};
\draw(0.5,-1)node{$c_j-x-1$};
\path[line width=0.5pt,-latex,dashed](1.7,4) edge [bend left=30] (1.9,-1);

\draw (1.7,5)-- ++(-2.5pt,0 pt) -- ++(5.0pt,0 pt) ++(-2.5pt,-2.5pt) -- ++(0 pt,5.0pt);
\draw (1.7,6)-- ++(-2.5pt,0 pt) -- ++(5.0pt,0 pt) ++(-2.5pt,-2.5pt) -- ++(0 pt,5.0pt);
\draw (1.7,2)-- ++(-2.5pt,0 pt) -- ++(5.0pt,0 pt) ++(-2.5pt,-2.5pt) -- ++(0 pt,5.0pt);
\draw (1.7,3)-- ++(-2.5pt,0 pt) -- ++(5.0pt,0 pt) ++(-2.5pt,-2.5pt) -- ++(0 pt,5.0pt);
\draw (1.7,1)-- ++(-2.5pt,0 pt) -- ++(5.0pt,0 pt) ++(-2.5pt,-2.5pt) -- ++(0 pt,5.0pt);
\draw  (1.7,-1)-- ++(-2.5pt,0 pt) -- ++(5.0pt,0 pt) ++(-2.5pt,-2.5pt) -- ++(0 pt,5.0pt);
\draw  (1.7,-3)-- ++(-2.5pt,0 pt) -- ++(5.0pt,0 pt) ++(-2.5pt,-2.5pt) -- ++(0 pt,5.0pt);
\end{scriptsize}

\draw (5,0) node{$\cdots$};
\draw (5,3) node{$\cdots$};
\draw (10.2,-3.4) node[anchor=north west] {$k$,\ $c_k\leq 0$};
\draw (10.7,-1.7)-- (10.7,-3.4);
\draw[dashed] (10.7,-1.7)-- (10.7,0);
\draw[dashed] (10.7,1.2)-- (10.7,5.5);
\draw (10.7,0)-- (10.7,2);
\draw (10.7,5.5)-- (10.7,6.5);
\draw [line width=1.pt,-latex](8.7,-2.5) --(8.7,3);
\draw (7,2.5)  node[anchor=north
west] {$+|c_{k}|$}; 
\draw [line width=1.pt](9.4,1.3)--(12.3,1.3);
\draw [line width=1.pt](12.3,1.3)--(12.3,-2.5);
\draw [line width=1.pt](12.3,-2.5)--(9.4,-2.5);
\draw [line width=1.pt](9.4,-2.5)--(9.4,1.3);

\begin{scriptsize}

\draw(11.2,1)node{$0$};
\draw(11.2,0)node{$1$};
\draw(11.2,-1)node{$x$};
\draw(11.6,-2)node{$|c_k|-1$};
\draw(11.3,-3)node{$|c_k|$};
\draw(12,4)node{$|c_k|-x-1$};
\path[line width=0.5pt,latex-,dashed](10.5,4) edge [bend right=30] (10.7,-1);
\draw (10.7,4)-- ++(-2.5pt,0 pt) -- ++(5.0pt,0 pt) ++(-2.5pt,-2.5pt) -- ++(0 pt,5.0pt);

\draw (10.7,6)-- ++(-2.5pt,0 pt) -- ++(5.0pt,0 pt) ++(-2.5pt,-2.5pt) -- ++(0 pt,5.0pt);
\draw (10.7,2)-- ++(-2.5pt,0 pt) -- ++(5.0pt,0 pt) ++(-2.5pt,-2.5pt) -- ++(0 pt,5.0pt);
\draw (10.7,1)-- ++(-2.5pt,0 pt) -- ++(5.0pt,0 pt) ++(-2.5pt,-2.5pt) -- ++(0 pt,5.0pt);
\draw  (10.7,0)-- ++(-2.5pt,0 pt) -- ++(5.0pt,0 pt) ++(-2.5pt,-2.5pt) -- ++(0 pt,5.0pt);
\draw  (10.7,-2)-- ++(-2.5pt,0 pt) -- ++(5.0pt,0 pt) ++(-2.5pt,-2.5pt) -- ++(0 pt,5.0pt);
\draw  (10.7,-3)-- ++(-2.5pt,0 pt) -- ++(5.0pt,0 pt) ++(-2.5pt,-2.5pt) -- ++(0 pt,5.0pt);
\end{scriptsize}
\end{tikzpicture}
\end{center}
If $c_j\geq 0$, then any black beads (resp. white beads) in the windows
give coarms (resp. legs) of $\lambda^j$ after applying $p^{-c_j}$. If
$c_j\leq 0$, then any white beads (resp. black beads) in the windows give
colegs (resp.  arms) of $\lambda^j$ after applying the $|c_j|$-push.
\end{remark}

\begin{example}
We consider again $\lambda=(5,5,4,2,1,1)$ and $t=5$. By
Example~\ref{ex:example2}, we have
$\Le_0^+(\lambda)=\Le_1^+(\lambda)=\Le_3^+(\lambda)=\emptyset,
\,\Le_2^+(\lambda)=\{2\},\, \Le_4^+(\lambda)=\{0,5\}$ and
$\Ar_0^+(\lambda)=\Ar_2^+(\lambda)=\emptyset$, $\Ar_1^+(\lambda)=\{1\}$,
$\Ar_3^+(\lambda)=\{3\}$, and $\Ar_4^+(\lambda)=\{4\}$. Hence,
$$\widetilde{\Le}_4^+(\lambda)=\{0,1\},\,\widetilde{\Le}_2^+(\lambda)=\{0\},
\
\widetilde{\Ar}^+_1(\lambda)=\widetilde{\Ar}^+_3(\lambda)
=\widetilde{\Ar}^+_4(\lambda)=\{0\}.$$
By Remark Equation~\ref{rk:abacusinfo}(ii), we obtain $c_0=0$, $c_1=1$,
$c_2=-1$, $c_3=1$, and $c_4=-1$.  Note that we recover
Example~\ref{ex:carquotient}.  Now, Proposition~\ref{prop:mainsensdirect}
gives
$$\Ar_0=\Ar_1=\Ar_2=\Ar_3=\emptyset,\,\Ar_4=\{1\}\ \text{and}\
\Le_0=\Le_1=\Le_2=\Le_3=\emptyset,\,\Le_4=\{0\},$$ and
$\lambda^{(5)}=(\emptyset,\emptyset,\emptyset,\emptyset,(2))$.
Furthermore, $\Fr(\lambda_5)=(2,0|1,3)$, that is $\lambda_{5}=(4,3,1)$. 
\end{example}

\begin{remark}
We do not need a graphical construction of the pointed $t$-abacus of
$\lambda$ to determine the $t$-quotient of $\lambda$. The process is
completely numerical. 
\end{remark}
Now we arrive at the main result of this section.
\begin{theorem}
\label{thm:main} 
Let $t$ be a positive integer. Let $\underline c=(c_0,\ldots,c_{t-1})\in
\core_t$ be a characteristic vector and a $t$-multipartition
$\underline\lambda=(\lambda^0,\ldots,\lambda^{t-1})$.  For $0\leq j\leq
t-1$, denote by $\Fr(\lambda^j)=(\Le_j\mid\Ar_j)$ the Frobenius symbol of
$\lambda^j$, and define $L_j$ and $A_j$ as follows.
\begin{enumerate}[(i)]
\item If $c_j\geq 0$, then
$$L_j=\{(l-c_j)t+t-j-1\mid l\in\Le_j\ \text{and}\ l\geq c_j\}$$
and
$$A_j=\{(a+c_j)t+j\mid a\in\Ar_j\}\cup\{(c_j-a-1)t+j\mid a\in \Arv_j^-\ \text{and
}a<c_j\},$$
where $\Arv_j^-=\Arv_j^-(\lambda^j)$ with the convention that
$\Arv_j^-=\N$ is $\lambda^j=\emptyset$. 
\item If $c_j\leq 0$, then
$$L_j=\{(l+|c_j|)t+t-j-1\mid l\in\Le_j\}\cup \{(|c_j|-l-1)t+t-j-1\mid
l\in \Lev_j^-\ \text{and
}l<|c_j|\},$$
where $\Lev_j^-=\Lev^-_j(\lambda^j)$ with the convention that
$\Arv_j^-=\N$ is $\lambda^j=\emptyset$, 
and
$$A_j=\{(a-|c_j|)t+j\mid a\in\Ar_j\ \text{and
}a\geq |c_j|\}.$$
\end{enumerate}
Set
\begin{equation}
\label{eq:setarmleg}
\Le=\bigcup_{j=0}^{t-1}L_j\quad\text{and}\quad\Ar=\bigcup_{j=0}^{t-1}A_j,
\end{equation}
and consider the partition $\lambda$ with Frobenius symbol
$(\Le\mid\Ar)$.  
Then we have
$$\psi_t(\lambda)=(\underline c,\,\underline \lambda).$$
\end{theorem}

\begin{proof}
First, we will show that the partition $\lambda$ is well-defined. We
remark that the sets $\Le_j$ (resp. the sets $\Ar_j$) are disjoint.  Let
$0\leq j\leq t-1$. If $c_j\geq 0$, then the definitions of $L_j$ and $A_j$
give
\begin{equation}
\label{eq:preuvemain1}
|\Le_j|=|L_j|+\alpha_j\quad\text{and}\quad |\Ar_j|=|A_j|-\beta_j, 
\end{equation}
where $\alpha_j=|\{l\in\Le_j\mid l<c_j\}|$ and
$\beta_j=|\{a\in\Arv_j^-\mid a<c_j\}|$.
Similarly, if $c_j\leq 0$, then
\begin{equation}
\label{eq:preuvemain2}
|\Le_j|=|L_j|-\alpha_j\quad\text{and}\quad |\Ar_j|=|A_j|+\beta_j, 
\end{equation}
where $\alpha_j=|\{l\in\Lev_j^-\mid l<|c_j|\}|$ and
$\beta_j=|\{a\in\Ar_j\mid a<|c_j|\}|$.
Furthermore, for any $0\leq j\leq t-1$
\begin{equation}
\label{eq:preuvemain3}
\alpha_j+\beta_j=|c_j|.
\end{equation}
Write $J^+$ (resp. $J^-$) for the set of $0\leq j\leq t-1$ such that
$c_j\geq 0$ (resp. $c_j<0$). By definition of
$\Le$,~(\ref{eq:preuvemain1}) and~(\ref{eq:preuvemain2}) we have
\begin{equation}
\label{eq:preuvemain4}
|\Le|=\sum_{j\in J^+}(|\Le_j|-\alpha_j)+\sum_{j\in
J^-}(|\Le_j|+\alpha_j)=\sum_{j=0}^{t-1}|\Le_j|-\sum_{j\in
J^+}\alpha_j+\sum_{j\in J^-}\alpha_j.
\end{equation}
Similarly,
\begin{equation}
\label{eq:preuvemain5}
|\Ar|=\sum_{j\in J^+}(|\Ar_j|+\beta_j)+\sum_{j\in
J^-}(|\Ar_j|-\beta_j)=\sum_{j=0}^{t-1}|\Ar_j|+\sum_{j\in
J^+}\beta_j-\sum_{j\in J^-}\beta_j.
\end{equation}
Now Equations (\ref{eq:preuvemain3}), (\ref{eq:preuvemain4}) and
(\ref{eq:preuvemain5}) and the fact that $|\Le_j|=|\Ar_j|$ (for any
$0\leq j\leq t-1$) imply
$$
|\Ar|-|\Le|=\sum_{j=0}^{t-1}(\underbrace{|\Ar_j|-|\Le_j|}_{=0})+
\sum_{j\in J^+}(\underbrace{\beta_j+\alpha_j}_{=c_j})-\sum_{j\in
J^-}(\underbrace{\beta_j+\alpha_j}_{=-c_j})=\sum_{j=0}^{t-1}c_j=0,$$
because $\underline c$ is a characteristic vector.
Hence, $(\Le\mid\Ar)$ defines a
partition $\lambda$ by Remark~\ref{rk:armlegbijpart}. Write
$({\lambda'}^0,\ldots,{\lambda'}^{t-1})$ for the $t$-quotient of
$\lambda$, and $c_t(\lambda)=(c'_0,\ldots,c'_{t-1})$.
For $0\leq j\leq t-1$, we set $\Fr({\lambda'}^j)=(\Le'_j\mid\Ar'_j)$.
First assume that $j\in J^+$.
By Remark~\ref{rk:carquotient}(ii),
$c'_j=|\widetilde{\Ar}_j^+(\lambda)|-|\widetilde{\Le}_j^+(\lambda)|$.
Furthermore, by definition of $\lambda$, we have 
\begin{equation}
\label{eq:compatibilite1}
\Ar_j^+(\lambda)=A_j\quad \text{and}\quad
\Le_j^+(\lambda)=L_j,
\end{equation}
hence
$|\widetilde{\Ar}_j^+(\lambda)|=|A_j|$ and
$|\widetilde{\Le}_j^+(\lambda)|=|L_j|$. 
Since $|\Ar_j|=|\Le_j|$ and by Equations~(\ref{eq:preuvemain1}),
(\ref{eq:preuvemain2}) and (\ref{eq:preuvemain3}), we obtain
$$c'_j=|A_j|-|L_j|=|\Ar_j|+\beta_j-(|\Le_j|-\alpha_j)=\beta_j+\alpha_j=c_j.$$
Now, using that $\Ar_j^+(\lambda)=A_j$ and the definition of
$\widetilde{\Ar}_j^+(\lambda)$, we deduce that
$$
\widetilde{\Ar}_j^+(\lambda)=\{a+c_j\mid a\in\Ar_j\}\cup\{c_j-a-1\mid
a\in\Arv_j^-\ \text{and }a<c_j\}.
$$
We remark that $a+c_j\geq c_j$ for $a\in\Ar_j$ and $c_j-a-1<c_j$ for
$a<c_j$. Hence, Proposition~\ref{prop:mainsensdirect} gives
$$\Ar'_j=\Ar_j.$$
We remark that
$$
\widetilde{\Le}_j^+(\lambda)=\{l-c_j\mid l\in\Le_j\ \text{and}\
l\geq c_j\},
$$
and that
$$\widetilde{\Le}_j^-(\lambda)=\{l+c_j\mid
l\in\Lev_j^-\}\cup\{c_j-l-1\mid l\in L_j\ \text{and}\
l<c_j\}.$$
Thus,
$$\{l+c_j\mid l\in\widetilde{\Le}_j^+(\lambda)\}=\{l\in\Le_j\mid j\geq
c_j\}$$
and
$$\{l\in\widetilde{\Le}_j^-(\lambda)\mid l<c_j\}=\{c_j-l-1\mid
l\in\Le_j\ \text{and}\ l<c_j\}.$$
\smallskip
Since $l=c_j-(c_j-l-1)-1$,
Proposition~\ref{prop:mainsensdirect} gives
$$\Le'_j=\{l\in\Le_j \mid l\geq c_j\}\cup \{l\in\Le_j\mid
 l< c_j\}=\Le_j.$$

A similar argument shows that, if $c_j\leq 0$, then
$$c_j'=c_j,\quad \Ar'_j=\Ar_j\quad\text{and}\quad \Le'_j=\Le_j.$$
It follows that $\lambda'^j=\lambda^j$ as required.
\end{proof}
\medskip

\begin{example} Let $t=5$. Consider $\underline c=(0,1,-1,1,-1)$ and
$\underline{\lambda}=(\emptyset,\emptyset,\emptyset,\emptyset,(2))$. We
have $\Fr((2))=(0\mid 1)$. Hence, Theorem~\ref{thm:main} gives
$$L_0=L_1=L_3=\emptyset,\ L_2=\{2\}\quad\text{and}\quad
L_4=\{5\}\cup\{0\},$$
and 
$$A_0=A_2=\emptyset,\ A_1=\{1\},\ A_3=\{3\}\quad\text{and}\quad
A_4=\{4\},$$
whence $\Fr(\lambda)=(5,2,0\mid 1,3,4)$.  Although, this process is
completely numerical, it has an abacus interpretation. Indeed, we can
recover the pointed $t$-abacus of $\lambda$ from the abacus of each
$\lambda^j$ and the characteristic vector. 
\begin{enumerate}[(i)]
\item First, we contruct the abacus of $\lambda^0,\,
\ldots,\,\lambda^{t-1}$ to obtain a pointed $t$-abacus. 
\medskip
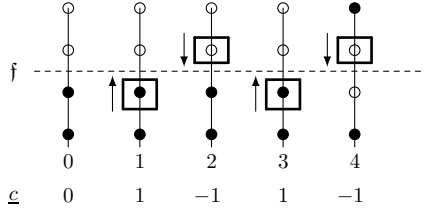
\begin{center}
\definecolor{sqsqsq}{rgb}{0.12549019607843137,0.12549019607843137,0.12549019607843137}
\begin{tikzpicture}[line cap=round,line join=round,>=triangle
45,x=0.7cm,y=0.7cm, scale=0.8,every node/.style={scale=0.8}]

\draw [dash pattern=on 2pt off 2pt](-2.5,1.5)-- (7,1.5);
\draw(-3,1.5)node{$\mathfrak f$};
\draw(-3,-1.5)node{$\underline c$};

\draw (-2,-1.1) node[anchor=north west] {$0$};
\draw (-0.3,-1.1) node[anchor=north west] {$1$};
\draw (1.1,-1.1) node[anchor=north west] {$-1$};
\draw (3.1,-1.1) node[anchor=north west] {$1$};
\draw (4.5,-1.1) node[anchor=north west] {$-1$};

\draw (-2,-0.3) node[anchor=north west] {$0$};

\draw (-1.7,3)-- (-1.7,-0.3);

\begin{scriptsize}

\draw (-1.7,2) circle (2.5pt);
\draw (-1.7,3) circle (2.5pt);

\draw [fill=black] (-1.7,1) circle (2.5pt);
\draw [fill=black] (-1.7,0) circle (2.5pt);
\end{scriptsize}

\draw (-0.3,-0.3) node[anchor=north west] {$1$};

\draw (0.,3)-- (0.,-0.3);

\begin{scriptsize}
	
\draw  (0,2) circle (2.5pt);
\draw (0,3) circle (2.5pt);

\draw [fill=black](0,1) circle (2.5pt);
\draw [fill=black](0,0) circle (2.5pt);

\draw [line width=1.pt](0.4,1.3)--(0.4,0.6);
\draw [line width=1.pt](0.4,0.6)--(-0.4,0.6);
\draw [line width=1.pt](-0.4,0.6)--(-0.4,1.3);
\draw [line width=1.pt](-0.4,1.3)--(0.4,1.3);

\draw [-latex](-0.65,0.55)--(-0.65,1.4);

\draw [line width=1.pt](3,1.3)--(3,0.6);
\draw [line width=1.pt](3,0.6)--(3.8,0.6);
\draw [line width=1.pt](3.8,0.6)--(3.8,1.3);
\draw [line width=1.pt](3.8,1.3)--(3,1.3);

\draw [-latex](2.75,0.55)--(2.75,1.4);

\draw [line width=1.pt](2.1,2.3)--(2.1,1.7);
\draw [line width=1.pt](2.1,1.7)--(1.3,1.7);
\draw [line width=1.pt](1.3,1.7)--(1.3,2.3);
\draw [line width=1.pt](2.1,2.3)--(1.3,2.3);

\draw [latex-](1.05,1.6)--(1.05,2.4);

\draw [line width=1.pt](4.7,2.3)--(4.7,1.7);
\draw [line width=1.pt](5.5,1.7)--(4.7,1.7);
\draw [line width=1.pt](5.5,1.7)--(5.5,2.3);
\draw [line width=1.pt](5.5,2.3)--(4.7,2.3);

\draw [latex-](4.45,1.6)--(4.45,2.4);

\end{scriptsize}

\draw (1.4,-0.3) node[anchor=north west] {$2$};

\draw (1.7,3)-- (1.7,-0.3);

\begin{scriptsize}
	
\draw (1.7,2) circle (2.5pt);
\draw (1.7,3) circle (2.5pt);

\draw [fill=black](1.7,1) circle (2.5pt);
\draw [fill=black](1.7,0) circle (2.5pt);
\end{scriptsize}

\draw (3.1,-0.3) node[anchor=north west] {$3$};

\draw (3.4,3)-- (3.4,-0.3);

\begin{scriptsize}
	
\draw (3.4,2) circle (2.5pt);
\draw (3.4,3) circle (2.5pt);

\draw [fill=black] (3.4,1) circle (2.5pt);
\draw [fill=black] (3.4,0) circle (2.5pt);
\end{scriptsize}

\draw (4.8,-0.3) node[anchor=north west] {$4$};

\draw (5.1,3)-- (5.1,-0.3);

\begin{scriptsize}

\draw (5.1,2) circle (2.5pt);
\draw [fill=black](5.1,3) circle (2.5pt);

\draw (5.1,1) circle (2.5pt);
\draw [fill=black](5.1,0) circle (2.5pt);
\end{scriptsize}
\end{tikzpicture}
\end{center}
Here we represent, using a window on each runner, the
characteristic vector before the $c$-push on the abacus.
\item In the example, we observe that no arms or legs ``disappear''.
However, two coarms (resp. two colegs) become two arms (resp. two legs),
hence two arms and two legs ``appear''. Applying the
$c$-push, where $c=(0,1,-1,1,-1)$, 
we obtain the pointed $5$-abacus of
$\lambda$. 
\medskip
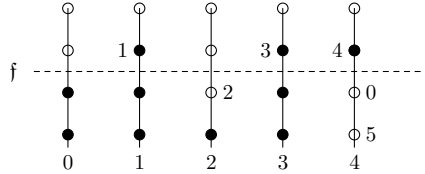
\begin{center}
\definecolor{sqsqsq}{rgb}{0.12549019607843137,0.12549019607843137,0.12549019607843137}
\begin{tikzpicture}[line cap=round,line join=round,>=triangle
45,x=0.7cm,y=0.7cm, scale=0.8,every node/.style={scale=0.8}]

\draw [dash pattern=on 2pt off 2pt](-2.5,1.5)-- (7,1.5);
\draw(-3,1.5)node{$\mathfrak f$};

\draw (-2,-.3) node[anchor=north west] {$0$};

\draw (-1.7,3)-- (-1.7,-0.3);

\begin{scriptsize}

\draw (-1.7,2) circle (2.5pt);
\draw (-1.7,3) circle (2.5pt);

\draw [fill=black] (-1.7,1) circle (2.5pt);
\draw [fill=black] (-1.7,0) circle (2.5pt);
\end{scriptsize}

\draw (-0.3,-0.3) node[anchor=north west] {$1$};

\draw (0.,3)-- (0.,-0.3);

\draw(-0.4,2)node{$1$};
\begin{scriptsize}
	
\draw [fill=black] (0,2) circle (2.5pt);
\draw (0,3) circle (2.5pt);

\draw [fill=black](0,1) circle (2.5pt);
\draw [fill=black](0,0) circle (2.5pt);

\end{scriptsize}

\draw (1.4,-0.3) node[anchor=north west] {$2$};

\draw (1.7,3)-- (1.7,-0.3);

\draw(2.1,1)node{$2$};
\begin{scriptsize}
	
\draw (1.7,2) circle (2.5pt);
\draw (1.7,3) circle (2.5pt);

\draw (1.7,1) circle (2.5pt);
\draw [fill=black](1.7,0) circle (2.5pt);
\end{scriptsize}

\draw (3.1,-0.3) node[anchor=north west] {$3$};

\draw (3.4,3)-- (3.4,-0.3);

 \draw(3,2)node{$3$};
\begin{scriptsize}
	
\draw (3.4,3) circle (2.5pt);

\draw [fill=black] (3.4,1) circle (2.5pt);
\draw [fill=black] (3.4,0) circle (2.5pt);
\draw [fill=black] (3.4,2) circle (2.5pt);
\end{scriptsize}

\draw (4.8,-0.3) node[anchor=north west] {$4$};

\draw (5.1,3)-- (5.1,-0.3);

\draw(5.5,1)node{$0$};
\draw(5.5,0)node{$5$};
\draw(4.7,2)node{$4$};
\begin{scriptsize}

\draw (5.1,2) circle (2.5pt);
\draw (5.1,3) circle (2.5pt);
\draw [fill=black](5.1,2) circle (2.5pt);
\draw (5.1,1) circle (2.5pt);
\draw (5.1,0) circle (2.5pt);
\end{scriptsize}
\end{tikzpicture}
\end{center}
\end{enumerate}
\end{example}
\subsection{Durfee number}
\label{subsec:durfee}

Let $t$ be a positive integer and $\lambda$ be a partition with
$t$-quotient $(\lambda^0,\ldots,\lambda^{t-1})$ and characteristic vector
$c(\lambda)=(c_0,\ldots,c_{t-1})$. We denote by $s(\lambda)$ the Durfee
number of $\lambda$. In this section, we propose to explain how to recover
the Durfee number of a partition $\lambda$ from the Durfee number of
$\lambda^j$ and $c(\lambda)$. To do this, however, we will need some
additional information. More precisely, for $0\leq j\leq t-1$, we return
to $\alpha_j$ and $\beta_j$ as in the proof of Theorem~\ref{thm:main}
in~(\ref{eq:preuvemain1}) and~(\ref{eq:preuvemain2}). Write
\begin{equation}
\label{eq:alphabeta}
\alpha=\sum_{j\in J^+}\alpha_j\quad\text{and}\quad \beta=\sum_{j\in
J^-}\beta_j,
\end{equation}
where  $J^+$ (resp. $J^-$) is the set of $0\leq j\leq t-1$ such that
$c_j\geq 0$ (resp. $c_j<0$). We also define $\mathcal Q_t(\lambda)$ to be
the partition satisfying
\begin{equation}
\label{eq:emptycorepart}
\psi_t(\mathcal
Q_t(\lambda))=((0,\cdots,0),(\lambda^0,\ldots,\lambda^{t-1})),
\end{equation}
where $\psi_t:\mathcal P\rightarrow \Z^t\times\mathcal P^t$ is the
bijection given in~(\ref{eq:bijectioncharquo}).

\begin{corollary}
\label{cor:durfee}
We keep the above notation. In particular, $s(\lambda)$ is the Durfee
number of $\lambda$. Then
$$s\left(\mathcal
Q_t(\lambda)\right)=\sum_{j=0}^{t-1}s\left(\lambda^j\right),$$
and
\begin{equation}
\label{eq:cordurfee}
s(\lambda)+\alpha+\beta=s\left(\mathcal Q_t(\lambda)\right)+
s\left(\lambda_{(t)}\right).
\end{equation}
\end{corollary}

\begin{proof}
Write $(\Le\mid\Ar)$ for the Frobenius symbol of $\lambda$.  By the
definition of the Frobenius symbol, $|\Le|=|\Ar|=s(\lambda)$.
Now,~(\ref{eq:preuvemain4}) gives
\begin{align*}
s(\lambda)&=\sum_{j=0}^{t-1}|\Le_j|-\sum_{j\in
J^+}\alpha_j+\sum_{j\in J^-}\alpha_j\\
&=\sum_{j=0}^{t-1}s(\lambda^j)-\alpha-\beta+\beta+\sum_{j\in
J^-}\alpha_j\\
&=\sum_{j=0}^{t-1}s(\lambda^j)-\alpha-\beta+\sum_{j\in
J^-}\underbrace{\alpha_j+\beta_j}_{=|c_j|\ \text{by
}(\ref{eq:preuvemain3})}\\
&=\sum_{j=0}^{t-1}s(\lambda^j)-\alpha-\beta+s(\lambda_{(t)}),
\end{align*}
by~(\ref{eq:duftycore1}). Applying this equality to the partition
$\mathcal Q_t(\lambda)$, we obtain
$$s\left(\mathcal Q_t(\lambda)\right)=\sum_{j=0}^{t-1}s(\lambda^j),$$
and the result follows.
\end{proof}

\begin{remark}
A key argument in~\cite{BrNa} is that, if $t$ is an odd prime number and
$\lambda$ is a self-conjugate partition, then 
\begin{equation}
\label{eq:syman}
s(\lambda)+s\left(\mathcal Q_t(\lambda)\right)\equiv
s(\lambda_{(t)})-2\beta\mod 4,
\end{equation}
In~\cite{BrNa}, this congruence is proved using specific properties of
self-conjugate partitions, and without the Frobenius symbol approach
developed here. However, Equation~(\ref{eq:syman}) is in fact a
consequence of Corollary~\ref{cor:durfee}, as we will see now.  Denote by
$(\mathcal R_0,\ldots,\mathcal R_{t-1})$ the pointed $t$-abacus of
$\mathcal Q_t(\lambda)$. Each $\mathcal R_j$ (for $0\leq j\leq t-1$) is
pointed, and $\mathcal R_{t-j-1}=\overline{\mathcal R}_j$ by
Corollary~\ref{cor:descconjugate}. Consider the runners $\mathcal R_j$ and
$\mathcal R_{t-j-1}$. Since $\mathcal R_j$ is pointed, the numbers of
positive black beads and of negative white beads on the runner $\mathcal
R_j$ are equal. On the other hand, the negative white beads correspond to
the positive black beads of $\mathcal R_{t-j-1}$ because $\mathcal
R_{t-j-1}=\overline{\mathcal R}_j$. Hence, the number of positive black
beads on the runners $\mathcal R_j$ and $\mathcal R_{t-j-1}$ is the same,
and we deduce that $s(\mathcal Q_t(\lambda))$ is even. It follows that
$2s(\mathcal Q_t(\lambda))\equiv 0\mod 4$. Now, adding $s(\mathcal
Q_t(\lambda))$ to the equation~(\ref{eq:cordurfee}) and looking it modulo
$4$, we obtain
$$s(\lambda)+s(\mathcal Q_t(\lambda))\equiv
s(\lambda_{(t)})-\alpha-\beta\mod 4.$$
Let $j\in J^+$. Then $t-j-1\in J^-$, and $\mathcal
R_{t-j-1}=\overline{\mathcal R}_j$ implies that $\alpha_j=\beta_{t-j-1}$.
Hence, $\alpha=\beta$ by~(\ref{eq:alphabeta}) and the result follows.
\end{remark}

\begin{remark}
We now give a geometric interpretation of Relation~(\ref{eq:preuvemain4}),
that can be rewritten by
\begin{equation}
\label{eq:nouvforme}
s(\lambda)=\sum_{j\in J^-}\left(s(\lambda^j)+\alpha_j\right)+
\sum_{j\in J^+}\left(s(\lambda^j)-\alpha_j\right).
\end{equation}
For any $0\leq j\leq t-1$, we consider the Young diagram $[\lambda^j]$.
We begin with the corner of the box $\mathfrak b_{s(\lambda^j)}$.
\begin{enumerate}[(i)]
\item If $c_j\geq 0$, then we follow the rim from the right to the left
during $c_j$ steps. To each step $1\leq k\leq c_j$, we set
$\varepsilon_k=-1$ if we go left, and $\varepsilon_k=0$ if we go down.
\item If $c_j\leq 0$, then we follow the rim from the left to the right
during $|c_j|$ steps. To each step $1\leq k\leq |c_j|$, we set
$\varepsilon_k=1$ if we go right, and $\varepsilon_k=0$ if we go up. 
\end{enumerate}
Note that, even if we arrive ``at the end'' of the diagram (that is, on
the $x$-axis or $y$-axis), we continue the movement to have globally
$|c_j|$ steps. In particular, for the empty partition, we begin at $(0,0)$
and we follow the axis.

When we follow the rim, each horizontal step below (resp. above)  the
diagonal corresponds to a leg (resp. to a coleg).  On the following
picture, we show the correspondence between the colegs (resp. legs) and
the horizontal path on the rim above (resp. below) the diagonal.

\bigskip
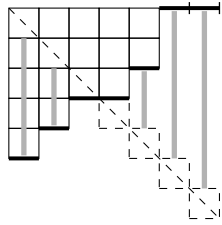
\begin{center}
\begin{tikzpicture}[scale=0.4,draw/.append style={black},baseline=\shadedBaseline]
      \draw(0,0)+(-.5,-.5)rectangle++(.5,.5);
      \draw(1,0)+(-.5,-.5)rectangle++(.5,.5);
      \draw(2,0)+(-.5,-.5)rectangle++(.5,.5);
      \draw(3,0)+(-.5,-.5)rectangle++(.5,.5);
      \draw(4,0)+(-.5,-.5)rectangle++(.5,.5);
      \draw(0,-1)+(-.5,-.5)rectangle++(.5,.5);
      \draw(1,-1)+(-.5,-.5)rectangle++(.5,.5);
      \draw(2,-1)+(-.5,-.5)rectangle++(.5,.5);
      \draw(3,-1)+(-.5,-.5)rectangle++(.5,.5);
      \draw(4,-1)+(-.5,-.5)rectangle++(.5,.5);
      \draw(0,-2)+(-.5,-.5)rectangle++(.5,.5);
      \draw(1,-2)+(-.5,-.5)rectangle++(.5,.5);
      \draw(2,-2)+(-.5,-.5)rectangle++(.5,.5);
      \draw(3,-2)+(-.5,-.5)rectangle++(.5,.5);
      \draw(0,-3)+(-.5,-.5)rectangle++(.5,.5);
      \draw(1,-3)+(-.5,-.5)rectangle++(.5,.5);
      \draw(0,-4)+(-.5,-.5)rectangle++(.5,.5);
      \draw[dashed](3,-3)+(-.5,-.5)rectangle++(.5,.5);
      \draw[dashed](4,-4)+(-.5,-.5)rectangle++(.5,.5);
      \draw[dashed](5,-5)+(-.5,-.5)rectangle++(.5,.5);
      \draw[dashed](6,-6)+(-.5,-.5)rectangle++(.5,.5);

      \draw[dashed](-0.5,0.5)--(6.5,-6.5);
\draw [line width=1.5pt](-0.5,-4.5)--(0.5,-4.5);
\draw [line width=1.5pt](0.5,-3.5)--(1.5,-3.5);
\draw [line width=1.5pt](1.5,-2.5)--(3.5,-2.5);
\draw [line width=1.5pt](3.5,-1.5)--(4.5,-1.5);
\draw [line width=1.5pt](4.5,.5)--(6.5,.5);
\draw [line width=.5pt](5.5,.3)--(5.5,.7);
\draw [line width=.5pt](6.5,.3)--(6.5,.7);
\draw [line width=2pt,gray!60](0,-0.5)--(0,-4.4);
\draw [line width=2pt,gray!60](1,-1.5)--(1,-3.4);
\draw [line width=2pt,gray!60](4,-1.6)--(4,-3.5);
\draw [line width=2pt,gray!60](5,.4)--(5,-4.5);
\draw [line width=2pt,gray!60](6,.4)--(6,-5.5);
\end{tikzpicture}
\end{center}
Hence,
$$\sum_{k=1}^j\varepsilon_k=\alpha_j\ \text{if $j\in
J^-$}\quad\text{and}\quad
\sum_{k=1}^j\varepsilon_k=-\alpha_j\ \text{if $j\in J^+$}.$$
It follows that
$$s(\lambda)=\sum_{j=0}^{t-1}\left(s(\lambda^j)+\sum_{k=1}^j \varepsilon_k\right).$$
\end{remark}
\begin{example}
For $\lambda=(5,5,4,2,1)$, we have
$\lambda^{(5)}=(\emptyset,\emptyset,\emptyset,\emptyset,(2))$ and
$c_5(\lambda)=(0,1,-1,1,-1)$. Thus,
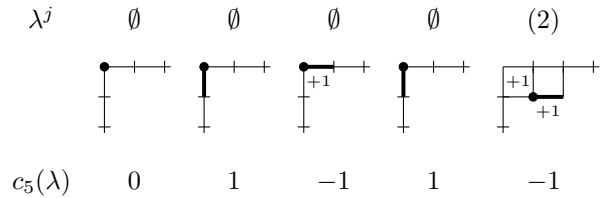
\begin{center}
\begin{tabular}{cccccc}
$\lambda^j$&$\emptyset$&$\emptyset$&$\emptyset$&$\emptyset$&$(2)$
\vspace{0.3cm}
\\
&
\begin{tikzpicture}[scale=0.4,draw/.append style={black},baseline=\shadedBaseline]
\draw [fill=black] (0.,2) circle (4pt);
\draw (0,-.2)--(0,2);
\draw (0,2)--(2.2,2);
\draw (1,2.2)--(1,1.8);
\draw (2,2.2)--(2,1.8);
\draw (-0.2,1)--(0.2,1);
\draw (-0.2,0)--(0.2,0);
\end{tikzpicture}
&
\begin{tikzpicture}[scale=0.4,draw/.append style={black},baseline=\shadedBaseline]
\draw [fill=black] (0.,2) circle (4pt);
\draw (0,-.2)--(0,2);
\draw (0,2)--(2.2,2);
\draw (1,2.2)--(1,1.8);
\draw (2,2.2)--(2,1.8);
\draw (-0.2,1)--(0.2,1);
\draw (-0.2,0)--(0.2,0);
\draw [line width=1.5pt](0,2)--(0,1);
\end{tikzpicture}
&
\begin{tikzpicture}[scale=0.4,draw/.append style={black},baseline=\shadedBaseline]
\draw [fill=black] (0.,2) circle (4pt);
\draw [line width=1.5pt](0,2)--(1,2);
\draw (0,-.2)--(0,2);
\draw (0,2)--(2.2,2);
\draw (1,2.2)--(1,1.8);
\draw (2,2.2)--(2,1.8);
\draw (-0.2,1)--(0.2,1);
\draw (-0.2,0)--(0.2,0);
\draw(0.5,1.5)node{\tiny$+1$};
\end{tikzpicture}
&
\begin{tikzpicture}[scale=0.4,draw/.append style={black},baseline=\shadedBaseline]
\draw [fill=black] (0.,2) circle (4pt);
\draw (0,-.2)--(0,2);
\draw (0,2)--(2.2,2);
\draw (1,2.2)--(1,1.8);
\draw (2,2.2)--(2,1.8);
\draw (-0.2,1)--(0.2,1);
\draw (-0.2,0)--(0.2,0);
\draw [line width=1.5pt](0,2)--(0,1);
\end{tikzpicture}
&
\begin{tikzpicture}[scale=0.4,draw/.append style={black},baseline=\shadedBaseline]
\draw [fill=black] (1.,1) circle (4pt);
\draw (0,-.2)--(0,2);
\draw (0,2)--(3.2,2);
\draw (0,1)--(2,1);
\draw (1,1)--(1,2);
\draw (2,1)--(2,2);
\draw (1,2.2)--(1,1.8);
\draw (2,2.2)--(2,1.8);
\draw (3,2.2)--(3,1.8);
\draw (-0.2,1)--(0.2,1);
\draw (-0.2,0)--(0.2,0);
\draw [line width=1.5pt](1,1)--(2,1);
\draw(1.5,0.5)node{\tiny$+1$};
\draw(0.5,1.5)node{\tiny$+1$};
\end{tikzpicture}
\\
$c_5(\lambda)$&$0$&$1$&$-1$&$1$&$-1$
\vspace{0.2cm}
\\
\end{tabular}
\end{center} 
\medskip
Hence, $s(\lambda)=1+1+1=3$.
\end{example}

\begin{example}
Consider $\lambda$ such that
$\lambda^{(5)}=((3,2),(1,1),(2,2),(1),(2))$ and
$c_5(\lambda)=(-5,2,4,1,-2)$. We have
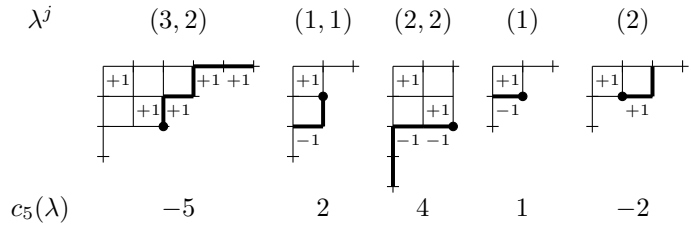
\begin{center}
\begin{tabular}{cccccc}
$\lambda^j$&$(3,2)$&$(1,1)$&$(2,2)$&$(1)$&$(2)$
\vspace{0.3cm}
\\

&
\begin{tikzpicture}[scale=0.4,draw/.append style={black},baseline=\shadedBaseline]
\draw [fill=black] (2.,0) circle (4pt);
\draw (0,-1.2)--(0,2);
\draw (1,0)--(1,2);
\draw (2,0)--(2,2);
\draw (3,1)--(3,2);
\draw (0,2)--(5.2,2);
\draw (0,1)--(3.2,1);
\draw (0,0)--(2.2,0);
\draw (1,2.2)--(1,1.8);
\draw (2,2.2)--(2,1.8);
\draw (3,2.2)--(3,1.8);
\draw (4,2.2)--(4,1.8);
\draw (5,2.2)--(5,1.8);
\draw (-0.2,1)--(0.2,1);
\draw (-0.2,0)--(0.2,0);
\draw (-0.2,-1)--(0.2,-1);
\draw [line width=1.5pt](2,0)--(2,1);
\draw [line width=1.5pt](2,1)--(3,1);
\draw [line width=1.5pt](3,1)--(3,2);
\draw [line width=1.5pt](3,2)--(4,2);
\draw [line width=1.5pt](4,2)--(5,2);
\draw(2.5,0.5)node{\tiny$+1$};
\draw(3.5,1.5)node{\tiny$+1$};
\draw(4.5,1.5)node{\tiny$+1$};
\draw(0.5,1.5)node{\tiny$+1$};
\draw(1.5,0.5)node{\tiny$+1$};
\end{tikzpicture}
&
\begin{tikzpicture}[scale=0.4,draw/.append style={black},baseline=\shadedBaseline]
\draw [fill=black] (1.,1) circle (4pt);

\draw (0,-1.2)--(0,2);
\draw (0,1)--(1,1);
\draw (0,0)--(1,0);
\draw (1,2)--(1,0);
\draw (0,2)--(2.2,2);
\draw (1,2.2)--(1,1.8);
\draw (2,2.2)--(2,1.8);
\draw (-0.2,1)--(0.2,1);
\draw (-0.2,0)--(0.2,0);
\draw (-0.2,-1)--(0.2,-1);
\draw [line width=1.5pt](1,1)--(1,0);
\draw [line width=1.5pt](1,0)--(0,0);
\draw(0.5,-0.5)node{\tiny$-1$};
\draw(0.5,1.5)node{\tiny$+1$};
\end{tikzpicture}
&
\begin{tikzpicture}[scale=0.4,draw/.append style={black},baseline=\shadedBaseline]
\draw [fill=black] (2.,0) circle (4pt);
\draw [line width=1.5pt](0,0)--(1,0);
\draw [line width=1.5pt](1,0)--(2,0);
\draw [line width=1.5pt](0,0)--(0,-2);
\draw (0,-2.2)--(0,2);
\draw (0,2)--(2.2,2);
\draw (0,1)--(2,1);
\draw (1,2)--(1,0);
\draw (2,2)--(2,0);
\draw (0,0)--(2,0);
\draw (1,2.2)--(1,1.8);
\draw (2,2.2)--(2,1.8);
\draw (-0.2,1)--(0.2,1);
\draw (-0.2,0)--(0.2,0);
\draw (-0.2,-1)--(0.2,-1);
\draw (-0.2,-2)--(0.2,-2);
\draw(0.5,-0.5)node{\tiny$-1$};
\draw(1.5,-0.5)node{\tiny$-1$};
\draw(0.5,1.5)node{\tiny$+1$};
\draw(1.5,0.5)node{\tiny$+1$};
\end{tikzpicture}
&
\begin{tikzpicture}[scale=0.4,draw/.append style={black},baseline=\shadedBaseline]
\draw [fill=black] (1.,1) circle (4pt);
\draw (0,-.2)--(0,2);
\draw (0,1)--(1,1);
\draw (1,1)--(1,2);
\draw (0,2)--(2.2,2);
\draw (1,2.2)--(1,1.8);
\draw (2,2.2)--(2,1.8);
\draw (-0.2,1)--(0.2,1);
\draw (-0.2,0)--(0.2,0);
\draw [line width=1.5pt](1,1)--(0,1);
\draw(0.5,0.5)node{\tiny$-1$};
\draw(0.5,1.5)node{\tiny$+1$};
\end{tikzpicture}
&
\begin{tikzpicture}[scale=0.4,draw/.append style={black},baseline=\shadedBaseline]
\draw [fill=black] (1.,1) circle (4pt);
\draw (0,-.2)--(0,2);
\draw (0,2)--(3.2,2);
\draw (0,1)--(2,1);
\draw (1,1)--(1,2);
\draw (2,1)--(2,2);
\draw (1,2.2)--(1,1.8);
\draw (2,2.2)--(2,1.8);
\draw (3,2.2)--(3,1.8);
\draw (-0.2,1)--(0.2,1);
\draw (-0.2,0)--(0.2,0);
\draw [line width=1.5pt](1,1)--(2,1);
\draw [line width=1.5pt](2,1)--(2,2);
\draw(0.5,1.5)node{\tiny$+1$};
\draw(1.5,0.5)node{\tiny$+1$};
\end{tikzpicture}
\\
$c_5(\lambda)$&$-5$&$2$&$4$&$1$&$-2$
\vspace{0.2cm}
\\
\end{tabular}
\end{center} 
\medskip
Hence, $s(\lambda)=5+0+0+0+2=7$.
\end{example}

\section{Consequences in representation theory}
\label{sec:part3}

Let $\mathcal R$ be a pointed abacus, and $0\leq j\leq t-1$. We denote by
$\iota(\mathcal R)$ the abacus with the same beads as $\mathcal R$, but
with negative slots labeled by $x\in\N$ on $\mathcal R$ labeled instead by
by $-x-1$ on $\iota(\mathcal R)$. The image of $\iota$ is the set of
$\Z$-labeled abacus with a finite number of positive black beads and of
negative white beads.

\subsection{Level $h$ action of the affine Weyl group $W_{t}$ on the
abacus}

Let $t\geq 2$ be an integer. Denote by $W_t$ the affine Weyl group
$\widetilde{A}_{t-1}$. Recall that $W_t$ is generated by $
w_0,\ldots,w_{t-1}$ with relations
$$\begin{array}{llll}
w_j^2&=&1&\text{for all } 0\leq j\leq t-1,\\
w_jw_i&=&w_iw_j&\text{if }i\neq j\pm 1,\\
w_jw_{j+1}w_j&=&w_{j+1}w_jw_{j+1}&\text{if }t\neq 2,\\
\end{array}$$
where the indices are interpreted modulo $t$. Now,
following~\cite[\S3]{Fayers}, recall that for any $h\in\Z$ prime to $t$,
the group $W_t$ acts on $\Z$ by
$$w_j(k)=\left\{
\begin{array}{ll}
k+h&\text{if } k\in (j-1)h-\frac 1 2(t-1)(h-1) +t\Z\\
k-h&\text{if } k\in jh-\frac 1 2(t-1)(h-1) +t\Z\\
k&\text{otherwise,}
\end{array}
\right.$$
for any $0\leq j\leq t-1$ and $k\in\Z$. This is the so-called level $h$
action of $W_t$ on $\Z$. Note that $\frac 1 2(t-1)(h-1)\in\Z$ since $h$
and $t$ are coprime.\medskip

We remark that the level $h$ action of $W_t$ on $\Z$ induces an action on
the set of pointed abaci as follows. Let $\mathcal R$ be a pointed abacus.
Write $\mathcal S=\iota(\mathcal R)$. For $0\leq j\leq t-1$, denote by
$w_j(\mathcal S)$ the $\Z$-labeled abacus obtained from $\mathcal S$ by
applying $w_j$ to each labels. Since the numbers of positive black beads
and of negative white beads of $\mathcal S$ are finite, this is the same
for $w_j(\mathcal S)$.  Hence, $\iota^{-1}(w_j(\mathcal S))$ is a pointed
abacus, denoted by $w_j(\mathcal R)$.  
 
\begin{lemma}
\label{lem:haction}
Let $t\geq 2$ be an integer and $h\in\Z$ be coprime to $t$. If $\mathcal
R$ is a pointed abacus of a partition, then for all $0\leq j\leq t-1$, the
abacus $w_j(\mathcal R)$ obtained by the level $h$ action of $W_t$ is the
pointed abacus of a partition. 
\end{lemma}

\begin{proof}
Let $\mathcal T=(\mathcal R_0,\ldots,\mathcal R_{t-1})$ be the pointed
$t$-abacus of $\lambda$.  Set $\mathcal S=\iota(\mathcal R)$ and $\mathcal
S_k=\iota(\mathcal R_k)$ for all $0\leq k\leq t-1$. Write
$$h=\alpha t+r,$$
where $\alpha\in \Z$ and $0\leq r\leq t-1$. 

Let $0\leq j\leq t-1$. Write $\mathcal R'=w_j(\mathcal R)$ and $\mathcal
S'=w_j(\mathcal S)$. Denote $(\mathcal R'_0,\ldots,\mathcal R'_{t-1})$ for
the pointed $t$-abacus of $\mathcal R'$ and set $\mathcal
S'_k=\iota(\mathcal R'_k)$ for all $0\leq k\leq t-1$.  We denote by $0\leq
r_j\leq t-1$ and $0\leq r_j'\leq t-1$ the residues of $hj-\frac 1
2(t-1)(h-1)$ and $h(j-1)-\frac 1 2(t-1)(h-1)$ modulo $t$.  We have
$$r_j-r\equiv hj-\frac 1 2(t-1)(h-1)-h\equiv r'_j\mod t.$$ Furthermore, if
$r_j-r<0$, then $t+r_j-r\in\{0,\ldots,t-1\}$. In particular
\begin{equation}
\label{eq:lienrjrjprine}
r'_j=\left\{
\begin{array}{ll}
r_j-r&\text{if }r_j-r\geq 0\\
t+r_j-r&\text{otherwise}.
\end{array}
\right. 
\end{equation}
Let $b\in\Z$ be a label of a slot of $\mathcal S$. Let $q,\,k\in\Z$ such
that $b=qt+k$ and $0\leq k\leq t-1$. In particular, the bead labeled by
$k$ on $\mathcal S$ is moved on the $q$th slot of $\mathcal S_k$.

Suppose $k\notin\{r_j,\,r'_j\}$. Then $w_j(b)=b$, and 
\begin{equation}
\label{eq:Zabac1}
\mathcal S'_k=\mathcal S_k.
\end{equation}
Suppose now that $k=r_j$. We have 
$w_j(b)=b-h=(q-\alpha)t+r_j-r$. Hence,~(\ref{eq:lienrjrjprine}) gives
$$w_j(b)=\left\{
\begin{array}{ll}
(q-\alpha)t+r_j'&\text{if }r_j-r\geq 0\\
(q-\alpha-1)t+r'_j&\text{otherwise}.
\end{array}
\right.$$
It follows that 
\begin{equation}
\label{eq:Zabac2}
\mathcal S'_{r'_j}=p^{-\alpha}(\mathcal S_{r_j})\ \text{if
}r_j-r\geq 0\quad \text{and}\quad
\mathcal S'_{r'_j}=p^{-\alpha-1}(\mathcal S_{r_j})\ \text{otherwise.
}
\end{equation}
Finally, assume that $k=r'_j$. We have 
$w_j(b)=b+h=(q+\alpha)t+r'_j+r$. Hence,~(\ref{eq:lienrjrjprine}) gives
\begin{equation}
\label{eq:Zabac3}
\mathcal S'_{r_j}=p^{\alpha}(\mathcal S_{r'_j})\ \text{if
}r_j-r\geq 0\quad \text{and}\quad
\mathcal S'_{r_j}=p^{\alpha+1}(\mathcal S_{r'_j})\ \text{otherwise.
}
\end{equation}
Now, define $d_j=\alpha$ if $r_j-r\geq 0$ and $d_j=\alpha+1$ otherwise. We
also set
\begin{equation}
\label{eq:epsilonpush}
\epsilon_j=(0,\ldots,0,\underbrace{-d_j}_{r'_j-\text{th}},0,\ldots,0,\underbrace{d_j}_{r_j-\text{th}},0,\ldots,0)\in
\core_t.\end{equation}
Then~(\ref{eq:Zabac1}), (\ref{eq:Zabac2}) and~(\ref{eq:Zabac3}) imply that
\begin{equation}
\label{eq:pushlascoux}
(\mathcal R'_0,\ldots,\mathcal R'_{r'_j},\ldots,\mathcal R'_{r_j},\ldots\mathcal R'_{t-1})=p^{\epsilon_j}(\mathcal
R_0,\ldots,\mathcal R_{r_j},\ldots,\mathcal R_{r'_j},\ldots,\mathcal R_{t-1}),
\end{equation}
where the $\epsilon_j$-push $p^{\epsilon_j}$ is defined
in~(\ref{eq:push}). Furthermore, $(\mathcal R_0,\ldots,\mathcal R_{t-1})$
is a pointed $t$-abacus of a partition. Hence, $(\mathcal
R_0',\ldots,\mathcal R'_{t-1})$ is also a pointed $t$-abacus of a
partition by Lemma~\ref{lem:addchar}. It follows that $\mathcal R'$ is a
pointed abacus of a partition, as required.
\end{proof}

\begin{remark}
\noindent
\begin{enumerate}[(i)]
\item By Lemma~\ref{lem:haction}, if $\lambda$ is a partition with pointed
abacus $\mathcal T$, then $w_j(\mathcal R)$ is the pointed abacus of a
partition $w_i(\lambda)$. The map $\lambda\mapsto w_i(\lambda)$ is
introduced and studied in~\cite{Fayers}.
\item The level $1$ action gives the Lascoux maps~\cite{Lascoux}.
\end{enumerate} 
\end{remark}

\begin{corollary}
Let $t\geq 2$ be an integer and $h\in\Z$ be coprime to $t$. Write
$h=\alpha t+r$ with $\alpha\in\Z$ and $0\leq r\leq t-1$. Let $0\leq j\leq
t-1$, and define $\epsilon_j$ as in Equation~(\ref{eq:epsilonpush}).  Let
$\lambda$ be a partition with Frobenius symbol
$\Fr(\lambda)=(\Le\mid\Ar)$.  For $0\leq k\leq t-1$ with $k\notin
\{r_j,\,r'_j\}$ set $\Ar'_k=\Ar^+_k(\lambda)$ and
$\Le'_k=\Le^+_k(\lambda)$.  

If $d_j\geq 0$, then set 
\begin{align*}
\Ar_j=&\{(q-d_j)t+r'_j\mid q\in \widetilde{\Ar}^+_{r_j}, q\geq d_j\}\cup
\{(q+d_j)t+r_j\mid q\in \widetilde{\Ar}^+_{r'_j} \}\\
&\cup \{(d_j-q-1)t+r_j\mid q\in \widetilde{\Ar}^-_{r'_j}\ \text{and}\
q < d_j \},
\end{align*}
and
\begin{align*}
\Le_j=&\{(q-d_j)t+r_j\mid q\in \widetilde{\Le}^+_{r'_j}, q\geq d_j\}\cup
\{(q+d_j)t+r'_j\mid q\in \widetilde{\Le}^+_{r_j} \}\\
&\cup \{(d_j-q-1)t+r'_j\mid q\in \widetilde{\Le}^-_{r_j}\ \text{and}\
q < d_j \},
\end{align*}

If $d_j\leq 0$, then set 
\begin{align*}
\Ar_j=&\{(q+d_j)t+r_j\mid q\in \widetilde{\Ar}^+_{r'_j}, q\geq |d_j|\}\cup
\{(q-d_j)t+r'_j\mid q\in \widetilde{\Ar}^+_{r_j} \}\\
&\cup \{(|d_j|-q-1)t+r'_j\mid q\in \widetilde{\Ar}^-_{r_j}\ \text{and}\
q < |d_j| \},
\end{align*}
and
\begin{align*}
\Le_j=&\{(q+d_j)t+r'_j\mid q\in \widetilde{\Le}^+_{r_j}, q\geq d_j\}\cup
\{(q-d_j)t+r_j\mid q\in \widetilde{\Le}^+_{r'_j} \}\\
&\cup \{(|d_j|-q-1)t+r_j\mid q\in \widetilde{\Le}^-_{r'_j}\ \text{and}\
q < |d_j| \}.
\end{align*}
Then the Frobenius symbol of $w_j(\lambda)$ is $(\Le'\mid\Ar')$
where
$$\Le'=\bigcup_{k\notin\{r_j,r'_j\}} \Le'_k\cup \Le_j\quad\text{and}\quad
\Ar'=\bigcup_{k\notin\{r_j,r'_j\}} \Ar'_k\cup \Ar_j.$$
\end{corollary}

\begin{proof}
This is just a reinterpretation in term of arms and legs of
Equation~(\ref{eq:pushlascoux}). 
\end{proof}

\begin{example}
Let $t=2$, $h=3$ and $\lambda=(4,1)$ as in~\cite[Example p.\,65]{Fayers}.
Since $\Fr(\lambda)=(1\mid 3)$, the pointed $2$-abacus of $\lambda$ is
\medskip
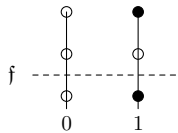
\begin{center}
\definecolor{sqsqsq}{rgb}{0.12549019607843137,0.12549019607843137,0.12549019607843137}
\begin{tikzpicture}[line cap=round,line join=round,>=triangle
45,x=0.7cm,y=0.7cm, scale=0.8,every node/.style={scale=0.8}]

\draw [dash pattern=on 2pt off 2pt](-2.5,1.5)-- (1,1.5);
\draw(-3,1.5)node{$\mathfrak f$};

\draw (-2,0.7) node[anchor=north west] {$0$};

\draw (-1.7,3)-- (-1.7,0.7);

\begin{scriptsize}

\draw (-1.7,2) circle (2.5pt);
\draw (-1.7,3) circle (2.5pt);

\draw (-1.7,1) circle (2.5pt);
\end{scriptsize}

\draw (-0.3,0.7) node[anchor=north west] {$1$};

\draw (0.,3)-- (0.,0.7);

\begin{scriptsize}
	
\draw  (0,2) circle (2.5pt);
\draw (0,3) [fill=black] circle (2.5pt);

\draw [fill=black](0,1) circle (2.5pt);

\end{scriptsize}
\end{tikzpicture}
\end{center}
Furthermore, we have $3=2\alpha+r$ with $\alpha=1$ and $r=1$.  Now, for
$j=0$, we have $r_0=1$ and $r'_0=0$. Hence, $r_0-r\geq 0$ and
$\epsilon_0=(-1,1)$.  For $j=1$, we have $r_1=0$ and $r'_1=1$, hence
$r_1-r<0$ and $\epsilon_1=(2,-2)$. It follows that $$w_0(\mathcal
R_0,\mathcal R_1)=p^{(-1,1)}(\mathcal R_1,\mathcal
R_0)\quad\text{and}\quad w_1(\mathcal R_0,\mathcal
R_1)=p^{(2,-2)}(\mathcal R_1,\mathcal R_0),$$ and the pointed $2$-abacus
of $w_0(\lambda)$ and $w_1(\lambda)$ are
\medskip
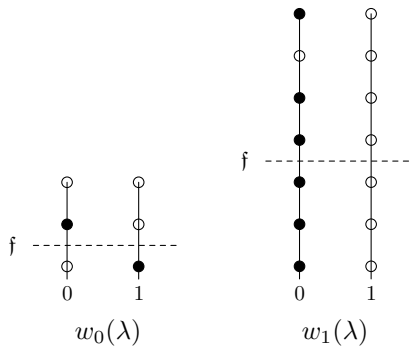
\begin{center}
\begin{tabular}{lcl}
\definecolor{sqsqsq}{rgb}{0.12549019607843137,0.12549019607843137,0.12549019607843137}
\begin{tikzpicture}[line cap=round,line join=round,>=triangle
45,x=0.7cm,y=0.7cm, scale=0.8,every node/.style={scale=0.8}]

\draw [dash pattern=on 2pt off 2pt](-2.5,1.5)-- (1,1.5);
\draw(-3,1.5)node{$\mathfrak f$};

\draw (-2,0.7) node[anchor=north west] {$0$};

\draw (-1.7,3)-- (-1.7,0.7);

\begin{scriptsize}

\draw (-1.7,2)[fill=black] circle (2.5pt);
\draw (-1.7,3) circle (2.5pt);

\draw (-1.7,1) circle (2.5pt);
\end{scriptsize}

\draw (-0.3,0.7) node[anchor=north west] {$1$};

\draw (0.,3)-- (0.,0.7);

\begin{scriptsize}
	
\draw  (0,2) circle (2.5pt);
\draw (0,3)  circle (2.5pt);

\draw [fill=black](0,1) circle (2.5pt);
\end{scriptsize}
\end{tikzpicture}
&
\quad
&
\begin{tikzpicture}[line cap=round,line join=round,>=triangle
45,x=0.7cm,y=0.7cm, scale=0.8,every node/.style={scale=0.8}]

\draw [dash pattern=on 2pt off 2pt](-2.5,1.5)-- (1,1.5);
\draw(-3,1.5)node{$\mathfrak f$};

\draw (-2,-1.3) node[anchor=north west] {$0$};

\draw (-1.7,5)-- (-1.7,-1.3);

\begin{scriptsize}

\draw (-1.7,4) circle (2.5pt);
\draw (-1.7,5)[fill=black] circle (2.5pt);

\draw (-1.7,2)[fill=black] circle (2.5pt);
\draw (-1.7,3)[fill=black] circle (2.5pt);

\draw (-1.7,1)[fill=black] circle (2.5pt);
\draw (-1.7,0)[fill=black] circle (2.5pt);
\draw (-1.7,-1)[fill=black] circle (2.5pt);
\end{scriptsize}

\draw (-0.3,-1.3) node[anchor=north west] {$1$};

\draw (0.,5)-- (0.,-1.3);

\begin{scriptsize}

\draw  (0,4) circle (2.5pt);
\draw  (0,5) circle (2.5pt);
\draw  (0,2) circle (2.5pt);
\draw (0,3) circle (2.5pt);

\draw (0,1) circle (2.5pt);
\draw  (0,0) circle (2.5pt);
\draw  (0,-1) circle (2.5pt);

\end{scriptsize}
\end{tikzpicture}\\
\vspace{3pt}
\hspace{1.1cm}$w_0(\lambda)$
&&
\hspace{0.9cm}$w_1(\lambda)$
\end{tabular}
\end{center}
Thus, $w_0(\lambda)=(1,1)$ and $w_1(\lambda)=(7,4,3,2,1)$, and
$\Fr(w_0)=(1\mid 0)$ and $\Fr(w_1)=(4,2,0\mid 0,2,6)$.
\end{example}

\begin{remark}
\label{rk:choixabacus} 
To define $\lambda\mapsto w_j(\lambda)$, we need to associate first to
$\lambda$ a subset of $\Z$, here the $\beta$-sequence
$Y=(\lambda_j-j)_{j\geq 1}$ (see Remark~\ref{rk:usualabacus}), so that
$W_t$ can act on this set. Note that if we change the $\beta$-sequence,
$w_j(\lambda)$ changes. So, in a way, $w_j(\lambda)$ is not canonically
well-defined, and depends on the choice of a $\beta$-sequence associated
to $\lambda$.
\end{remark}

\subsection{Scopes maps}
\label{subsec:scopes}
In this section, $p$ denotes a prime number, and we consider the level $1$
action of the affine Weyl group $W_p=\langle w_0,\ldots,w_{p-1}\rangle$ on
$\Z$. By \cite[6.1.21]{James-Kerber}, the $p$-blocks of symmetric groups
are labeled by the set of tuples $(\lambda,w)$ where $\lambda$ is a
$p$-core partition and $w\in\N$ is the $p$-weight of the block.
Conversely, for any such a tuple $(\lambda,w)\in\mathcal C_p\times \N$,
there is a unique $p$-block with $p$-core $\lambda$ and $p$-weight $w$.  

Let $w\in\N$. Recall that in~\cite{scopes} page 443, Scopes associates to
any $p$-core partition $\lambda$ a $p$-abacus, called \emph{the Scopes
abacus of $\lambda$}, and denoted by $\mathcal
R_{\operatorname{sc}}(\lambda)$. We write $X_{\text{sc}}(\lambda)$ for its
corresponding $\beta$-sequence; see Remark~\ref{rk:usualabacus}. For any
$0\leq j\leq p-1$, denote by $b_j$ the number of black beads on the $j$th
runner of $\mathcal R_{\operatorname{sc}}(\lambda)$. Assume that $1\leq
j\leq p-1$. When $b_j-b_{j-1}\geq w$, we write
$\operatorname{Sc}_j(\lambda)$ for the $p$-core partition associated to
the $\beta$-sequence $w_j(X_{\operatorname{sc}}(\lambda))$.  The map
$\operatorname{Sc}_j$, only defined for $p$-core partition $\lambda$ whose
Scopes abacus satisfies the condition $b_j-b_{j-1}\geq w$, is called
\emph{the $j$-Scopes map}. Note that in the following, a Scopes map will
designate a $j$-Scopes map for some $1\leq j\leq p-1$.  Now, we can define
the Scopes process, abbreviated by \emph{Sc-process}: begin with a
$p$-core $\lambda$ and apply iteratively a Scopes map until no Scopes map
can be applied.

We now describe Scopes process on characteristic vectors corresponding to
$p$-core partitions. First, for any $c=(c_0,\ldots,c_{p-1})\in \core_p$,
$w\in\N$, and $0\leq j\leq p-1$, we say that $c$ is \emph{ $j$-allowed}
whenever the following condition is satisfied.
\begin{equation}
\label{eq:condscopes}
\begin{cases}
c_j-c_{j-1}\geq w &\text{if }1\leq j\leq p-1,\\
c_0-c_{t-1}> w&\text{if }j=0.
\end{cases}  
\end{equation}
Let $0\leq j\le p-1$. If $c\in \core_p$ is $j$-allowed, then
we set
\begin{equation}
\label{eq:scopesmap}
\operatorname{sc}_j(c)=
\begin{cases}
(c_0,\ldots,c_{j-2},c_j,c_{j-1},c_{j+1},\ldots,c_{p-1}) &\text{if }1\leq j\leq 1,\\
(c_{p-1}+1,c_1,\ldots,c_{p-1},c_0-1)&\text{if }j=0.
\end{cases}  
\end{equation}
Similarly as above, we define the \emph{sc-process} with respect to these
maps, beginning with a characteristic vector $c\in \core_p$ and applying
iteratively such a map until no map can be applied.
\begin{lemma}
Let $w\in\N$ and $p$ be a prime number. Identifying the $p$-core
partitions with their characteristic vectors as in
\S\ref{subsec:charcore}. The Sc-process and sc-process are the same.
\end{lemma}

\begin{proof}
First, we remark that for any $\beta$-sequence $X$ representing $\lambda$
and $r\in\Z$, we have
\begin{equation}
\label{eq:lienbetaset} 
w_j(X^{+r})=w_{j-r}(X)^{+r},
\end{equation}
where the index is taken modulo $p$.  Indeed, for any $x\in X$,
\begin{align*}
w_{j-r}(x)+r&=
\begin{cases}
x+r+1&\text{if } x\equiv j-r-1\mod t\\ 
x+r-1&\text{if } x\equiv j-r\mod t\\
x+r&\text{otherwise}
\end{cases}\\
&=
\begin{cases}
(x+r)+1&\text{if } x+r\equiv j-1\mod t\\ 
(x+r)-1&\text{if } x+r\equiv j\mod t\\
(x+r)&\text{otherwise}
\end{cases}\\
&=w_j(x+r),
\end{align*}
as required.

Let $c=(c_0,\ldots,c_{p-1})$ be a characteristic vector, and
$\lambda_c\in\mathcal C_p$ its associated $p$-core partition as in
\S\ref{subsec:charcore}. Denote by $Y=(\lambda_k-k)_{k\geq 1}$, where
$\lambda_c=(\lambda_1,\lambda_2,\ldots)$, the $\beta$-sequence
corresponding to the pointed abacus of $\lambda_c$ as in
Remark~\ref{rk:usualabacus}. Then there is $r\in\Z$ such that
$Y^{+r}=X_{\operatorname{sc}}(\lambda_c)$. Let $1\leq i\leq t-1$.  Then
$\operatorname{Sc}_i(\lambda_c)$ has $\beta$-sequence
$w_i(X_{\operatorname{sc}}(\lambda_c))=w_i(Y^{+r})=w_{i-r}(Y)^{+r}$ by
Equation~(\ref{eq:lienbetaset}). In particular, $w_{i-r}(Y)$ is also a
$\beta$-sequence of the $\operatorname{Sc}_i(\lambda_c)$.  Furthermore,
for $0\leq j\leq p-1$, the $j$-runner of the $p$-abacus of $Y$ is sent to
the $(j+r)$-runner of $X_{\operatorname{sc}}(\lambda_c)$. Hence, the
$(i-r)$-runner and $(i-r-1)$-runner of $Y$ are sent to the $i$ and $(i-1)$
runners of $X_{\operatorname{sc}}(\lambda_c)$, and by construction of
$\lambda_c$, we have 
\begin{equation}
\label{eq:scopecond}
b_{i}-b_{i-1}=
\begin{cases}c_{i-r}-c_{i-r-1}& \text{if }i-r\not\equiv 0\mod p\\
c_{0}-c_{p-1}-1&\text{if } i-r\equiv 0\mod p.
\end{cases}
\end{equation}
We then can apply $\operatorname{Sc}_i$ to $\lambda_c$ if and only if
$b_i-b_{i-1}\geq w$ if and only if $c$ is $(i-r)$-allowed, that is we can
apply $\operatorname{sc}_{i-r}$ to $c$.  Conversely, if $j$ satisfies
condition~(\ref{eq:condscopes}), then we cannot have $j\equiv -r\mod p$,
because the $0$-runner of $X_{\operatorname{sc}}(\lambda_c)$ corresponds
to a runner with a minimum number of black beads.

Finally, assume $0\leq j\leq p-1$ such that
condition~(\ref{eq:condscopes}) is satisfied, then $w_j(Y)$ is a
$\beta$-sequence of $\operatorname{Sc}_{j+r}(\lambda_c)$, and if $\mathcal
R$ denotes the pointed abacus of $\lambda_c$, then $w_j(\mathcal R)$ is
the one of $w_j(\lambda)$ by Remark~\ref{rk:usualabacus} and
Lemma~\ref{lem:haction}.  Since $1=h=\alpha p+r$ with $\alpha=0$ and
$r=1$, and $r_j=j$, $r'_j=j-1$ for $1\leq j\leq p-1$ and $r_0=0$,
$r'_0=p-1$, we deduce that $\epsilon_j=(0,\ldots,0)$ for $1\leq j\leq p-1$
and $\epsilon_0=(1,0,\ldots,0,-1)$. Now, Equation~(\ref{eq:pushlascoux})
and the construction of $\lambda_c$ from $c$ give the result.
\end{proof}

\begin{example}
\label{ex:scopes} 
Let $w=2$. Consider the $3$-core partition $\mu=(4,2)$.  We have
$\Fr(\mu)=(1,0\mid 0,3)$, hence $c(\mu)=(2,-1,-1)$, and
$Y=(3,0,-3,-4,\ldots)$. Furthermore, 
$$X_{\operatorname{sc}(\mu)}=\{11,8,5,4,3,2,1,0,-1,\ldots\}=Y^{+8}.$$
Then, the $0$-runner of the pointed abacus $\mathcal R$ of $\mu$ is sent
on the $0+8\equiv 2$ runner of the Scopes abacus $\mathcal
R_{\operatorname{sc}}(\mu)$ of $\mu$. Similarly, the $1$-runner and the
$2$-runner of $\mathcal R$ are sent to the $0$-runner and the $1$-runner
of $\mathcal R_{\operatorname{sc}}(\mu)$.

\medskip
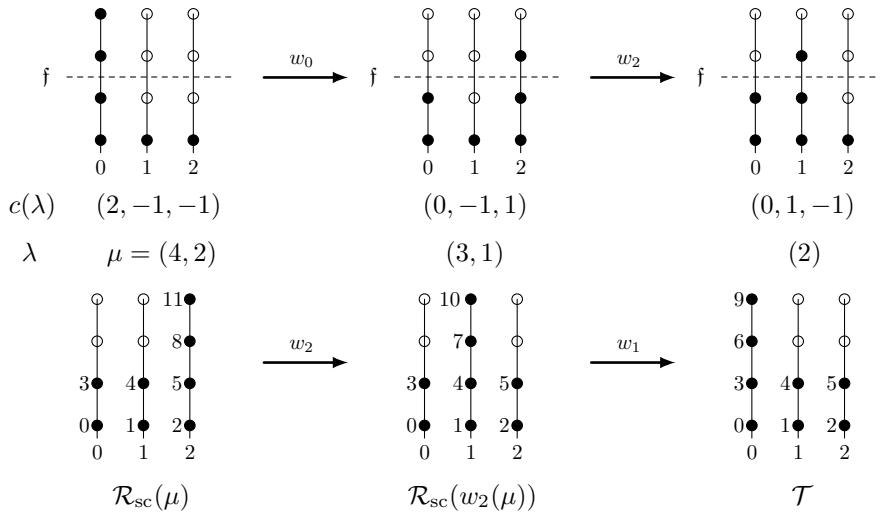
\begin{center}
\begin{tabular}{lllll}
\begin{tikzpicture}[line cap=round,line join=round,>=triangle
45,x=0.7cm,y=0.7cm, scale=0.8,every node/.style={scale=0.8}]

\draw [dash pattern=on 2pt off 2pt](-2.5,1.5)-- (1.5,1.5);
\draw(-3,1.5)node{$\mathfrak f$};

\draw (-2,-0.3) node[anchor=north west] {$0$};

\draw (-1.7,3)-- (-1.7,-0.3);

\begin{scriptsize}

\draw (-1.7,2)[fill=black] circle (2.5pt);

\draw (-1.7,3)[fill=black]circle (2.5pt);

\draw (-1.7,1)[fill=black] circle (2.5pt);

\draw (-1.7,0)[fill=black] circle (2.5pt);
\end{scriptsize}

\draw (-0.9,-0.3) node[anchor=north west] {$1$};

\draw (-0.6,3)-- (-0.6,-0.3);

\begin{scriptsize}
	
\draw  (-0.6,2) circle (2.5pt);
\draw (-0.6,3)  circle (2.5pt);

\draw (-0.6,1) circle (2.5pt);

\draw (-0.6,0)[fill=black]  circle (2.5pt);
\end{scriptsize}

\draw (0.2,-0.3) node[anchor=north west] {$2$};

\draw (0.5,3)-- (0.5,-0.3);

\begin{scriptsize}

\draw  (0.5,2) circle (2.5pt);
\draw (0.5,3) circle (2.5pt);
\draw (0.5,1) circle (2.5pt);
\draw (0.5,0)[fill=black]  circle (2.5pt);
\end{scriptsize}

\end{tikzpicture}
&
\begin{tikzpicture}[line cap=round,line join=round,>=triangle
45,x=0.7cm,y=0.7cm, scale=0.8,every node/.style={scale=0.8}]
\path[line width=1pt,-latex](-1,3.5) edge (1,3.5);
\draw(-0.1,3.9)node{$w_0$};
\draw(-0.1,1.2)node{\relax};
\end{tikzpicture}
\hspace{-0.3cm}
&
\begin{tikzpicture}[line cap=round,line join=round,>=triangle
45,x=0.7cm,y=0.7cm, scale=0.8,every node/.style={scale=0.8}]

\draw [dash pattern=on 2pt off 2pt](-2.5,1.5)-- (1.5,1.5);
\draw(-3,1.5)node{$\mathfrak f$};

\draw (-2,-0.3) node[anchor=north west] {$0$};

\draw (-1.7,3)-- (-1.7,-0.3);

\begin{scriptsize}

\draw (-1.7,2) circle (2.5pt);

\draw (-1.7,3) circle (2.5pt);

\draw (-1.7,1)[fill=black] circle (2.5pt);

\draw (-1.7,0)[fill=black] circle (2.5pt);
\end{scriptsize}

\draw (-0.9,-0.3) node[anchor=north west] {$1$};

\draw (-0.6,3)-- (-0.6,-0.3);

\begin{scriptsize}
	
\draw  (-0.6,2) circle (2.5pt);
\draw (-0.6,3)  circle (2.5pt);

\draw (-0.6,1) circle (2.5pt);
\draw (-0.6,0)[fill=black]  circle (2.5pt);

\end{scriptsize}

\draw (0.2,-0.3) node[anchor=north west] {$2$};

\draw (0.5,3)-- (0.5,-0.3);

\begin{scriptsize}

\draw  (0.5,2)[fill=black] circle (2.5pt);
\draw (0.5,3) circle (2.5pt);
\draw (0.5,1)[fill=black] circle (2.5pt);
\draw (0.5,0)[fill=black] circle (2.5pt);

\end{scriptsize}
\end{tikzpicture}
&
\begin{tikzpicture}[line cap=round,line join=round,>=triangle
45,x=0.7cm,y=0.7cm, scale=0.8,every node/.style={scale=0.8}]
\path[line width=1pt,-latex](-1,3.5) edge (1,3.5);
\draw(-0.1,3.9)node{$w_2$};
\draw(-0.1,1.2)node{\relax};
\end{tikzpicture}
\hspace{-0.3cm}
&
\begin{tikzpicture}[line cap=round,line join=round,>=triangle
45,x=0.7cm,y=0.7cm, scale=0.8,every node/.style={scale=0.8}]

\draw [dash pattern=on 2pt off 2pt](-2.5,1.5)-- (1.5,1.5);
\draw(-3,1.5)node{$\mathfrak f$};

\draw (-2,-0.3) node[anchor=north west] {$0$};

\draw (-1.7,3)-- (-1.7,-0.3);

\begin{scriptsize}

\draw (-1.7,2) circle (2.5pt);

\draw (-1.7,3)circle (2.5pt);

\draw (-1.7,1)[fill=black] circle (2.5pt);

\draw (-1.7,0)[fill=black] circle (2.5pt);
\end{scriptsize}

\draw (-0.9,-0.3) node[anchor=north west] {$1$};

\draw (-0.6,3)-- (-0.6,-0.3);

\begin{scriptsize}
	
\draw  (-0.6,2)[fill=black] circle (2.5pt);
\draw (-0.6,3)  circle (2.5pt);

\draw (-0.6,1)[fill=black] circle (2.5pt);

\draw (-0.6,0)[fill=black] circle (2.5pt);
\end{scriptsize}

\draw (0.2,-0.3) node[anchor=north west] {$2$};

\draw (0.5,3)-- (0.5,-0.3);

\begin{scriptsize}

\draw  (0.5,2)circle (2.5pt);
\draw (0.5,3) circle (2.5pt);
\draw (0.5,1) circle (2.5pt);
\draw (0.5,0)[fill=black] circle (2.5pt);
\end{scriptsize}
\end{tikzpicture}
\\
\vspace{3pt}
$\hspace{-10pt}c(\lambda)$\hspace{0.4cm} $(2,-1,-1)$&&$\hspace{0.8cm}(0,-1,1)$&&
\hspace{0.8cm}$(0,1,-1)$\vspace{3pt}\\
$\hspace{-5pt}\lambda$\hspace{0.8cm} $\mu=(4,2)$&&$\hspace{1.1cm}(3,1)$&&
\hspace{1.3cm}$(2)$
\vspace{0.2cm}
\\

\begin{tikzpicture}[line cap=round,line join=round,>=triangle
45,x=0.7cm,y=0.7cm, scale=0.8,every node/.style={scale=0.8}]

\draw(-3,1.5)node{\relax\ };

\draw (-2,0.7) node[anchor=north west] {$0$};

\draw (-1.7,4)-- (-1.7,0.7);

\draw(-2,1)node{$0$};
\draw(-2,2)node{$3$};

\begin{scriptsize}

\draw (-1.7,2)[fill=black] circle (2.5pt);

\draw (-1.7,3) circle (2.5pt);
\draw (-1.7,4) circle (2.5pt);
\draw (-1.7,1)[fill=black] circle (2.5pt);
\end{scriptsize}

\draw (-0.9,0.7) node[anchor=north west] {$1$};

\draw (-0.6,4)-- (-0.6,0.7);

\draw(-0.9,1)node{$1$};
\draw(-0.9,2)node{$4$};

\begin{scriptsize}
	
\draw  (-0.6,2)[fill=black] circle (2.5pt);
\draw (-0.6,3)  circle (2.5pt);
\draw (-0.6,4)  circle (2.5pt);
\draw (-0.6,1)[fill=black] circle (2.5pt);

\end{scriptsize}

\draw (0.2,0.7) node[anchor=north west] {$2$};

\draw (0.5,4)-- (0.5,0.7);

\draw(0.2,1)node{$2$};
\draw(0.2,2)node{$5$};
\draw(0.2,3)node{$8$};
\draw(0.1,4)node{$11$};

\begin{scriptsize}

\draw  (0.5,2)[fill=black] circle (2.5pt);
\draw  (0.5,4)[fill=black] circle (2.5pt);
\draw (0.5,3)[fill=black] circle (2.5pt);
\draw (0.5,1)[fill=black] circle (2.5pt);

\end{scriptsize}
\end{tikzpicture}
&
\begin{tikzpicture}[line cap=round,line join=round,>=triangle
45,x=0.7cm,y=0.7cm, scale=0.8,every node/.style={scale=0.8}]
\path[line width=1pt,-latex](-1,3.5) edge (1,3.5);
\draw(-0.1,3.9)node{$w_2$};
\draw(-0.1,1.2)node{\relax};
\end{tikzpicture}
\hspace{-0.3cm}
&

\begin{tikzpicture}[line cap=round,line join=round,>=triangle
45,x=0.7cm,y=0.7cm, scale=0.8,every node/.style={scale=0.8}]

\draw(-3,1.5)node{\relax};

\draw (-2,0.7) node[anchor=north west] {$0$};

\draw (-1.7,4)-- (-1.7,0.7);

\draw(-2,1)node{$0$};
\draw(-2,2)node{$3$};

\begin{scriptsize}

\draw (-1.7,2)[fill=black] circle (2.5pt);

\draw (-1.7,3) circle (2.5pt);
\draw (-1.7,4) circle (2.5pt);
\draw (-1.7,1)[fill=black] circle (2.5pt);
\end{scriptsize}

\draw (-0.9,0.7) node[anchor=north west] {$1$};

\draw (-0.6,4)-- (-0.6,0.7);

\draw(-0.9,1)node{$1$};
\draw(-0.9,2)node{$4$};
\draw(-0.9,3)node{$7$};
\draw(-1.1,4)node{$10$};
\begin{scriptsize}
	
\draw  (-0.6,2)[fill=black] circle (2.5pt);
\draw (-0.6,3)[fill=black]  circle (2.5pt);
\draw (-0.6,4)[fill=black]  circle (2.5pt);
\draw (-0.6,1)[fill=black] circle (2.5pt);

\end{scriptsize}

\draw (0.2,0.7) node[anchor=north west] {$2$};

\draw (0.5,4)-- (0.5,0.7);

\draw(0.2,1)node{$2$};
\draw(0.2,2)node{$5$};

\begin{scriptsize}

\draw  (0.5,2)[fill=black] circle (2.5pt);
\draw  (0.5,4) circle (2.5pt);
\draw (0.5,3) circle (2.5pt);
\draw (0.5,1)[fill=black] circle (2.5pt);

\end{scriptsize}
\end{tikzpicture}
&
\begin{tikzpicture}[line cap=round,line join=round,>=triangle
45,x=0.7cm,y=0.7cm, scale=0.8,every node/.style={scale=0.8}]
\path[line width=1pt,-latex](-1,3.5) edge (1,3.5);
\draw(-0.1,3.9)node{$w_1$};
\draw(-0.1,1.2)node{\relax};
\end{tikzpicture}
\hspace{-0.3cm}
&

\begin{tikzpicture}[line cap=round,line join=round,>=triangle
45,x=0.7cm,y=0.7cm, scale=0.8,every node/.style={scale=0.8}]

\draw(-3,1.5)node{\relax};

\draw (-2,0.7) node[anchor=north west] {$0$};

\draw (-1.7,4)-- (-1.7,0.7);

\draw(-2,1)node{$0$};
\draw(-2,2)node{$3$};
\draw(-2,3)node{$6$};
\draw(-2,4)node{$9$};

\begin{scriptsize}

\draw (-1.7,2)[fill=black] circle (2.5pt);

\draw (-1.7,3)[fill=black] circle (2.5pt);
\draw (-1.7,4)[fill=black] circle (2.5pt);
\draw (-1.7,1)[fill=black] circle (2.5pt);
\end{scriptsize}

\draw (-0.9,0.7) node[anchor=north west] {$1$};

\draw (-0.6,4)-- (-0.6,0.7);

\draw(-0.9,1)node{$1$};
\draw(-0.9,2)node{$4$};

\begin{scriptsize}
	
\draw  (-0.6,2)[fill=black] circle (2.5pt);
\draw (-0.6,3)  circle (2.5pt);
\draw (-0.6,4)  circle (2.5pt);
\draw (-0.6,1)[fill=black] circle (2.5pt);
\end{scriptsize}

\draw (0.2,0.7) node[anchor=north west] {$2$};

\draw (0.5,4)-- (0.5,0.7);

\draw(0.2,1)node{$2$};
\draw(0.2,2)node{$5$};

\begin{scriptsize}

\draw  (0.5,2)[fill=black] circle (2.5pt);
\draw  (0.5,4)circle (2.5pt);
\draw (0.5,3) circle (2.5pt);
\draw (0.5,1)[fill=black] circle (2.5pt);

\end{scriptsize}
\end{tikzpicture}
\vspace{2pt}\\
\hspace{0.9cm}
$\mathcal R_{\operatorname{sc}}(\mu)$&&$\hspace{0.6cm}\mathcal R_{\operatorname{sc}}(w_2(\mu))$&&
\hspace{1.35cm}$\mathcal T$\\
\end{tabular}
\end{center}
\medskip
We first apply $w_0$ to the pointed $3$-abacus of $\mu$. Since $i\equiv
j+r\equiv 0+8\equiv 2\mod 3$, it corresponds to applying $w_2$ to
$X_{\operatorname{sc}}(\mu)$. Similarly, the second step where we apply
$w_2$ to the pointed abacus, corresponds applying $w_1$ to the Scopes
abacus, since $i\equiv j+r\equiv 2+8\equiv 1\mod 3$. Note that $\mathcal
T$ is not a Scopes abacus of the $3$-core partition $(2)$. After the third
step, we have to replace $\mathcal T$ by the Scopes abacus
$\mathcal R_{\operatorname{sc}}((2))$, namely

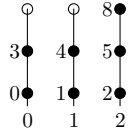
\begin{center}
\begin{tikzpicture}[line cap=round,line join=round,>=triangle
45,x=0.7cm,y=0.7cm, scale=0.8,every node/.style={scale=0.8}]

\draw(-3,1.5)node{\relax};

\draw (-2,0.7) node[anchor=north west] {$0$};

\draw (-1.7,3)-- (-1.7,0.7);

\draw(-2,1)node{$0$};
\draw(-2,2)node{$3$};
\begin{scriptsize}

\draw (-1.7,2)[fill=black] circle (2.5pt);

\draw (-1.7,3) circle (2.5pt);
\draw (-1.7,1)[fill=black] circle (2.5pt);
\end{scriptsize}

\draw (-0.9,0.7) node[anchor=north west] {$1$};

\draw (-0.6,3)-- (-0.6,0.7);

\draw(-0.9,1)node{$1$};
\draw(-0.9,2)node{$4$};
\begin{scriptsize}
	
\draw  (-0.6,2)[fill=black] circle (2.5pt);
\draw (-0.6,3)  circle (2.5pt);
\draw (-0.6,1)[fill=black] circle (2.5pt);

\end{scriptsize}

\draw (0.2,0.7) node[anchor=north west] {$2$};

\draw (0.5,3)-- (0.5,0.7);

\draw(0.2,1)node{$2$};
\draw(0.2,2)node{$5$};
\draw(0.2,3)node{$8$};

\begin{scriptsize}

\draw  (0.5,2)[fill=black] circle (2.5pt);
\draw (0.5,3)[fill=black] circle (2.5pt);
\draw (0.5,1)[fill=black] circle (2.5pt);
\end{scriptsize}
\end{tikzpicture}
\end{center}
Now we can clearly see that there is no allowed Scopes map on this abacus
because the condition on beads is not satisfied. Hence, the process is
finished.
\end{example}

\begin{definition}
Let $c\in \core_p$.  If $c$ is $j$-allowed for some $0\leq j\leq p-1$,
then we say that $c$ is \emph{allowed}. We say that $c$ is an
\emph{ancestor} when $c$ is not allowed.  If there is no $j$-allowed
$c'\in \core_p$ for all $0\leq j\leq p-1$ such that
$\operatorname{sc}_j(c')=c$, then we say that $c$ is a \emph{coancestor}.
\end{definition}

\begin{example}
Let $w=2$. The characteristic vector $(0,1,-1)$ is an ancestor. It is not
a coancestor because $\operatorname{Sc}_2(0,-1,1)=(0,1,-1)$. The vector
$(1,0,-1)$ is a coancestor. 
\end{example}

Now we introduce the following relation $\mathcal \sim$ on $\core_p$
defined by $c\sim c'$ if and only if there is $0\leq j\leq p-1$ such that
either $c'$ is $j$-allowed and $c=\operatorname{Sc}_j(c')$ or $c$ is
$j$-allowed $c'=\operatorname{Sc}_j(c)$. We complete $\sim$ by reflexivity
and transitivity. The classes on $\mathcal C_p$ for this equivalence
relation are called \emph{the Scopes families}. Note that each Scopes
family contains exactly one ancestor. In particular, the set of ancestors
parametrizes the Scopes families. 
\medskip

Define 
$$\mathcal E=\left\{(k_1,\ldots,k_p)\in\N^t\mid
\sum\limits_{j=1}^pk_j=pw-(p-1)\right\},$$
and for all $0\leq j\leq p-1$,
$$\mathcal E_j=\left\{(k_1,\ldots,k_{p})\in\mathcal E\mid 
\sum_{i=1}^p ik_i\equiv j \mod p\right\}.$$
We notice that $(\mathcal E_j)_{0\leq j\leq p-1}$ forms a
set-partition of $\mathcal E$.
\begin{theorem}
\label{thm:ancesters} 
Let $p$ be a prime number and $w\geq 1$.
Write $j_0\in\N$ such that $0\leq j_0\leq p-1$ and 
\begin{equation}
\label{eq:j0}
j_0\equiv \frac 1 2(w-1)p(p+1)\mod p.
\end{equation}
For any $\underline k=(k_1,\ldots,k_p)\in\mathcal E_{j_0}$, we set
$$c_{\underline k,0}=\frac 1 p\left(\frac{(w-1)p(p+1)}2 -\sum_{j=1}^p
jk_j\right)+1,$$
and for all $1\leq j\leq p-1$,
$$c_{\underline k,j}=c_{\underline
k,0}+j(w-1)-\sum\limits_{i=1}^jk_j.$$  
Then the set
$$\mathcal S=\{(c_{\underline k,0},\ldots,c_{\underline k, p-1})\in
\core_p\mid \underline k\in\mathcal E_{j_0}\}$$
is the set of ancestors for $p$ and $w$, and parametrizes the Scopes
families.
\end{theorem}
\begin{proof}
By (\ref{eq:condscopes}),
$c=(c_0,\ldots,c_{p-1})\in\Z^p$ is an ancestor if and only if
\begin{equation}
\label{eq:systemeancester} 
\begin{cases}
c_j-c_{j-1}\leq w-1\quad\text{for all } 1\leq j\leq p-1,\\
c_0-c_{p-1}\leq w,\\
c_0+\cdots+c_{p-1}=0.
\end{cases}
\end{equation}

We set $c'_0=c_0$ and $c'_j=c_j-c_{j-1}$ for $1\leq j\leq p-1$. In
particular, $c_j=c'_0+\cdots+c'_j$ for all $0\leq j\leq p-1$.
Furthermore, the condition $c_0+\ldots+c_{t-1}=0$ is equivalent to
$\sum\limits_{j=0}^{p-1}(p-j)c'_j=0$. Hence, (\ref{eq:systemeancester}) is
equivalent to
\begin{equation}
\label{eq:systemeancester2} 
\begin{cases}
c'_0=c_0\\
c'_j\leq w-1\quad\text{for all } 1\leq j\leq p-1,\\
-\sum_{j=1}^{p-1}c'_j\leq w,\\
\sum_{j=0}^{p-1}(p-j)c'_j=0.
\end{cases}
\end{equation}
\smallskip

Now, note that the system~(\ref{eq:systemeancester2}) is equivalent to the
existence of $k_1,\ldots, k_p\in\N$ such that
$$
\begin{cases}
c'_0=c_0\\
c'_j+k_j = w-1,\ 1\leq j\leq p-1,\\
-\sum_{j=1}^{p-1}c'_j+k_p=w,\\
\sum_{j=0}^{p-1}(p-j)c'_j=0.
\end{cases}
\Longleftrightarrow
\
\begin{cases}
c'_0=c_0\\
c'_j= w-1-k_j,\ 1\leq j\leq p-1,\\
\sum_{j=1}^{p}k_j=pw-(p-1),\\
\sum_{j=0}^{p-1}(p-j)c'_j=0.
\end{cases}
$$
\smallskip
We note that since
\begin{eqnarray*}
\sum_{j=0}^{p-1}(p-j)c'_j&=&pc_0+\sum_{j=1}^p(p-j)c'_j\\
&=&pc_0+\sum_{j=1}^p(p-j)(w-1-k_j)\\
&=&pc_0+\sum_{j=1}^p p(w-1)-(w-1)\sum_{j=1}^pj-p\sum_{j=1}^p
k_j+\sum_{j=1}^pj k_j\\
&=&pc_0+p^2(w-1)-\frac 1 2 (w-1)p(p+1)-p(pw-p-1)
+\sum_{j=1}^pj k_j\\
&=&p(c_0-1)-\frac 1 2(w-1)p(p+1)+\sum_{j=1}^pj k_j,
\end{eqnarray*}
it follows that $\sum\limits_{j=0}^{p-1}(p-j)c'_j=0$ if and only if
$$p(c_0-1)=\frac 1 2(w-1)p(p+1)-\sum_{j=1}^pj k_j,$$
if and only if $p$ divides $\frac 1 2(w-1)p(p+1)-\sum\limits_{j=1}^pj k_j$ and
$$c_0=\frac 1 p\left(\frac{(w-1)p(p+1)}2 -\sum_{j=1}^p
jk_j\right)+1.$$
The result follows.
\end{proof}
\begin{remark}
\noindent
\begin{enumerate}[(i)]
\item The condition that $p$ is a prime number is not relevant to
solve the system~(\ref{eq:systemeancester}).
\item If $p$ or $w$ is odd, then $(w-1)(p+1)$ is even. In particular,
$\frac 1 2(w-1)p(p+1)\equiv 0\mod p$. Hence, $j_0=0$ and
$$c_{\underline k,0}=\frac 1
2(w-1)(p+1)+1+\frac 1 p\sum_{j=1}^pjk_j.$$
\end{enumerate}
\end{remark}
\begin{theorem}
\label{thm:nbscopes}
Let $p$ be a prime number and $w\geq 1$. Then the number of Scopes
families is $\frac 1 p\binom{pw}{p-1}$. 
\end{theorem}

\begin{proof}
The number of Scopes families is equal to the number of ancestors, which,
by Theorem~\ref{thm:ancesters}, is equal to $|\mathcal E_{j_0}|$, where
$j_0$ is defined in Equation~(\ref{eq:j0}).  By a standard argument, we
have
$$|\mathcal E|=\binom{pw-(p-1)+(p-1)}{p-1}=\binom{pw}{p-1}.$$
Now, consider the map
$$\sigma:\N^p\longrightarrow \N^p,\,(k_1,\ldots,k_p)\longmapsto
(k_p,k_1,\ldots,k_{p-1}).
$$
We prove that $\sigma$ induces a bijection
$$\sigma:\mathcal E_j\rightarrow \mathcal E_{j+1},$$
for all $0\leq j\leq p-1$, where the indices are taken modulo $p$.  For
$(k_1,\ldots,k_p)\in\mathcal E$, write $k'_i=k_{i-1}$ for $1\leq i\leq
p-1$ and $k'_1=k_p$ so that $(k'_1,\ldots,k'_p)=\sigma(k_1,\ldots,k_p)$.
Then
\begin{eqnarray*}
\sum_{i=1}^pik'_i&\equiv&k_p+\sum_{i=1}^{p-1}(i+1)k_i \mod p\\
&\equiv&\sum_{i=1}^p k_i+\sum_{i=1}^p i k_i\mod p\\
&\equiv&1+\sum_{i=1}^p i k_i\mod p,
\end{eqnarray*}
because $\sum\limits_{i=1}^p k_i=pw-(p-1)\equiv 1\mod p$. We deduce that,
if $(k_1,\ldots,k_p)\in\mathcal K_j$, then $\sum\limits_{i=1}^p i
k_i\equiv j\mod p$ and $\sum\limits_{i=1}^pi k'_i\equiv j+1\mod p$.
Hence, $\sigma$ induces a bijective map from $\mathcal E_j$ to $\mathcal
E_{j+1}$, with inverse $\sigma^{p-1}$. It follows that 
$$|\mathcal E_j|=|\mathcal E_{j+1}|\ \text{for $0\leq j\leq
p-2$}\quad\text{and}\quad 
|\mathcal E_{p-1}|=|\mathcal E_0|.$$ 
Using the fact that $(\mathcal E_j)_{0\leq j\leq p-1}$ is a partition of
$\mathcal E$,  we obtain
$$|\mathcal E_{j_0}|=\frac 1 p\binom{pw}{p-1},$$
as required.  
\end{proof}
\medskip

\begin{example}
\noindent
\begin{enumerate}[(i)]
\item If $w=1$ then for any prime number $p$, there is only one Scopes
family.
\item Assume $p=3$. Then for any $w\geq 1$, there are $\frac 1 2 w(3w-1)$
Scopes families. When $w=2$, $j_0=0$, and
$$ \mathcal E_{j_0}=\{(0, 0, 4),\, (0, 3, 1),\, (1, 1, 2),\, (2, 2, 0),\,
(3, 0, 1) \}.$$ Theorem~\ref{thm:ancesters} then gives the set of
ancestors
$$\mathcal S=\{(-1,0,1),\,(0,1,-1),\,(0,0,0),\,(1,0,-1),\,(1,-1,0)\}.$$
\end{enumerate}
\end{example}

\begin{lemma}
\label{lem:coancestors}
Let $p$ be a prime number and $w\geq 1$. A characteristic vector
$c=(c_0,\ldots,c_{p-1})$ is a coancestor if and only if it is a solution
of the system 
\begin{equation}
\label{eq:systemecoancester} 
\begin{cases}
c_{j-1}-c_{j}\leq w-1\quad\text{for all } 1\leq j\leq p-1,\\
c_{p-1}-c_{0}\leq w-2,\\
c_0+\cdots+c_{p-1}=0.
\end{cases}
\end{equation}
\end{lemma}

\begin{proof}
Assume $c$ is not a coancestor. Then there is $0\leq j\leq p-1$ and a
$j$-allowed $c'=(c'_0,\ldots,c'_{p-1})\in \core_p$ such that
$\operatorname{Sc}_j(c')=c$. Assume $j\neq 0$.  By
Condition~(\ref{eq:condscopes}), we have $c'_j-c'_{j-1}\geq w$.
Furthermore, Equation~(\ref{eq:scopesmap}) imply that $c'_i=c_i$ if
$i\notin\{ j-1,j\}$ and $c'_j=c_{j-1}$ and $c'_{j-1}=c_j$. It follows that
$$c_{j-1}-c_j\geq w.$$ 
If $j=0$, then $c'_0-c'_{p-1}\geq w$ by~(\ref{eq:condscopes}), and
$c'_0-1=c_{p-1}$ and $c'_{p-1}+1=c_0$, that is $c'_0=1+c_{p-1}$ and
$c'_{p-1}=c_0-1$. Hence, Equation~(\ref{eq:scopesmap}) gives
$$c'_{0}-c'_{p-1}>w\quad\Longleftrightarrow\quad
c_{p-1}+1-c_{0}+1>w\quad\Longleftrightarrow\quad c_{p-1}-c_{0}>w-2.$$
This proves that $c$ is a coancestor if and only if
Condition~(\ref{eq:systemecoancester}) is satisfied.
\end{proof}

\begin{theorem}
\label{thm:finitefamilies}
The number of Scopes families with finite cardinality is $\frac 1
p\binom{pw-2}{p-1}$.
\end{theorem}

\begin{proof}
We remark that a Scopes family has a coancestor if and only if its
cardinality is finite. On the other hand, by Lemma~\ref{lem:coancestors} a
vector $c=(c_0,\ldots,c_{p-1})\in\Z^p$ is a coancestor if and only if
Condition~\ref{eq:systemecoancester} holds. By the same method used in the
proof of Theorem~\ref{thm:ancesters}, we prove that $c$ is a coancestor if
and only if there is $(k_1,\ldots,k_p)\in\N^p$ satisfying
$k_1+\cdots+k_p=p(w-1)-1$, $p$ divides $\sum\limits_{j=1}^pj k_j-\frac 1
2(w-1)p(p+1)$ and
$$c_{0}=\frac 1 p\left(\frac{(w-1)p(p+1)}2 -\sum_{j=1}^p
jk_j\right)-1,$$
and $c_{j}=c_{0}+j(w-1)-\sum\limits_{i=1}^jk_j$ for all $1\leq j\leq p-1$.
A similar argument as in the proof of Theorem~\ref{thm:nbscopes} gives the
number of coancestors is
$$\frac 1 t\binom{pw-2}{p-1}.$$
\end{proof}

\begin{remark}
\noindent
\begin{enumerate}[(i)]
\item From Theorems~\ref{thm:nbscopes} and~\ref{thm:finitefamilies}, we
deduce that the number of infinite Scopes families is
$$\frac{1}p \left(\binom{pw}{p-1}-\binom{pw-2}{p-1}\right).$$ 
\item The proof of Theorem~\ref{thm:finitefamilies} gives a
parametrization of the set of ancestors corresponding to finite Scopes
families.
\end{enumerate}
\end{remark}

\begin{example}
For $p=3$ and $w=2$, the set of finite Scopes families is parametrized by
$(0,0,2)$ and $(1,1,0)$. The corresponding coancestors are $(1,0,-1)$ and
$(0,0,0)$.
\end{example}
Consider $b$ the principal $p$-block (that is the $p$-block that contains
the trivial representation) of $\Sym_{pw}$. Then the characteristic vector
associated to $b$ is $c=(0,\ldots,0)\in \core_p$. By
Conditions~(\ref{eq:systemeancester}) and~(\ref{eq:systemecoancester}),
the Scopes family of $b$ is $\{b\}$. We now compute the number of
$p$-blocks of weight $w$ with the same property. 

\begin{proposition}
\label{prop:sizeone} Let $p$ be a prime number and $w\geq 1$. Then the
number of Scopes families of cardinality $1$ is
$$\frac 1 p
\sum_{j}(-1)^j\binom{p}{j}\binom{pw-j(2w-1)}{p(w-1)+1-j(2w-1)},$$ 
where the sum runs over $0\leq j\leq
\operatorname{min}(p,\frac{p(w-1)+1}{2(w-1)+1})$.
\end{proposition}

\begin{proof}
A vector $c=(c_0,\ldots,c_{p-1})\in\Z^p$ is both an ancestor and a
coancestor if and only if Conditions (\ref{eq:systemeancester}) and
(\ref{eq:systemecoancester}) are satisfied. We set $c'_0=c_0$ and
$c'_j=c_j-c_{j-1}$ for all $1\leq j\leq p-1$. The above conditions are
equivalent to the existence of nonnegative integers
$k_1,\ldots,k_p,k'_1,\ldots,k'_p$ such that
$$
\begin{cases}
c'_0=c_0\\
c'_j+k_j= w-1\quad\text{for all } 1\leq j\leq p-1,\\
k_j+k'_j= 2w-2\quad\text{for all } 1\leq j\leq p-1,\\
-(c'_1+\cdots+c'_{p-1})+k_p=w,\\
k_p+k'_p=2w-2,\\
\sum_{j=0}^{p-1}(p-j)c'_j=0.
\end{cases}
$$

Following the proof of Theorem~\ref{thm:ancesters}, we show that this
system is parametrized by the set
$$\mathcal K=\left\{(k_1,\ldots,k_p)\in\mathcal
E_{j_0},(k'_1,\ldots,k'_p)\in\N\mid k_j+k'_j=2(w-1)\right\}.
$$
Using the argument in the proof of Theorem~\ref{thm:nbscopes}, we have 
$$|\mathcal K|=\frac 1 p|\mathcal K'|,$$
where $\mathcal K'=\left\{(k_1,\ldots,k_p)\in\mathcal
E,\,(k'_1,\ldots,k'_p)\in\N^p\mid k'_j+k_j=2(w-1)\right\}$. We remark that
the cardinal of $\mathcal K'$ is equal to the number of
$(k_1\ldots,k_p)\in\{0,\ldots,2w-2\}^p$ such that
$k_1+\cdots+k_p=p(w-1)+1$.  If we write 
$$\left(1+x+\cdots+x^{2w-2}\right)^p=\sum_{n=0}^{p(2w-2)}a_n x^n,$$ 
where $a_n\in\N$, then $|K'|=a_{p(w-1)+1}$. Set $r=2w-2$. We have
\begin{eqnarray*}
(1+x+\cdots+x^r)^p&=&\left(\frac{1-x^{r+1}}{1-x}\right)^p\\
&=&\left(\sum_{j\geq
0}\binom{k+j-1}{j}x^j\right)\cdot\left(\sum_{j=0}^p(-1)^j\binom{p}{j}x^{(r+1)j}\right)\\
&=&\sum_{n=0}^{rp}\left(\sum_{0\leq j\leq m}
(-1)^j\binom{p}{j}\binom{p+n-j(r+1)-1}{n-j(r+1)}\right)x^n,
\end{eqnarray*}
where $m=\operatorname{min}(p,\frac n {r+1})$. The result follows.
\end{proof}

\section{Some applications in Number theory}
\label{sec:part4}
For positive integers $t$ and $n$, we define
$$C_{t,n}=\{c\in \core_t\mid |\lambda_c|=n\}\quad\text{and}\quad
SC_{t,n}=\{c\in C_{t,n}\mid \lambda_c^*=\lambda_c\},$$
where $\core_t$ and $\lambda_c$ are defined in
Equation~(\ref{eq:defct}) and in~\S\ref{subsec:charcore}. We write
\begin{equation}
\label{eq:asc}
\as_t(n)=|C_{t,n}|\quad\text{and}\quad \asc_t(n)=|SC_{t,n}|. 
\end{equation}

In~\cite[Corollary 2(1)]{GarvanKimStanton}, Garvan, Kim and Stanton prove,
using generating functions, that 
\begin{equation}
\label{eq:egalite1}
\asc_5(n)=\asc_{5}(2n+1).
\end{equation}
They also give a combinatorial proof of this equality by constructing a
bijection between $SC_{5,n}$ and $SC_{5,2n+1}$. More precisely, they
observe that the map 
$$\varphi:SC_{5,n}\longrightarrow SC_{5,2n+1},\
(-b,-a,0,a,b)\mapsto(a+b+1,a-b,0,b-a,-a-b-1)$$
is a bijection.  Indeed, if we write $c=(-b,-a,0,a,b)\in SC_{5,n}$ for
some integers $a,\,b$, then Corollary~\ref{cor:sizesym} implies
\begin{eqnarray}
\label{eq:loc3}
|\lambda_{\varphi(c)}|&=&5((b-a)^2+(-a-b-1)^2)+2(b-a)-4(a+b+1)\nonumber\\
&=&10a^2+10b^2+10a+10b+5-6a-2b-4\\
&=&2(5a^2+5b^2+2a+4b)+1\nonumber\\
&=&2|\lambda_c|+1.\nonumber
\end{eqnarray} 
In particular, the map $\varphi$ is well-defined and clearly injective. It
is surjective, because for $(a',b')\in\Z^2$ such that
$(-b',-a',0,a',b')\in SC_{5,2n+1}$, we have
\begin{equation}
\label{eq:loc2}
5a'^2+5b'^2+2a'+4b'=2n+1
\end{equation}
by Corollary~\ref{cor:sizesym}.  On the other hand, there are $a,\,b\in\Z$
such that $a'=b-a$ and $b'=-a-b-1$ if and only if
$$a=-\frac 1 2 (a'+b'+1)\quad\text{and}\quad b=\frac 1 2(a'-b'-1).$$ 
However, $a'^2+b'^2=1\mod 2$ implies that $a'+b'+1\equiv 0 \mod 2$ and
$a'-b'-1\equiv 0\mod 2$. Hence, $a$ and $b$ are integers satisfying
$\varphi(-b,-a,0,a,b)=(-b',-a',0,a',b')$, and $(-b,-a,0,a,b)\in SC_{5,n}$
by Equation~(\ref{eq:loc3}) because $(-b',-a',0,a',b')\in SC_{5,2n+1}$, as
required.
\smallskip

In this part, we will give, in the same spirit, combinatorial proofs of
equalities similar to Equation~(\ref{eq:egalite1}) by constructing
explicit bijections.

\subsection{Results on $3$-core partitions}

\begin{theorem}
\label{thm:3coreinjection}
Let $k$ be a positive integer. Assume that $k\not\equiv 0\mod 3$. Let
$q\in\N$ and $\varepsilon\in\{-1,1\}$ be such that $k=3q+\varepsilon$.
Then the map 
$$\varphi:C_{3,n}\longrightarrow C_{3,k^2n+\frac{k^2-1}3},\ (-x-y,x,y)\longmapsto
(-\varepsilon(k(y+x)+q),\varepsilon k x,\varepsilon(ky+q))$$ 
is injective.
\end{theorem}

\begin{proof}
First, we remark that if $(-x-y,x,y)\in C_{3,n}$ then 
\begin{equation}
\label{eq:loc0}
n=\frac{3}{2}\left((-x-y)^2+x^2+y^2\right)+x+2y=3x^2+3y^2+3xy+x+2y
\end{equation}
by Corollary~\ref{cor:sizecore}. The vector
$(-\varepsilon(k(y+x)+q),\varepsilon k x,\varepsilon(ky+q))\in
core_3$
is a characteristic vector, and  Equation~(\ref{eq:loc0}) gives 
\begin{eqnarray}
\left|\lambda_{\varphi(-x-y,x,y)}\right|&=&
3((kx)^2+(ky+q)^2+kx(ky+q))+\varepsilon kx+2\varepsilon
(ky+q)\nonumber\\
&=&k^2(3x^2+3y^2+3xy)+k(x+2y)(3q+\varepsilon)+q\nonumber\\
&=&k^2(3x^2+3y^2+3xy)+k(x+2y)k+\frac{k^2-1}{3}\\
&=&k^2|\lambda_{(-x-y,x,y)}|+\frac{k^2-1}{3}\nonumber,
\end{eqnarray}
hence $(-\varepsilon(k(y+x)+q),\varepsilon k x,\varepsilon(ky+q))\in
C_{3,k^2n+\frac{k^2-1}3}$, and $\varphi$ is well-defined.
It is immediate that $\varphi$ is injective.
\end{proof}

\begin{remark}
\label{rk:3coreconj}
\noindent
\begin{enumerate}[(i)]
\item For any $c=(-x-y,x,y)\in C_{3,n}$, we have
\begin{eqnarray*}
\varphi(c^*)&=&\varphi(-y,-x,y+x)\\
&=&\left(-\varepsilon(k(-x+x+y)+q),-\varepsilon kx,\varepsilon
(k(y+x)+q)\right)\\
&=&\left(-\varepsilon(ky+q),-\varepsilon kx,\varepsilon
(k(y+x)+q)\right)\\
&=&\left(-\varepsilon(k(y+x)+q),\varepsilon kb,\varepsilon
(ky+q)\right)^*\\
&=&\varphi(c)^*.
\end{eqnarray*}
\item Since $\varphi(c^*)=\varphi(c)^*$, the map $\varphi$ induces by
restriction an
injective map
$$\varphi:
SC_{3,n}\longrightarrow SC_{3,k^2n+\frac{k^2-1}3}, (-y,0,y)\longmapsto
(-\varepsilon(ky+q),0,\varepsilon(ky+q)).$$
By Corollary~\ref{cor:sizesym}, $c=(-x,0,x)\in C_{3,n}$ if and only if
$3x^2+2x=n$ if and only if
$$3(x+\frac 1 3)^2-\frac 1 3=n \Longleftrightarrow (3x+1)^2=3n+1.$$
However, when $3n+1$ is a square, it has two roots with only one congruent
to $1$ mod $3$. Hence, $|C_{3,n}|\in\{0,1\}$ and $|C_{3,n}|=1$ if and only
if $3n+1$ is a square.  Now, we remark that $3n+1$ is a square if and only
if $k^2(3n+1)=3(k^2n+\frac{k^2-1}3)+1$ is a square. Hence, for any integer
$k$ not dividing by $3$, the map $\varphi:SC_{3,n}\longrightarrow
SC_{3,k^2n+\frac{k^2-1}3}$ is bijective, and
$$\asc_3(n)=\asc_3\left(k^2n+\frac{k^2-1}3\right).$$
In particular, we recover~\cite[Theorem 5.1]{BaruahBerndt} for $k=2$
and~\cite[Theorem 3.6]{BaruahNath} for $k=p^{2m}$ where $p$ is a prime
number such that $p\equiv 2 \mod 3$.
\end{enumerate}
\end{remark}

\begin{corollary}
\label{cor:3core4n1}
The map
$$\varphi:C_{3,n}\longrightarrow C_{3,4n+1},\ (-x-y,x,y)\longmapsto
(2(y+x)+1,-2x,-2y-1)$$ 
is a bijection.
\end{corollary}

\begin{proof}
Applying Theorem~\ref{thm:3coreinjection} for $k=2$, we deduce that
$\varphi$ is well-defined and injective. Now, we prove that $\varphi$ is
surjective. Let $(-x'-y',x',y')\in C_{3,4n+1}$. Hence,
\begin{equation}
\label{eq:loc1}
4n+1=3(x'^2+y'^2+x'y')+x'+2y'.
\end{equation}
Remark that $3x'^2+x'\equiv 0\mod 2$. Thus, Equation~(\ref{eq:loc1})
implies that $y'(x'+1)\equiv 1\mod 2$. It follows that $y\equiv 1\mod 2$
and $x\equiv 0\mod 2$. There exist integers $x$ and $y$ such that $x'=-2x$
and $y'=-2y-1$. Hence, $(-x-y,x,y)\in\C_{3,n}$ and
$$\varphi(-x-y,x,y)=(-x'-y',x',y').$$
The result follows.
\end{proof}

\begin{remark}
Since $\varphi$ is bijective, we have
$$\as_3(4n+1)=|C_{3,4n+1}|=|\varphi(C_{3,n})|=|C_{3,n}|=\as_3(n),$$
and we recover~\cite[Theorem 4.1]{BaruahBerndt}. 
\end{remark}

\begin{example}
Assume $n=2$. There are two $3$-core partitions of $2$, $(2)$ and $(1,1)$,
with characteristic vectors $(0,1,-1)$ and $(1,-1,0)$. Then the
characteristic vectors of $3$-core partitions of $4\times 2+1=9$ are
$(-1,2,-1)=\varphi(1,-1,0)$ and $(1,-2,1)=\varphi(0,1,-1)$ corresponding
to the $3$-core partitions $(5,3,1)$ and $(3,2,2,1,1)$ of $9$.
\end{example}

\begin{corollary}
\label{cor:gene3core}
We keep the notation as above. We assume that $k=p^{2m}$ for some positive
integer $m$ and prime number $p$ such that $p\equiv 2\mod 3$. Then the map
$\varphi:C_{3,n}\longrightarrow C_{3,p^{2m}n+\frac{p^{2m}-1}3}$ defined by
$$\varphi(-x-y,x,y)=
\left(p^m(y+x)+\frac{p^{m}+1}3,-p^mk x,-p^m y-\frac{p^{m}+1}3\right)$$
is bijective.
\end{corollary}

\begin{proof}
By Theorem~\ref{thm:3coreinjection}, the map $\varphi$ is well-defined and
injective. The result follows by cardinality since
$\as_3(n)=\as_3(p^{2m}n+\frac{p^{2m}-1}3)$ by~\cite[Corollary
8]{HirschhornSellers}.
\end{proof}

\begin{remark}
\noindent
The condition $p\equiv 2\mod 3$ is relevant. For example, for $k=7$, we
have $C_{65,3}=\{\lambda_1,\,\lambda_2,\,\lambda_3\}$, where
$\Fr(\lambda_1)=(14, 11, 8, 5, 2 \mid 1, 2, 4, 5, 8)$,
$\Fr(\lambda_1)=(12, 9, 6, 3, 0 \mid  0, 3, 6, 9, 12)$ and
$\Fr(\lambda_3)=(8, 5, 4, 2, 1 \mid 2, 5, 8, 11, 14)$. It follows that
$$c_3(\lambda_1)=(-5,2,3),\,
c_3(\lambda_2)=(5,0,-5)\ \text{and}\ 
c_3(\lambda_1)=(-3,-2,5).$$
We have $c_3((1))=(1,0,-1)$. Hence, $\varphi(1,0,-1)=(5,0,-5)$. Note that
$\varphi$ is not bijective since $\as_3(65)=3\neq 1=\as_3(1)$.
\item
\end{remark}

\subsection{Results on self-conjugate $5$-core partitions}
\label{subsec:5core}

In this paragraph, we consider self-conjugate $5$-core partitions. To
simplify the notation, we identify $\Z^2$ and the set characteristic
vectors of self-conjugate $5$-core partitions by 
$$(x,y)\in\Z^2\longmapsto (-y,-x,0,x,y)\in \core_5.$$

Let $k$ be a positive integer.  Now, we define a map
$\varphi_k:SC_{5,n}\rightarrow SC_{5,(k^2+1)n+k^2}$ depending on the
congruence of $k$ modulo $5$. Write $q$ and $r$ for the quotient and the
remainder obtained by the euclidean division of $k$ by $5$. For $(x,y)\in
SC_{5,n}$, set
\begin{equation}
\label{eq:mapcore5gen1}
 \varphi_k(x,y)=\begin{cases}
 (x+ky+2q,-kx+y-q)&\text{if } k\equiv 0\mod 5\\
 (-kx+y-q,-x-ky-2q-1)&\text{if } k\equiv 1\mod 5\\
 (-kx-y-q-1,-yk+x-2q-1)&\text{if } k\equiv 2\mod 5\\
 (kx-y+q,x+ky+2q+1)&\text{if } k\equiv 3\mod 5\\
 (-x-ky-2q-2,kx-y+q)&\text{if } k\equiv 4\mod 5\\
\end{cases}
\end{equation}

\begin{proposition}
\label{prop:map5cores}
The map $\varphi_k:SC_{5,n}\longrightarrow SC_{5,(k^2+1)n+k^2}$ defined in
Equation~(\ref{eq:mapcore5gen1}) is injective.
\end{proposition}

\begin{proof}
The fact that $\varphi_k$ is injective is immediate. We have to prove that
$\varphi$ is well-defined. We only consider the case $k\equiv 0\mod 5$,
the other cases are similar. Let $c=(x,y)\in SC_{5,n}$.
Corollary~\ref{cor:sizesym} gives
\begin{eqnarray*}
\left|\lambda_{\varphi_k(c)}\right|&=&
5(x+ky+2q)^2+5(-kx+y-q)^2+2(x+ky+2q)+4(-kx+y-q)\\
&=&(k^2+1)(5x^2+5y^2+2x+4y)+x(20q-2k^2+10kq-4k)\\
&&\quad +y(20kq-10q+2k-4k^2)+25q^2\\
&=&(k^2+1)|\lambda_{c}|+k^2,
\end{eqnarray*}
as required.
\end{proof}

\begin{remark}
\noindent
\begin{enumerate}[(i)]
\item For $k=1$, we recover the map of Garvan, Kim and
Stanton~\cite{GarvanKimStanton} giving Equation~(\ref{eq:egalite1}).
\item For $k=2$, the map $\varphi_2:SC_{5,n}\rightarrow SC_{5,5n+4},
(x,y)\mapsto (-2x-y-1,x-2y-1)$ is bijective.  By
Proposition~\ref{prop:map5cores}, $\varphi_2$ is injective. Let
$(x',y')\in SC_{5,5n+4}$. There are $(x,y)\in SC_{5,n}$ such that
$\varphi(x,y)=(x',y')$ if and only 
$$x=\frac{y'-2x'-1}5\quad\text{and}\quad y=\frac{-x'-2y'-3}{5}.$$
However, considering the relation $5x'^2+5y'^2+2x'+4y'=5n+4$ modulo $5$,
we obtain that $x'+2y'+3\equiv 0\mod 5$ and $2x'-y'+1\equiv 0\mod 5$.
Hence, $\varphi_2$ is surjective.
In particular, we recover the equality~\cite[Corollary 2]{GarvanKimStanton}
$$\asc_5(n)=|C_{5,n}|=|C_{5,5n+4}|=\asc_5(5n+4).$$
\item For $k=3$, we can check that the corresponding map
$\varphi_3:SC_{5,n}\rightarrow SC_{5,10n+9}$ is bijective, giving by
cardinality the relation
$$\asc_5(10n+9)=\asc_5(n).$$
In fact, this last equality is not surprising by (i) and (ii), since
$10n+9=5(2n+1)+4$. We also remark that 
$$\varphi_3=\varphi_2\circ\varphi_1.$$
\end{enumerate}
\end{remark}

\begin{theorem}
\label{thm:5core2}
Let $k$ be a positive integer not divisible by $5$. We define a map
$\psi_k: SC_{5,n}\longrightarrow SC_{5,k^2n+k^2-1}$ by setting, for all
$(x,y)\in SC_{5,n}$
$$\psi_k(x,y)=
\begin{cases}
(\varepsilon(kx+q),\varepsilon(ky+2q))&\text{if }k=5q+\varepsilon \\
(\varepsilon(-ky+2q+1),\varepsilon(kx+q))&\text{if
}k=5q+2\varepsilon\\
\end{cases},
$$
where $\varepsilon\in\{-1,1\}$. Then $\psi_k$ is injective.
\end{theorem}

\begin{proof}
Let $\varepsilon\in\{-1,1\}$. The map is clearly injective. We only have
to prove that it is well-defined. We only consider the case
$k=5q+\varepsilon$ for some $q\in\N$ and $\varepsilon\in\{-1,1\}$. The
case $k=5q+2\varepsilon$ is similar. Let $c=(x,y)\in SC_{5,n}$. Then
Corollary~\ref{cor:sizesym} implies
\begin{eqnarray*}
\left|\lambda_{\psi_k(c)}\right|&=&
5(kx+q)^2+5(ky+2q)^2+2\varepsilon(kx+q)+4\varepsilon(ky+2q)\\
&=&k^2(5x^2+5y^2)+2xk(5q+\varepsilon)+4yk(5q+\varepsilon)+25q^2+10\varepsilon
q\\
&=&k^2|\lambda_{c}|+(5q+\varepsilon)^2-1\\
&=&k^2|\lambda_{c}|+k^2-1,
\end{eqnarray*}
and the result follows.
\end{proof}

\begin{corollary}
Assume $p$ is a prime number such that $p\equiv 3\mod 4$. Then for all
$m\in\N$, the map $\psi_{p^m}:SC_{5,n}\longrightarrow SC_{5,p^{2m}
n+p^{2m}-1}$ given in Theorem~\ref{thm:5core2} is a bijection.
\end{corollary}

\begin{proof}
Since $p$ is a prime and $p\neq 5$ because $p\equiv 3\mod 4$, we can apply
Theorem~\ref{thm:5core2} to $k=p^{2m}$ and we conclude with~\cite[Theorem
5.4]{BaruahNath2}.
\end{proof}
\subsection{Results on self-conjugate $7$-core partitions}
\label{subset:7core}
As in \S\ref{subsec:5core}, we identify $\Z^3$ with the set of
characteristic vectors of self-conjugate $7$-core partitions by
$$(x,y,z)\in\Z^3\longmapsto
(-z,-y,-x,0,x,y,z)\in \core_.$$
Let $k$ be a positive integer not dividing by $7$. For $(x,y,z)\in
SC_{7,n}$, we define
\begin{equation}
\label{eq:mapcore7gen1}
 \varphi_k(x,y,z)=\begin{cases}
 (kx+q,kx+2q,kx+3q)&\text{if } k=7q+1\\
 (-kz-3q-1,kx+q,-yk-2q-1)&\text{if } k=7q +2\\
 (-ky-2q-1,zk+3q+1,kx+q)&\text{if } k=7q+3\\
 (ky+2q+1,-kz-3q-2,-kx-q-1)&\text{if } k=7q+4\\
 (kz+3q+2,-kx-q-1,ky+2q+1)&\text{if } k=7q+5\\
(-kx-q-1,-kx-2q-2,-kz-3q-3)&\text{if } k=7q+6\\
\end{cases}
\end{equation}

Using Corollary~\ref{cor:sizesym}, we can show that
$\varphi_k:SC_{7,n}\longrightarrow SC_{7,k^2(n+2)-2}$ is well-defined by a
computation analogue to that in the proof of the
proposition~\ref{prop:map5cores}. Note that $\varphi_k$ is injective.

\begin{theorem}
The map $\varphi_2:SC_{7,n}\longrightarrow SC_{7,4n+6}$ given by
$$\varphi_2(-z,-y,-x,0,x,y,z)=(2y+1,-2x,2z+1,0,-2z-1,2x,-2y-1).$$
is a bijection.
\end{theorem}

\begin{proof}
First, we remark that $\varphi_2$ is the map given in
Equation~(\ref{eq:mapcore7gen1}) for $k=2$. Hence, $\varphi_2$ is
well-defined and surjective. Now, let $(x',y',z')\in SC_{7,4n+6}$.
\begin{equation}
\label{eq:loc5}
7x'^2+7y'^2+7z'^2+2x'+4y'+6z'=4n+6
\end{equation}
by Corollary~\ref{cor:sizesym}.  On the other hand, we have $x,\,y,\,z
\in\Z$ such that $x'=-2z-1$ and $y'=2x$ and $z'=-2y-1$ if and only if
$$x=\frac 1 2 y',\quad y=-\frac 1 2 (z'+1)\quad\text{and}\quad z=-\frac 1
2(x'+1).$$ 
By Equation~(\ref{eq:loc5}), either $x'$ or $y'$ or $z'$ is odd. Assume
$y'$ is odd. If $x'$ (resp. $z'$) is odd, then $z'$ (resp. $x'$) is odd,
and $x'^2\equiv y'^2\equiv z'^2\equiv 1 \mod 4$. We also necessarily have
$2x'\equiv 6z'\equiv 2\mod 4$. Thus, Equation~(\ref{eq:loc5}) gives
$1\equiv 2\mod 4$, which is a contradiction. It follows that $x'$ and $z'$
are even. Then $x'^2\equiv z'^2\equiv 2x'\equiv 6z'\equiv 0\mod 4$, and
Equation~(\ref{eq:loc5}) gives $3\equiv 2\mod 4$. We deduce that $y'$ have
to be even, and Equation~(\ref{eq:loc5}) now implies that $x'$ and $z'$
are both odd. Thus, $x,\,y,\,z$ are integers. The result follows.
\end{proof}

\begin{remark}
We recover~\cite[(9.1)]{GarvanKimStanton} the relation
$$\asc_7(4n+6)=\asc_7(n).$$ 
\end{remark}

\begin{proposition}
The map $\varphi_{4^k}:SC_{7,n}\longrightarrow SC_{7,4^k(n+2)-2}$ defined
in Equation~(\ref{eq:mapcore7gen1}) is bijective.  
\end{proposition}

\begin{proof}
First we remark that $\varphi_{4^k}$ is well-defined and injective, and e
conclude by cardinality with~\cite[Corollary 6.4]{BaruahNath2}.
\end{proof}

\subsection{Results on self-conjugate $9$-core partitions} 
\label{subsec:7cores}

We identify $\Z^4$ with the set of characteristic vectors of
self-conjugate $9$-core partitions by
$$(x,y,z,w)\in\Z^4\longmapsto
(-w,-z,-y,-x,0,x,y,z,w)\in \core_9.$$

\begin{theorem}
\label{thm:9core}
The map
$$\varphi:SC_{9,n}\longrightarrow SC_{9,4n+10},\
(x,y,z,w)\longrightarrow (-2w-1,2x,-2z-1,2y)$$
is injective.
\end{theorem}

\begin{proof}
Let $c=(x,y,z,w)\in SC_{9,n}$. Then Corollary~\ref{cor:sizesym} gives
\begin{eqnarray*}
\left|\lambda_{\varphi(c)}\right|&=&
9(2w+1)^2+9(2x)^2+9(2z+1)^2+9(2y)^2-2(2w+1)\\
&&+4(2x)-6(2z+1)+8(2y)\\
&=&9(4w^2+4w+4x^2+4z^2+4z++4y^2+2)-4w+8x-12z+16y-8\\
&=&4(9x^2+9y^2+9z^2+9w^2+2x+4y+6z+8w)+10\\
&=&4|\lambda_{c}|+10.
\end{eqnarray*}
Hence, $\varphi$ is well-defined. On the other hand, it is immediate that
$\varphi$ is injective.
\end{proof}

\begin{corollary}
\label{cor:9core}
The restriction to $SC_{9,2n}$ of the map $\varphi$ given in
Theorem~\ref{thm:9core} induces a bijection $\varphi:
SC_{9,2n}\longrightarrow SC_{9,8n+10}$.
\end{corollary}

\begin{proof}
By Theorem~\ref{thm:9core}, the map is well-defined and injective. It is
bijective by cardinality using~\cite[Theorem 4.1]{BaruahSarmah}.
\end{proof}

\begin{remark}
The map $\varphi$ given in Theorem~\ref{thm:9core} is not surjective in
general. Let $(1)\in\mathcal C_1$. We have $\Fr((2,1))=(1\mid 1)$ and
$c_9((2,1))=(0,0,-1,0)$. Furthermore, $\varphi(c_9((2,1)))=(-1,0,1,0)$ is
the characteristic vector of the $9$-core partition $(8,5,2^3,1^3)$ of
$22=4\times 3+10$. However, $\asc_9(3)=1$ and $\asc_9(22)=2$ since the
partition $(9,4,2^2,1^5)$  is a $9$-core partition of $22$ with
characteristic vector $(0,-1,0,1)$.
\end{remark}

\section{Remarks on the hooklengths}
\label{sec:part5}

In this section, $t$ denotes a positive integer.  Let $\lambda$ be a
partition with $t$-quotient
$\lambda^{(t)}=(\lambda^0,\ldots,\lambda^{t-1})\in\mathcal P^t$ and
characteristic vector $c=(c_0,\ldots,c_{t-1})\in \core_t$. 

\begin{definition}
The hooks of $\lambda^{(t)}$ are the hooks of $\lambda^{j}$ for $0\leq
j\leq t-1$.
\end{definition}

\subsection{Description of the hooklenghts}

\begin{lemma} 
\label{lem:hooktabacus}
Let $b$ be a box of $\lambda$.
\begin{enumerate}[(i)]
\item Assume $b$ is a Durfee box; See
Definition~\ref{def:nombdesboxes}.
Then there are $0\leq i,\,j\leq t-1$, $a\in\widetilde{\Ar}^+_j(\lambda)$
and $l\in\widetilde{\Le}^+_i(\lambda)$ such that
$$\mathfrak h(b)=(a+l+1)t+j-i.$$
\item Assume $b$ is an Arm-Coleg box. Then there are $0\leq i,\,j\leq t-1$,
$a\in\widetilde{\Ar}^+_j(\lambda)$, $0\leq i\leq t-1$ and $l\in\N$
satisfying either $i\leq j$ and $l\leq a$ or $j>i$ and $l<a$, such that
$$\mathfrak h(b)=(a-l)t+j-i.$$
\item Assume $b$ is a Leg-Coarm box. Then there are $0\leq i,\,j\leq t-1$,
$l\in\widetilde{\Le}^+_j(\lambda)$, $0\leq i\leq t-1$ and $a\in\N$
satisfying either $i\geq j$ and $a\leq l$ or $j<i$ and $a<l$, such that
$$\mathfrak h(b)=(l-a)t+j-i.$$
\end{enumerate}
\end{lemma}

\begin{proof}
Let $b$ be a Durfee box. By (i) of Proposition~\ref{prop:froblabelprop},
its Frobenius label is $(a',l')_{++}$ for some $a'\in\Ar^+(\lambda)$ and
$l'\in\Le^+(\lambda)$.  By Equation~(\ref{eq:armjrecover}) and
Equation~(\ref{eq:legjrecover}), there are $0\leq i\leq t-1$ and $0\leq
j\leq t-1$, and $a\in \widetilde{\Ar}^+_j(\lambda)$,
$l\in\widetilde{\Le}^+_i(\lambda)$ such that
$$a'=at+j\quad\text{and}\quad l'=lt+t-i-1.$$
Now, Lemma~\ref{lem:hooklength}(i) gives
$$h(b)=a'+l'+1=(a+l+1)t+j-i.$$
Assume $b$ is an Arm-Coleg box of $\lambda$. By (ii) of
Proposition~\ref{prop:froblabelprop}, its Frobenius label is
$(a',l')_{+-}$ for some admissible pair
$(a',l')\in\Ar^+(\lambda)\times\Le^-(\lambda)$. By
Equation~(\ref{eq:armjrecover}) and Equation~(\ref{eq:legjrecover}), there
are $0\leq j\leq t-1$ and $a\in\widetilde{\Ar}_j^+$ such that $a'=at+j$.
Set $i=r_t(l')$ and $l=q_t(l')$. Since $(a',l')$ is admissible,
$(a,j)>(l,i)$ for the lexicographic order. This implies the conditions on
$l$ and $i$ of the statement and Lemma~\ref{lem:hooklength}(ii) gives
$$\mathfrak h(b)=a'-l'=at+j-lt-i=(a-l)t+j-i,$$
as required.  The proof of (iii) is similar.
\end{proof}

By Theorem~\ref{thm:main}, for  all $0\leq j\leq t-1$, we can construct
the set $L_j$ and $A_j$ from the Frobenius symbol of the quotient
partitions and the characteristic vector. Hence, we obtain the sets
$\widetilde\Ar^+_j(\lambda)$ and $\widetilde\Le^+_j(\lambda)$ by
Equation~(\ref{eq:compatibilite1}).

On the other hand, we notice that 
$$\Le^-(\lambda)=\{0,\ldots,\mathfrak a\}\backslash
\bigcup_{j=0}^{t-1}\Ar^+_j(\lambda)\quad\text{and}\quad
\Ar^-(\lambda)=\{0,\ldots,\mathfrak l\}\backslash
\bigcup_{j=0}^{t-1}\Le^+_j(\lambda),$$
where $\mathfrak a=\max(\cup A_j)$ and $\mathfrak l=\max(\cup L_j)$.  In
particular, we have a way to recover the Frobenius label of $\lambda$.
However, it seems, in general, to be difficult to recover $\Le^-(\lambda)$
and $\Ar^-(\lambda)$ from the characteristic vector and the Frobenius
label of $\lambda^j$.

\begin{example}
\label{ex:virtualpb}
Consider $\lambda=(10)$ with pointed $5$-abacus :
\medskip
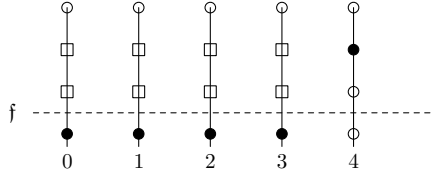
\begin{center}
\definecolor{sqsqsq}{rgb}{0.12549019607843137,0.12549019607843137,0.12549019607843137}
\begin{tikzpicture}[line cap=round,line join=round,>=triangle
45,x=0.7cm,y=0.7cm, scale=0.8,every node/.style={scale=0.8}]

\draw [dash pattern=on 2pt off 2pt](-2.5,1.5)-- (7,1.5);
\draw(-3,1.5)node{$\mathfrak f$};

\draw (-2,0.7) node[anchor=north west] {$0$};

\draw (-1.7,4)-- (-1.7,0.7);

\begin{scriptsize}

\draw (-1.7,2) node[draw] {\relax};
\draw (-1.7,3) node[draw] {\relax};

\draw (-1.7,4) circle (2.5pt);

\draw [fill=black] (-1.7,1) circle (2.5pt);
\end{scriptsize}

\draw (-0.3,0.7) node[anchor=north west] {$1$};

\draw (0.,4)-- (0.,0.7);

\begin{scriptsize}
	
\draw (0,2) node[draw] {\relax};
\draw (0,3) node[draw] {\relax};

\draw (0,4) circle (2.5pt);

\draw [fill=black](0,1) circle (2.5pt);

\end{scriptsize}

\draw (1.4,0.7) node[anchor=north west] {$2$};

\draw (1.7,4)-- (1.7,0.7);

\begin{scriptsize}
	
\draw (1.7,2) node[draw] {\relax};
\draw (1.7,3) node[draw] {\relax};
\draw (1.7,4) circle (2.5pt);
\draw [fill=black](1.7,1) circle (2.5pt);

\end{scriptsize}

\draw (3.1,0.7) node[anchor=north west] {$3$};

\draw (3.4,4)-- (3.4,0.7);

\begin{scriptsize}
	
\draw (3.4,2) node[draw] {\relax};

\draw (3.4,3) node[draw] {\relax};

\draw (3.4,4) circle (2.5pt);
\draw [fill=black] (3.4,1) circle (2.5pt);

\end{scriptsize}

\draw (4.8,0.7) node[anchor=north west] {$4$};

\draw (5.1,4)-- (5.1,0.7);

\begin{scriptsize}

\draw (5.1,2) circle (2.5pt);
\draw [fill=black](5.1,3) circle (2.5pt);
\draw (5.1,4) circle (2.5pt);
\draw (5.1,1) circle (2.5pt);

\end{scriptsize}
\end{tikzpicture}
\end{center}
The ``square'' beads are colegs of $\lambda$, but not colegs of
$\lambda^0$, $\lambda^1$, $\lambda^2$ and $\lambda^3$. Note that, without
the information of the position of the highest black bead of the runner of
$\lambda^4$, and considering only the beads of the first runner $\mathcal
R_0$, we cannot detect the two ``square'' boxes of $\mathcal R_0$. 
\end{example}
In~\cite{JOB}, C. Bessenrodt, J.B. Gramain and J. Olsson give a surprising
result connecting the hooklengths of the $t$-core of $\lambda$ with those
of $\lambda$. The following is \cite[Theorem 4.7]{JOB}.
\begin{theorem}
\label{thm:JOB}
Let $t$ be a positive integer. For any partition $\lambda$ with $t$-core
$\lambda_{(t)}$, the set (with multiplicities) of hooklengths of
$\lambda_{(t)}$ is a subset (with multiplicities) of the one of $\lambda$.
\end{theorem}

\begin{proof}
The hooklengths of $\lambda$ and $\lambda_{(t)}$ can be read off their
pointed $t$-abacus. Let $0\leq i<j\leq t-1$. We will prove that the hooks
of $\lambda_{(t)}$ arising from the runners labeled by $i$ and $j$ can be
associated, in a one-to-one fashion, with a subset of hooks of $\lambda$
arising from the same runners and with the same hooklengths.  Since the
hooklengths do not depend on the choice of a $t$-abacus, we consider the
following one: On runner $i$ and $j$, we put $r_i$ and $r_j$ black beads,
where $r_i$ and $r_j$ are the numbers of beads that are exchanged in the
procedure of Remark~\ref{rk:usualcore}. We obtain here the runners $R_i$
and $R_j$ of a $t$-abacus of $\lambda_{(t)}$. Exchanging the beads  in the
reverse procedure of Remark~\ref{rk:usualcore}, we recover the runners
$R'_i$ and $R'_j$ of $\lambda$. On each runner, we label the slots by
$\N^*$ from the bottom to the top. Let positive integers $\alpha$ and
$\beta$ be labels of a slot on the runner $i$ and $j$ respectively.  We
then define the level between these two slots by
$\ell(\alpha,\beta)=\alpha-\beta$.  Any hook is labeled by a tuple $(b,w)$
of black and white beads such that either $\ell(b,w)>0$ or $\ell(b,w)=0$
and $b$ is to the left of $w$.  Furthermore, if $b$ is a black bead, we
say that  $b$ gives a level $d$ hook if there is a white bead $w$ such
that $(b,w)$ is a hook and $\ell(b,w)=d$.  If $r_j=r_i$, then
$\lambda_{(t)}$ has no hook coming from the runner $R_i$ and $R_j$.

Assume $r_i < r_j$, and write $k=r_i$ and $l=r_i-r_j$.
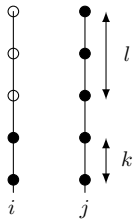
\medskip 

\begin{center}
\begin{tikzpicture}[line cap=round,line join=round,>=triangle
45,x=0.7cm,y=0.7cm, scale=0.8,every node/.style={scale=0.8}]

\draw[latex-latex](0.5,-1.1)--(0.5,0);
\draw(1,-0.5)node{$k$};
\draw[latex-latex](0.5,0.9)--(0.5,3);
\draw(1,2)node{$l$};

\draw (-2,-1.3) node[anchor=north west] {$i$};

\draw (-1.7,3)-- (-1.7,-1.3);

\begin{scriptsize}

\draw (-1.7,2) circle (2.5pt);
\draw (-1.7,3) circle (2.5pt);

\draw (-1.7,1) circle (2.5pt);
\draw (-1.7,0)[fill=black] circle (2.5pt);
\draw (-1.7,-1)[fill=black] circle (2.5pt);
\end{scriptsize}

\draw (-0.3,-1.3) node[anchor=north west] {$j$};

\draw (0.,3)-- (0.,-1.3);

\begin{scriptsize}

\draw  (0,2)[fill=black] circle (2.5pt);
\draw (0,3)[fill=black] circle (2.5pt);

\draw (0,1)[fill=black] circle (2.5pt);
\draw  (0,0)[fill=black] circle (2.5pt);
\draw  (0,-1)[fill=black] circle (2.5pt);

\end{scriptsize}
\end{tikzpicture}\\
\end{center}

Let $0\leq d\leq l-1$. Note that there are $N_d=l-d$ hooks of level $d$ on
the runners $R_i$ and $R_j$. Denote by $N'_d$ the number of hooks of level
$d$ on $R'_i$ and $R'_j$. For $b$ a black bead of $R'_j$, we say that a
black bead $x$ on $R'_i$ kills the level $d$ hook given by $b$ if
$\ell(b,x)=d$. Suppose there are $m$ black beads on $R'_j$ whose level $d$
hooks are killed by some black beads of $R'_i$. We remark that the first
$d$ beads of $R'_j$ cannot give a level $d$ hook. Hence, there are at
least $k+l-m-d$ black beads on $R'_j$ that give a level $d$ hook, that is,
$N'_d\geq k+l-m-d$. However, $k\geq m$, and thus
$$N'_d\geq k+l-m-d\geq l-d=N_d.$$ 
We conclude by noting that two hooks have level $d$ if and only if they
have length $d(j-i+1)+j-i$.

Assume $r_i> r_j$, and write again $k=r_j$ and $l=r_i-r_j$. Then for all
$1\leq d\leq l-1$, there are $N_d=l-d$ hooks of level $d$ on $R_i$ and
$R_j$. The same computation as above shows that the number of $d$ level
hooks coming from a black bead of $R_i$ is bigger that $N_j$.  Since in
this case, two hooks have level $d$ if and only if they have length
$d(j-i-1)+j-i$. The result follows.
\end{proof}

\begin{remark}
The property described in Theorem \ref{thm:JOB} is difficult to see on the
Frobenius symbol. Let $t=2$ and $\lambda=(4,2,2,1,1,1)$. Then
$\lambda_{(2)}=(2,1)$. 
\medskip
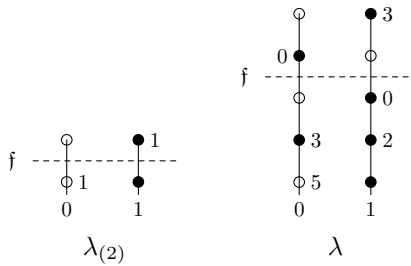
\begin{center}
\begin{tabular}{lcl}
\definecolor{sqsqsq}{rgb}{0.12549019607843137,0.12549019607843137,0.12549019607843137}
\begin{tikzpicture}[line cap=round,line join=round,>=triangle
45,x=0.7cm,y=0.7cm, scale=0.8,every node/.style={scale=0.8}]

\draw [dash pattern=on 2pt off 2pt](-2.5,1.5)-- (1,1.5);
\draw(-3,1.5)node{$\mathfrak f$};

\draw (-2,0.7) node[anchor=north west] {$0$};

\draw (-1.7,2)-- (-1.7,0.7);

\draw(-1.3,1)node{$1$};
\begin{scriptsize}

\draw (-1.7,2) circle (2.5pt);

\draw (-1.7,1) circle (2.5pt);
\end{scriptsize}

\draw (-0.3,0.7) node[anchor=north west] {$1$};

\draw (0.,2)-- (0.,0.7);

\draw(0.4,2)node{$1$};
\begin{scriptsize}
	
\draw [fill=black] (0,2) circle (2.5pt);
\draw [fill=black](0,1) circle (2.5pt);

\end{scriptsize}
\end{tikzpicture}
&
\quad
&
\begin{tikzpicture}[line cap=round,line join=round,>=triangle
45,x=0.7cm,y=0.7cm, scale=0.8,every node/.style={scale=0.8}]

\draw [dash pattern=on 2pt off 2pt](-2.5,1.5)-- (1,1.5);
\draw(-3,1.5)node{$\mathfrak f$};

\draw (-2,-1.3) node[anchor=north west] {$0$};

\draw (-1.7,3)-- (-1.7,-1.3);

 \draw(-2.1,2)node{$0$};
\draw(-1.3,0)node{$3$};
\draw(-1.3,-1)node{$5$};
\begin{scriptsize}

\draw (-1.7,2)[fill=black] circle (2.5pt);
\draw (-1.7,3) circle (2.5pt);

\draw (-1.7,1) circle (2.5pt);
\draw (-1.7,0)[fill=black] circle (2.5pt);
\draw (-1.7,-1) circle (2.5pt);
\end{scriptsize}

\draw (-0.3,-1.3) node[anchor=north west] {$1$};

\draw (0.,3)-- (0.,-1.3);

\draw(0.4,3)node{$3$};
\draw(0.4,1)node{$0$};
\draw(0.4,0)node{$2$};
\begin{scriptsize}

\draw  (0,2) circle (2.5pt);
\draw (0,3)[fill=black] circle (2.5pt);

\draw (0,1)[fill=black] circle (2.5pt);
\draw  (0,0)[fill=black] circle (2.5pt);
\draw  (0,-1)[fill=black] circle (2.5pt);

\end{scriptsize}
\end{tikzpicture}\\
\vspace{3pt}
\hspace{1.2cm}$\lambda_{(2)}$
&&
\hspace{1.25cm}$\lambda$
\end{tabular}
\end{center}
Notice the arm $a$ of $\lambda_{(2)}$ labeled by $1$ moves to the arm $a'$
of $\lambda$ labeled by $3$. The partition $\lambda_{(2)}$ has a $3$-hook
corresponding to $a$ and the leg $l$ of $\lambda_{(2)}$ labeled by $1$.
However, there are no $3$-hook from $a'$. The arm of $\lambda$ at position
$0$ does not come from a coarm of $\lambda_{(2)}$. Even so, $\lambda$ has
a $3$-hook (corresponding to the coarm $0$ and the leg $5$ of $\lambda$),
as expected. One of the difficulties with interpreting
Theorem~\ref{thm:JOB} using the concrete description of the hooklengths by
the Frobenius label, is that the hooks of a $t$-core fall into the
phenomena of Example~\ref{ex:virtualpb}. This example illustrates the
challenge inherent in finding a canonical bijection between the hooks of
$\lambda$ and $\lambda_{(t)}$.
\end{remark}

\newpage
\subsection{Parametrization of the diagonal hooks of self-conjugate
partitions}
\label{subsec:diagsym}

Let $\lambda$ be a partition with Frobenius symbol
$\Fr(\lambda)=(l_{s-1},\ldots,l_0\mid a_0,\ldots,a_{s-1})$. Write
$\mathcal D(\lambda)$ for the set of diagonal hooklengths of $\lambda$.
By definition of the Frobenius symbol, we have 
$$\mathcal D(\lambda)=\{l_j+a_j+1\mid 0\leq j\leq s-1\}.$$ 
Strictly speaking, these numbers are not completely determined from
$\Le^+(\lambda)$ and $\Ar^+(\lambda)$ in the sense that these two sets
have to be ordered to recover $\mathcal D(\lambda)$.  However, when
$\lambda$ is a self-conjugate partition, we have
$\Ar^+(\lambda)=\Le^+(\lambda)$, and
\begin{equation}
\label{eq:diaghooksym}
\varphi:\Ar^+(\lambda)\longrightarrow \mathcal D(\lambda),\,x\in\Ar^+(\lambda)\longmapsto 2x+1. 
\end{equation}
Now, we will see that can derive the set of diagonal hooklengths of
self-conjugate partitions from their quotient and core.

\begin{corollary}
\label{cor:diagsym}
Let $t$ be a positive integer. Let $\lambda$ be a self-conjugate partition
with characteristic vector $c(\lambda)=(c_0,\ldots,c_{t-1})$ and
$t$-quotient $(\lambda^0,\ldots,\lambda^{t-1})$. Then
$$\mathcal D(\lambda)=\bigcup_{j=0}^{t-1}\{2(xt+j)+1\mid x\in
A_j\},$$
where $A_j$ is described in Theorem~\ref{thm:main}.
\end{corollary}

\begin{proof}
By Theorem~\ref{thm:main}, $\bigcup_j A_j$ is the set of arms of
$\lambda$.  We then conclude by Equation~(\ref{eq:diaghooksym}). 
\end{proof}

\begin{example} We consider $t=5$ and the partition $\lambda$ with 
$$c_5(\lambda)=(0,2,0,-2,0)\quad \text{and}\quad
\lambda^{(5)}=((1^2),(1),(1),(1),(2)).$$
Then $\lambda$ is self-conjugate by Corollary~\ref{cor:descconjugate}; see
also Example~\ref{ex:sympart}. 
Then 
$$\Fr(\lambda_0)=(1\mid 0),\
 \Fr(\lambda_1)=(0\mid 0),\ 
\Fr(\lambda_2)=(0\mid 0),\ 
\Fr(\lambda_3)=(0\mid 0),\ 
\Fr(\lambda_4)=(0\mid 1),$$
and
$$A_0=\{0\},\quad A_{1}=\{0,2\},\quad A_{2}=\{0\},\quad
A_3=\emptyset,\quad A_4=\{1\}.$$
We recover that $\lambda$ has five diagonal hooks and
Corollary~\ref{cor:diagsym} gives
\begin{eqnarray*}
\mathcal D(\lambda)&=&\{2(5\cdot 0+0)+1,\,2(5\cdot 0+1)+1,\,2(5\cdot
2+1)+1,\,2(5\cdot
0+2)+1,\\
&&\quad 2(5\cdot 1+4)+1\}\\
&=&
\{1,\,3,\,23,\,5,\,19\}.
\end{eqnarray*}

\end{example}

\subsection{A well-known correspondence revisited}
\label{subsec:bij}

We now discuss a special case where our approach allows us to compare,
canonically, some hook information of the partition, with that of its
characteristic vector and $t$-quotient.  Recall that there is a well-known
correspondence~\cite[Proposition 3.1]{MorrisOlsson} between the set of the
$kt$-hooks of a partition, and the set of $k$-hooks of its $t$ quotient.
We now make precise this correspondence by constructing an explicit and
canonical bijection between boxes of $\lambda^{(t)}$ and the ones of
$\lambda$ whose hooklength is divisible by $t$.

\begin{proposition}
\label{prop:olssonmap}
Let $\lambda$ be a partition and a positive integer $t$. We denote by
$\lambda^{(t)}=(\lambda^0,\ldots,\lambda^{t-1})$ and
$c(\lambda)=(c_0,\ldots,c_{t-1})$ its $t$-quotient and its characteristic
vector. Let $0\leq j\leq t-1$ and $b$ be a box of $\lambda^j$.
\begin{enumerate}[(i)]
\item Suppose $b$ is an Durfee box, \emph{ie} $\fl(b)=(a,l)_{++}$ for some
$a\in\Ar^{+}(\lambda^j)$ and $l\in\Le^+(\lambda^j)$.
\begin{enumerate}[(a)]
\item Assume $c_j\geq 0$. If $l\leq c_j-1$, then we define $f(b)$ to be
the Arm-Coleg box of $\lambda$ with Frobenius label
\begin{equation}
\label{eq:olss1}
((a+c_j)t+j,(c_j-l-1)t+j)_{+-}.
\end{equation}
If $l\geq c_j$, then we define $f(b)$ to be the Durfee box of $\lambda$
with Frobenius label
\begin{equation}
\label{eq:olss2}
((a+c_j)t+j,(l-c_j)t+t-j-1)_{++}.
\end{equation}
\item Assume $c_j\leq 0$. If $a\leq c_j-1$, then we define $f(b)$ to be
the Leg-Coarm box of $\lambda$ with Frobenius label
\begin{equation}
\label{eq:olss3}
((|c_j|-a-1)t+t-j-1,(|c_j|+l)t+t-j-1)_{-+}.
\end{equation}
If $a\geq c_j$, then we define $f(b)$ to be the Durfee box of $\lambda$
with Frobenius label
\begin{equation}
\label{eq:olss4}
((a-|c_j|)t+j,(b+|c_j|)t+t-j-1)_{++}.
\end{equation}
\end{enumerate}

\item Suppose $b$ is an Leg-Coarm box with Frobenius label
$\fl(b)=(a,l)_{-+}$.
\begin{enumerate}[(a)]
\item Assume $c_j\geq 0$. 
If $a<l\leq c_j-1$, then we define $f(b)$ to be the Arm-Coleg box of
$\lambda$ with Frobenius label
\begin{equation}
\label{eq:olss5}
((c_j-a-1)t+j,(c_j-a-1)t+j)_{+-}.
\end{equation}
If $a\geq c_j-1<l$, then we define $f(b)$ to be the Durfee box of
$\lambda$ with Frobenius label
\begin{equation}
\label{eq:olss6}
((c_j-a-1)t+j,(b-c_j)t+t-j-1)_{++}.
\end{equation}
If $c_j\leq a$, then we define $f(b)$ to be the Leg-Coarm box of $\lambda$
with Frobenius label
\begin{equation}
\label{eq:olss7}
((a-c_j)t+t-j-1,(b-c_j)t+t-j-1)_{-+}.
\end{equation}
\item Assume $c_j\leq 0$. We define $f(b)$ to be the Leg-Coarm box of
$\lambda$ with Frobenius label
\begin{equation}
\label{eq:olss8}
((|c_j|+a)t+t-j-1,(|c_j|+l)t+t-j-1)_{-+}.
\end{equation}
\end{enumerate}
\item Suppose $b$ is an Arm-Coleg box  with Frobenius label
$\fl(b)=(a,l)_{+-}$.
\begin{enumerate}[(a)]
\item Assume $c_j\geq 0$. We define $f(b)$ to be the Arm-Coleg box of
$\lambda$ with Frobenius label
\begin{equation}
\label{eq:olss9}
((a+c_j)t+j,(c_j+l)t+j)_{+-}.
\end{equation}
\item Assume $c_j\leq 0$. If $l<a\leq c_j-1$, then we define $f(b)$ to be
the Leg-Coarm box of $\lambda$ with Frobenius label
\begin{equation}
\label{eq:olss10}
((|c_j|-a-1)t+t-j-1,(|c_j|-l-1)t+t-j-1)_{-+}.
\end{equation}
If $l\leq c_j-1<a$, then we define $f(b)$ to be the Durfee box of
$\lambda$ with Frobenius label
\begin{equation}
\label{eq:olss11}
((a-|c_j|)t+j,(|c_j|-l-1)t+t-j-1)_{++}.
\end{equation}
If $c_j\leq l$, then we define $f(b)$ to be the Arm-Coleg box of $\lambda$
with Frobenius label
\begin{equation}
\label{eq:olss12}
((a-|c_j|)t+j,(l-|c_j|)t+j)_{+-}.
\end{equation}
\end{enumerate}
\end{enumerate}
Then the map $f$ from the set of boxes of $\lambda^0$, \ldots,
$\lambda^{t-1}$ into the set of boxes of $\lambda$ induces a bijection
between the set of $k$-hooks of $\lambda^{(t)}$ and the set of $kt$-hooks
of $\lambda$.
\end{proposition}

\begin{proof}
Let $\widetilde b$ be a box of $\lambda$ with Frobenius label $(\widetilde
a,\widetilde l)_{\pm,\pm}$ and such that $\mathfrak h(\widetilde b)$ is a
$\widetilde k$-hook for some integer $\widetilde k$. Denote by $\mathcal
R$ the pointed abacus of $\lambda$, and $\widetilde u$ and $\widetilde v$
the black and white beads of $\mathcal R$ associated with $\widetilde b$
as in Remark~\ref{rk:abacusinfo}(iii). Write $\mathcal T=(\mathcal
R_0,\ldots,\mathcal R_{t-1})$ for the pointed $t$-abacus of $\lambda$ and
set $u=\varphi_t(\widetilde u)$ and $v=\varphi_t(\widetilde v)$.  Suppose
$u\in\mathcal R_i$ and $v\in\mathcal R_j$ for some $0\leq i\leq t-1$ and
$0\leq j\leq t-1$. By Lemma~\ref{lem:hooktabacus}, we have  $\mathfrak
h(\widetilde b)\equiv 0\mod t$ if and only if $i\equiv j\mod t$, if and
only if $i=j$ if and only if $u$ and $v$ lie on the same runner.  More
precisely, suppose that the $u$ and $v$ are on the runner $j$. On
$\mathcal T$, there are $t-j-1$ slots to the right of $v$ and $j$ slots to
the left of $u$. On the other hand, since $\mathcal R_j$ is the image of
the pointed abacus of $\lambda^j$ by $p^{c_j}$, and that the number of
beads between $u$ and $v$ on this runner is invariant under $p^{c_j}$, it
corresponds to a $k$-hook of $\lambda^j$. This means there are strictly
$k-1$ beads between $u$ and $v$ on the runner $j$ by
Remark~\ref{rk:abacusinfo}(iv). Hence, Remark~\ref{rk:abacusinfo}(iv)
gives
$$\mathfrak h(\widetilde b)=(t-j-1)+j+t(k-1)+1=kt.$$
The map $f$ between the hooks of $\lambda^{(t)}$ and the hooks of
$\lambda$ dividing by $t$ is then constructed following the reverse
procedure. First, consider a box $b$ of $\lambda^j$, that corresponds to a
black bead $u$ and a white bead $v$ of the pointed abacus of $\lambda^j$.
Then, apply $p^{c_j}$ to this abacus and derive $(\widetilde u,\widetilde
v)$ on $\mathcal R$ using the bijection $\varphi_t$, which gives a box
$f(b)$ of $\lambda$. The preceding discussion says that $f$ is
well-defined, and sends bijectively the $k$-hooks of $\lambda^{(t)}$ into
the $kt$-hooks of $\lambda$.

Following the procedure, we can deduce the statement. We only prove (i),
(a). The other cases are similar.  Assume that $c_j\geq 0$ and $l\leq
c_j-1$. Suppose $u$ and $v$ correspond to $b$ as above, and are in
position $a$ and $l$ on the pointed abacus of $\lambda^j$.  After applying
$p^{c_j}$ to the runner, the new position of $u$ is $a+c_j$ and the one of
$v$ is $c_j-l-1$. Hence, the $f(b)$ is a Arm-Coleg box of $\lambda$ with
Frobenius label $((a+c_j)t+j, (c_j-l-1)t+j)_{+-}$.  If $l\geq c_j$, then
after the $c_j$-push, $v$ is still a leg at position $l-c_j$.  Thus,
$f(b)$ is Durfee box with Frobenius label
$(a+c_j)t+j,(l-c_j)t+t-j-1)_{++}$, as required.  
\end{proof}

\begin{remark}
\noindent
\begin{enumerate}[(i)]
\item In the proof of Proposition~\ref{prop:olssonmap}, we show that the
$k$-hook corresponding to a box $b$ of $\lambda^j$ is sent by $f$ to a
$kt$-hook of $\lambda$ corresponding to $f(b)$.  We can also check this
directly using the Frobenius labels of boxes and
Lemma~\ref{lem:hooklength}. For example, if $c_j\geq 0$ and $b$ is a
Durfee box of $\lambda^j$ with Frobenius label $(a,l)_{++}$ corresponding
to a $k$-hook, then $k=a+l+1$ and the hook of $f(b)$ has length 
$$\mathfrak h(f(b))= (a+c_j)t+j-(c_j-l-1)t-j=(a+l+1)t=kt$$
when $l\leq c_j-1$ and
$$\mathfrak h(f(b))= ((a+c_j)t+j)+((l-c_j)t+t-j-1)+1=(a+l+1)t=kt$$
when $l\geq c_j$. The other cases are similar.
\item The map $f$ of Proposition~\ref{prop:olssonmap} gives a
parametrization of the boxes of $\lambda$ with hooklength dividing by $t$
by the boxes of $\lambda^{(t)}$.
\end{enumerate}
\end{remark}

\begin{example}
Consider $\lambda=(7,7,3,2,1^7)$ with $t=3$. Then $c(\lambda)=(1,-1,0)$
and $\lambda^{(3)}=((2,1),(1,1,1),(2))$. We have $\Fr(\lambda)=(10,2,0\mid
0,5,6)$, $\Fr(2,1)=(1\mid 1)$, $\Fr(1,1,1)=(2\mid 0)$ and $\Fr(2)=(0\mid
1)$. Now, we label the boxes of $\lambda^0$, $\lambda^1$ and $\lambda^2$
as follows.
\bigskip

\begin{center} 
\begin{tabular}{lcr}
\frob[]
    {{\small $a$,\small $b$},{\small $c$}}
    {}
    {}
    {}
    {}
    {}
&
\hspace{1.5cm}
\frob[]
    {{\small $d$},{\small $e$},{\small $g$}}
    {}
    {}
    {}
    {}
    {}
&
\hspace{1.5cm}
\frob[]
    {{\small $h$,\small $i$}}
    {}
    {}
    {}
    {}
    {}
\\
\\
\small$(2,1)$
&
\hspace{1.5cm}
\small$(1,1,1)$
&
\small$(2)$\hspace{5pt}
\end{tabular}
\end{center} 
We have
$$\fl(a)=(1,1)_{++},\ \fl(b)=(1,0)_{+-},\ \fl(c)=(0,1)_{-+},\
\fl(d)=(0,2)_{++},$$
and
$$\fl(e)=(0,2)_{-+},\ \fl(g)=(1,2)_{-+},\ \fl(h)=(1,0)_{++},\
\fl(i)=(1,0)_{+-}.$$
By Proposition~\ref{prop:olssonmap}, we deduce that
$$\fl(f(a))=(6,2)_{++},\ \fl(f(b))=(6,3)_{+-},\ \fl(f(c))=(0,2)_{++},\
\fl(d)=(1,10)_{-+},$$
and
$$\fl(f(e))=(4,10)_{-+},\ \fl(f(g))=(7,10)_{-+},\ \fl(f(h))=(5,0)_{++},\
\fl(i)=(5,2)_{+-}.$$
\normalsize
Thus,
\bigskip

\begin{center}
\vspace{-2.5cm}
\begin{tabular}{lcr}
\begin{tikzpicture}[scale=0.4,draw/.append style={black},baseline=\shadedBaseline]
      \draw(0,0)+(-.5,-.5)rectangle++(.5,.5);
      \draw(1,0)+(-.5,-.5)rectangle++(.5,.5);
      \draw(2,0)+(-.5,-.5)rectangle++(.5,.5);
      \draw(3,0)+(-.5,-.5)rectangle++(.5,.5);
      \draw(4,0)+(-.5,-.5)rectangle++(.5,.5);
      \draw(5,0)+(-.5,-.5)rectangle++(.5,.5);
      \draw(6,0)+(-.5,-.5)rectangle++(.5,.5);
      \draw(0,-1)+(-.5,-.5)rectangle++(.5,.5);
      \draw(1,-1)+(-.5,-.5)rectangle++(.5,.5);
      \draw(2,-1)+(-.5,-.5)rectangle++(.5,.5);
      \draw(3,-1)+(-.5,-.5)rectangle++(.5,.5);
      \draw(4,-1)+(-.5,-.5)rectangle++(.5,.5);
      \draw(5,-1)+(-.5,-.5)rectangle++(.5,.5);
      \draw(6,-1)+(-.5,-.5)rectangle++(.5,.5);
      \draw(0,-2)+(-.5,-.5)rectangle++(.5,.5);
      \draw(1,-2)+(-.5,-.5)rectangle++(.5,.5);
      \draw(2,-2)+(-.5,-.5)rectangle++(.5,.5);
      \draw(0,-3)+(-.5,-.5)rectangle++(.5,.5);
      \draw(1,-3)+(-.5,-.5)rectangle++(.5,.5);
      \draw(0,-4)+(-.5,-.5)rectangle++(.5,.5);
      \draw(0,-5)+(-.5,-.5)rectangle++(.5,.5);
      \draw(0,-6)+(-.5,-.5)rectangle++(.5,.5);
      \draw(0,-7)+(-.5,-.5)rectangle++(.5,.5);
      \draw(0,-8)+(-.5,-.5)rectangle++(.5,.5);
      \draw(0,-9)+(-.5,-.5)rectangle++(.5,.5);
      \draw(0,-10)+(-.5,-.5)rectangle++(.5,.5);

      \draw(1,0)node{\footnotesize$a'$};
      \draw(5,0)node{\footnotesize$b'$};
      \draw(1,-2)node{\footnotesize$c'$};
      \draw(0,-3)node{\footnotesize$d'$};
      \draw(0,-5)node{\footnotesize$e'$};
      \draw(0,-8)node{\footnotesize$g'$};
      \draw(2,-1)node{\footnotesize$h'$};
      \draw(4,-1)node{\footnotesize$i'$};
\end{tikzpicture}
&
\hspace{0.8cm}
$\small\begin{array}{l}
\relax\\
\relax\\
\relax\\
\relax\\
\relax\\
\relax\\

a'=f(a),\\
b'=f(b),\\
c'=f(c),\\
d'=f(d),\\
e'=f(e),\\
g'=f(f),\\
h'=f(g),\\
i'=f(h),\\
\end{array}$
\normalsize
&
\hspace{0.5cm}
$\small\begin{array}{l}
\relax\\
\relax\\
\relax\\
\relax\\
\relax\\
\relax\\

\mathfrak h(a')=6+2+1=9,\\
\mathfrak h(b')=6-3=3,\\
\mathfrak h(c')=0+2+1=3,\\
\mathfrak h(d')=10-1=9,\\
\mathfrak h(e')=10-4=6,\\
\mathfrak h(g')=10-7=3,\\
\mathfrak h(h')=5+0+1=6,\\
\mathfrak h(i')=5-2=3.\\
\end{array}$
\normalsize
\end{tabular}
\end{center}
\end{example}
\bigskip

\bigskip
{\bf Acknowledgements.} The first author acknowledges the support of the
ANR grant GeRepMod ANR-16-CE40-0010-01. The second author acknowledges the
support of PSC-CUNY TRADA-50-645. The authors would also like to thank the
IMJ-PRG at the University of Paris Diderot and the Department of
Mathematics and Computer Science at York College, City University of New
York for the financial and logistical support which allowed the completion
of this project. The second author would like to thank Matt Fayers for a
conversation which lead to investigations on level $h$ actions. Both authors
sincerely wish to thank Pascal Alessandri for helpful discussions,
especially for the resolution of system~(\ref{eq:systemeancester}).

\bibliographystyle{abbrv}
\bibliography{references_19_1}

\begin{thebibliography}{10}

\bibitem{BaruahBerndt}
N.~D. Baruah and B.~C. Berndt.
\newblock Partition identities and {R}amanujan's modular equations.
\newblock {\em J. Combin. Theory Ser. A}, 114(6):1024--1045, 2007.

\bibitem{BaruahNath2}
N.~D. Baruah and K.~Nath.
\newblock Infinite families of arithmetic identities for self-conjugate 5-cores
  and 7-cores.
\newblock {\em Discrete Math.}, 321:57--67, 2014.

\bibitem{BaruahNath}
N.~D. Baruah and K.~Nath.
\newblock Some results on 3-cores.
\newblock {\em Proc. Amer. Math. Soc.}, 142(2):441--448, 2014.

\bibitem{BaruahSarmah}
N.~D. Baruah and B.~K. Sarmah.
\newblock Identities for self-conjugate 7- and 9-core partitions.
\newblock {\em Int. J. Number Theory}, 8(3):653--667, 2012.

\bibitem{JOB}
C.~Bessenrodt, J.-B. Gramain, and J.~B. Olsson.
\newblock Generalized hook lengths in symbols and partitions.
\newblock {\em J. Algebraic Combin.}, 36(2):309--332, 2012.

\bibitem{BrNa}
O.~Brunat and R.~Nath.
\newblock The {N}avarro conjecture for the alternating groups.
\newblock {\em preprint}, 2018.

\bibitem{Farahat}
H.~Farahat.
\newblock On {$p$}-quotients and star diagrams of the symmetric group.
\newblock {\em Proc. Cambridge Philos. Soc}, 48:737--740, 1952.

\bibitem{Fayers}
M.~Fayers.
\newblock A generalisation of core partitions.
\newblock {\em J. Combin. Theory Ser. A}, 127:58--84, 2014.

\bibitem{frobenius}
G.~Frobenius.
\newblock {\"U}ber die {C}haraktere der symmetrischen {G}ruppe.
\newblock {\em Sitzungsberichte der {A}kademie der {W}iss. zu {B}erlin}, pages
  516--534, 1900.

\bibitem{GarvanKimStanton}
F.~Garvan, D.~Kim, and D.~Stanton.
\newblock Cranks and {$t$}-cores.
\newblock {\em Invent. Math.}, 101(1):1--17, 1990.

\bibitem{Garvan}
F.~G. Garvan.
\newblock Some congruences for partitions that are {$p$}-cores.
\newblock {\em Proc. London Math. Soc. (3)}, 66(3):449--478, 1993.

\bibitem{HirschhornSellers}
M.~D. Hirschhorn and J.~A. Sellers.
\newblock Elementary proofs of various facts about 3-cores.
\newblock {\em Bull. Aust. Math. Soc.}, 79(3):507--512, 2009.

\bibitem{James-Kerber}
G.~James and A.~Kerber.
\newblock {\em The {R}epresentation {T}heory of the {S}ymmetric {G}roup},
  volume~16 of {\em Encyclopedia of Mathematics and its Applications}.
\newblock Addison-Wesley Publishing Co., Reading, Mass., 1981.

\bibitem{james}
G.~D. James.
\newblock Some combinatorial results involving {Y}oung diagrams.
\newblock {\em Math. Proc. Cambridge Philos. Soc.}, 83(1):1--10, 1978.

\bibitem{Lascoux}
A.~Lascoux.
\newblock Ordering the affine symmetric group.
\newblock In {\em Algebraic combinatorics and applications ({G}\"{o}\ss
  weinstein, 1999)}, pages 219--231. Springer, Berlin, 2001.

\bibitem{Littlewood}
D.~E. Littlewood.
\newblock Modular representations of symmetric groups.
\newblock {\em Proc. Roy. Soc. London Ser. A}, 209:333--353, 1951.

\bibitem{MorrisOlsson}
A.~O. Morris and J.~B. Olsson.
\newblock On {$p$}-quotients for spin characters.
\newblock {\em J. Algebra}, 119(1):51--82, 1988.

\bibitem{Nakayama}
T.~Nakayama.
\newblock On some modular properties of irreducible representations of a
  symmetric group. i.
\newblock {\em Jap. J. Math}, 18:89--108, 1941.

\bibitem{OlssonFrob}
J.~r.~B. Olsson.
\newblock Frobenius symbols for partitions and degrees of spin characters.
\newblock {\em Math. Scand.}, 61(2):223--247, 1987.

\bibitem{Robinson}
G.~d.~B. Robinson.
\newblock On a conjecture by nakayama.
\newblock {\em Trans. Roy. Soc. Canada Sec. III}, 41:20--25, 1947.

\bibitem{scopes}
J.~Scopes.
\newblock Cartan matrices and {M}orita equivalence for blocks of the symmetric
  groups.
\newblock {\em J. Algebra}, 142(2):441--455, 1991.

\bibitem{Staal}
R.~A. Staal.
\newblock Star diagrams and the symmetric group.
\newblock {\em Can. J. Math.}, 2:79--92, 1950.

\end{thebibliography}
\end{document}